\documentclass[12pt]{amsart}
\usepackage[all]{xy}
\usepackage{amssymb}
\usepackage{amsmath}
\usepackage{MnSymbol}
\usepackage{setspace}
\usepackage{graphicx}
\usepackage{graphics,graphicx}
\usepackage{pgf,color}
\usepackage{xcolor}
\usepackage{epstopdf}
\usepackage{epsfig,textpos,mathrsfs}
\usepackage{amsthm}
\newtheorem{definition}{Definition}[section]
\usepackage[utf8]{inputenc}
\usepackage[english]{babel}
\usepackage{amsthm}
\newtheorem{remark}{Remark}[section]
\newtheorem{theorem}{Theorem}[section]
\newtheorem{lemma}{Lemma}[section]
\newtheorem{proposition}{Proposition}[section]
\newtheorem{corollary}{Corollary}[section]
\newtheorem{example}{Example}[section]

\theoremstyle{remark}
\newcommand{\R}{\mathbb{R}}
\newcommand{\N}{\mathbb{N}}

\newcommand{\Z}{\mathbb{Z}}

\begin{document}

\title[Topological Dynamics of Enveloping Semigroups]{Topological Dynamics of  Enveloping Semigroups}
\vspace{1cm}
\author{ Anima Nagar and Manpreet Singh}
\address{Department of Mathematics, Indian Institute of Technology Delhi, Hauz Khas, New Delhi 110016, INDIA}
\email{anima@maths.iitd.ac.in and maz158145@maths.iitd.ac.in}

    \vspace{.2cm}
%\date{August 22, 2019}

\renewcommand{\thefootnote}{}



\footnote{The second named author thanks CSIR for financial support.}

\renewcommand{\thefootnote}{\arabic{footnote}}
\setcounter{footnote}{0}

\thanks{}

\begin{abstract}
A compact metric space $X$ and a discrete topological acting group $T$ give a flow $(X,T)$. Robert Ellis had initiated the study of  dynamical properties of the flow $(X,T)$ via the algebraic properties of its ``\emph{Enveloping Semigroup}" $E(X)$. This concept of  \emph{Enveloping Semigroups} that he defined, has turned out to be a very fundamental tool in the abstract theory of ``\emph{topological dynamics}".

The flow $(X,T)$ induces the flow $(2^X,T)$. Such a study was first initiated by Eli Glasner who  studied the properties of this induced flow by defining and using the notion of a ``\emph{circle operator}" as an action of $\beta T$  on $2^X$, where $\beta T$ is the \emph{Stone-$\check{C}$ech compactification} of $T$ and also a universal enveloping semigroup.  We propose that the study of properties for the induced flow $(2^X,T)$ be made using the algebraic properties of $E(2^X)$ on the lines of Ellis' \ theory, instead of looking into the action of $\beta T$ on $2^X$ via the circle operator as done by Glasner. Such a study requires extending the present theory on the flow $(E(X),T)$. In this article, we take up such a study giving some subtle relations between the semigroups $E(X)$ and $E(2^X)$ and some interesting associated consequences.

\end{abstract}
\maketitle

\newpage

\tableofcontents

\section{ Introduction}

 A \emph{Dynamical System} is usually denoted as a  pair $(X,T)$, where $X$ is  called a phase space, and defined as an action of some \emph{ group(semigroup)} $T$ on $X$.  The set $Tx = \{tx: t \in T\}$, for a point $x \in X$,  is called the \emph{orbit} of $x$ and \emph{Topological Dynamics} is the study   of orbits for all points in $X$.

 \vskip .5cm

  Suppose $T$ be a topological group(semigroup) and $X$ be a compact Hausdorff space or a compact metric space then $\pi(X,T)$ is  called a \emph{flow (semiflow)}, where $\pi$ is the continuous group(semigroup) action of $T$ on $X$. Thus for every $t \in T$, one can consider the function $\pi^t: X \to X$, given by $\pi^t(x) = (x,t)$. If $T = \mathbb{Z}$ or $\N$ then the flow is called a \emph{cascade or semi cascade}.  So, we can write $\pi(X,\mathbb{Z})$ or $(X,\mathbb{N})$ as $(X,f)$   and the action of $\mathbb{Z}$ or $\N$ will be defined as $\pi(x,m) = m \cdot x:= f^{m}(x)$, where $f= \pi^1$ gives the generator of the action. We drop naming the action $\pi$ of $T$ on $X$, and abbreviate  the flow  as $(X,T)$.

 \vskip .5cm

 Robert Ellis \cite{EL} had defined and studied the algebraic properties of a flow $(X,T)$ via the ``\emph{Enveloping Semigroup}". This algebraic theory has proved to be a very fundamental tool in the abstract theory of topological dynamics.  The enveloping semigroup  $E(X,T) = E(X)$ of a flow $(X, T)$ for a topological group $T$ acting on a compact space $X$
is defined as the closure of the set $\{\pi^t : X \to X : t \in T\}$ considered as a subset of $X^X$ given the product
topology. Equivalently, for the cascade $(X,f)$, the enveloping semigroup $E(X,f) = E(X)$ is the closure of the family $\{f^n: n \in \Z\}$ in the product topology on $X^X$. The enveloping semigroups are  compact, usually non-metrizable, right topological semigroup subset of $X^X$.

We can consider a ``universal" enveloping semigroup $\beta T$ which is the \emph{Stone-$\check{C}$ech compactification} of the group $T$. The group operation of $T$  induces a semigroup structure on $\beta T$.

Identify $t\in T$ with the map $t\longrightarrow tx$. So without loss of generality $T$ can be considered as a subset of $X^X$ and the enveloping semigroup $E(X)$ is its closure in $X^X$. In fact as shown by Ellis,  there exists a continuous map $\Psi: \beta T \rightarrow X^X$ which is an extension of $\psi: T \longrightarrow X^X$ such that $\Psi(\beta T) = E(X)$.

\vskip .5cm

The main aim of this article is to look into the dynamical properties of the Induced Systems via their Enveloping Semigroups. This requires studying Enveloping Semigroups as phase spaces under the action of the acting group(semigroup). We undertake such a study and wherever possible  exhibit  connections between their properties to some dynamical behavior of the system. We are essentially interested in the dynamical system $(E(X),T)$, which has mostly been studied for minimal systems, by extending the known results to non-minimal cases that have not been studied in a systematic manner so far. Throughout we discuss  basic definitions and results,  developing our notations, and include some observations that have not appeared in print so far.

 \vskip .5cm

 For the basics and most of the related theory of \emph{Topological Dynamics} and \emph{Enveloping Semigroups} we recommend the books \cite{AK, AUS,  ELL, F, SGP, JOIN, VER1, VER2, WAL}.

\vskip .5cm

 In \cite{F} Furstenberg had introduced the notion of disjointness of flows - the  flows $(X,T)$ and $(Y,T)$ are \emph{disjoint}, and denoted as $X \perp Y$, if whenever we have homomorphisms $\alpha: Z\rightarrow X, \ \beta: Z \rightarrow Y$ from any flow $(Z,T)$ then there exists a homomorphism $\gamma: Z \rightarrow X \times Y$ such that $\alpha= \pi_{X} \gamma$ and $\beta= \pi_{Y} \gamma$ where $\pi_X, \pi_Y$ are the respective projections. In \cite{SG} Glasner consider the induced flow $(2^X,T)$ on the space $2^X$ of non empty closed subsets of $X$ and gave a necessary and sufficient condition for two minimal flows to be disjoint. For that he introduced the notion of quasifactors. Briefly, quasifactors are the minimal subsets of $2^X$.

\vskip .5cm

For the induced flow $(2^X,T)$,  the notion of a ``\emph{circle operator}" as an action of $\beta T$  on $2^X$ was defined by Ellis, Glasner and Shapiro \cite{EGS1}.

Since $\beta T$ acts on both $X$ and $2^X$, the action of $p\in \beta T$ on $X$ is given as $x\longrightarrow px$ where $px=\lim t_{i}x$ whereas the action of $p\in \beta T$ on $2^X$ is given as $A\longrightarrow p\circ A$ where $p\circ A = \lbrace x\in X: t_{i}\longrightarrow p, \lbrace a_i \rbrace \subset A, t_{i}a_{i}\longrightarrow x\rbrace= \lim t_{i}A $ in $2^X$ where $t_{i}\longrightarrow p$ in $\beta T$.

We note that for $T=\Z$, $\beta \Z$ is the set of all ultrafilters, and so computing an element of $\beta\Z$ is not simple though finding $E(X)$ or $E(2^X)$ can be relatively easy.

Hence for the study of induced systems it will be easier to look into $E(2^X)$ directly without using the circle operator.

\vskip .5cm

Thus, in order to study the properties of induced flows we need to study the properties of $E(2^X)$. But we note that $2^X$ can never be minimal.  Most of the results known for $E(X)$ are when $(X,T)$ is minimal. So we first extend the theory of enveloping semigroups for non-minimal systems, i.e.  we  first study the properties of $E(X)$ when $(X,T)$ is not minimal. These are then  used to study $E(2^X)$.

\vskip .5cm

In this article we  study  the dynamics of $E(X)$ and $E(2^X)$ under the  action of $T$. Our $X$ is mostly a compact, infinite metric space and we are mainly interested in the case when $T$ and $E(X)$ are also infinite. For a finite group $T$, many of our results may not have the same form and we simply ignore those cases. We note that both $E(X)$ and $E(2^X)$ are factors of $\beta T$, though in general  $E(X)$ may not be a subset or superset of $E(2^X)$.

\xymatrixrowsep{12mm}
\xymatrixcolsep{15mm}
\begin{align*}
\xymatrix{
 \beta T \ar[r]^{\Psi_1} \ar[d]_{\Psi_2}
 &  E(X) \ar@{-->}[ld]_{f_2} \\
  E(2^X) \ar@<-5pt>[ru]_{f_1}
}
\end{align*}

One immediate problem that springs up is to determine all possible relations between $E(X)$ and $E(2^X)$ that exist and the dynamical consequences that determine the possibility of such  relations  that arise from such a possibility.

 \vskip 0.5cm

Here we study the topological dynamical properties of $(E(X),T)$, where $T$ is either a group or a monoid. This study can be extended to the study of the dynamical properties of $(E(X),E(X))$. We have not advanced in this direction, since a lot of what we could have obtained trivially would have coincided with the study made by Akin, Auslander and Glasner \cite{AAG}. We suggest the enthusiastic reader to look into \cite{AAG} after reading this work.

\vskip 0.5cm

\section{Basic Definitions and Elementary Properties}

\subsection{Some Topological Dynamics}

Let $(X,T)$ be a flow or a semiflow. A subset $A\subseteq X$ is \emph{invariant} if $TA=\{ta\ |\ a\in A,t\in T\}\subseteq A$. If $A$ is \emph{invariant} then $\pi^{\prime}=\pi|_{T\times A}: T\times A \rightarrow A$ is an action and $\pi^{\prime}(A, T)$ is a (semi)flow. When $T= \Z$ or $\N$, a point $x$ in a the cascade or semicascade $(X,f)$ is called the \emph{periodic point} if $f^n(x)=x$ for some $n$ and the least such positive $n$ is called the \emph{period} of $x$. Period of $x$ is denoted as $\emph{Per(x)}$. Also  the orbit of $x$ in this case is $\mathcal{O}(x)=\lbrace f^n(x): n\in \N, \Z \rbrace$.

 Let $(X,T)$ and $(Y,T)$ be two flows and $\phi:X \rightarrow Y$, then $\phi$ is said to be \emph{homomorphism of flows} or a \emph{factor} if $\phi$ is continuous surjection and $\phi(tx)= t\phi(x) \ \forall\ x\in X, t\in T$. A cascade $(Y,g)$ is called a \emph{factor} of the cascade $(X,f)$ if there exists a continuous surjection $h:X \to Y$ such that $h\circ f=g\circ h,$ and $(X,f)$ and $(Y,g)$ are called conjugate if $h$ is homeomorphism.

\vskip .5cm

$(X,T)$ is \emph{point transitive} if there exists a point $x_0\in X$ with $\overline{Tx_0}= X$, where $\overline{A}$ denotes the closure of $A$. A subset $M\subseteq X$ is said to be \emph{minimal} if $M \neq \emptyset $, $M$ is closed, $M$ is invariant and $ M$ is minimal with respect to these properties and the flow $(X,T)$ is \emph{minimal} if $X$ itself is minimal. Equivalently, $M$ is minimal if $\overline{Tx}=M$ for all $x\in M$.

 We note that if $\phi: (X,T)\rightarrow (Y,T)$ is a flow homomorphism; then

  (a) If $M$ is minimal subset of $X$ then $\phi(M)$ is minimal subset of $Y$.

   (b) If $N$ is minimal subset of $\phi(X)$ then there exists a minimal subset $M$ of $X$ with $\phi(M)=N$.

 A point $x\in X$ is called \emph{almost periodic} if $\overline{Tx}$ is minimal set. A flow $(X,T)$ is called \emph{pointwise almost periodic} if every $x\in X$ is almost periodic.

 \vskip .5cm

 The (semi)flow $(X,T)$ is called  \emph{transitive(topologically transitive)} if for every nonempty, open subsets $U$ and $V$ in $X$, there exist a $t\in T$ such that $tU\cap V\neq\emptyset$.
$(X,T)$ is \emph{weak mixing} if the flow $(X \times X, T)$ is topologically transitive. $(X,T)$ is called a \emph{mixing system} if for any pair of nonempty open sets $U,V\subset X$, there exist a compact subset $C$ of $T$ such that $tU\cap V\neq\emptyset$, for all $t\in T\smallsetminus C$.

\vskip .5cm

Note that in case of cascades or semicascades $(X,f)$, one usually talks of forward orbits.  Hence for cascades and semicascades, the usual definitions are considered with the \emph{orbit} of $x$ as $\mathcal{O}(x) \ = \ \{f^n(x) : n \in \N \}$, \emph{topologically transitive} when for every  open and nonempty set $U \subset X$,
$\bigcup \limits_{n \in \N} \ f^n(U)$ is dense in $X$, or, equivalently if for every nonempty, open pair
$U, V \subset X$, there exists a $n \in \N$ such that the set  $f^n(U) \cap V$ is nonempty, \emph{weakly mixing} if given nonempty, open sets $U_1, U_2, V_1, V_2 \subset X$, there exists a $n \in \N$ such that $f^n(U_1) \cap V_1 \neq \emptyset$ and $f^n(U_2) \cap V_2 \neq \emptyset$.
The celebrated Fursternberg intersection lemma \cite{F} gives us that the cascade or semicascade $(X,f)$ is \emph{weakly mixing} if and only if for every $k \geq 2$, the product system $ (X^k, f^{(k)}) = (\underbrace{X \times \ldots \times X}_{k- \text{times}}, \underbrace{f \times \ldots \times f}_{k- \text{times}})  $ is topologically transitive. And the cascade or semicascade $(X,f)$ is \emph{topologically mixing} if for any pair of nonempty, open subsets $U, V$ of $X$, there exists  $N \in \N$, such that $ f^n(U) \cap V \neq \emptyset$ for $n \geq N$, with it being \emph{minimal} if $\mathcal{O}(x)$ is dense in $X$ for every $x \in X$.

In case of cascades or semicascades $(X,f)$, for any nonempty, open $U,V\subset X$, the set $N(U,V)=\lbrace n\in \N: f^{n}(U)\cap V \neq\emptyset\rbrace$ is called the hitting set. $(X,T)$ is called transitive if any such $N(U,V)$ is nonempty, weak mixing if $N(U,V)$ is thick, where a subset of natural numbers is called \emph{thick} if for each $k\in \N$, there exist $n_k \in \N$ such that it contains the set $\lbrace n_k, n_{k}+1,\ldots, n_{k}+k \rbrace$, and the system is mixing if $N(U,V)$ is cofinite.

\vskip .5cm

For cascades or semicascades $(X,f)$, the notion of \emph{strongly transitive} is studied in \cite{AANT}. See also \cite{ANVK}. $(X,f)$ is called \emph{Strongly Transitive} if for every nonempty, open $U \subset X$, $\bigcup\limits_{n=1}^\infty f^n(U) = X$, and \emph{Strongly Product Transitive} if for every positive integer $k$ the product system $(X^k, f^{(k)})$ is strongly transitive.

Again as in \cite{AANT, ANVK} we take the \emph{backward orbit} of $x \in X$ as
$$\mathcal{O}^\leftarrow(x) = \bigcup \limits_{n \in \N} \{ f^{-n}(x)   \} \ = \ \{ y \in X : f^n(y) = x \ \mbox{for some} \ n \in \N \}.$$

\begin{theorem} \cite{AANT} \label{1.1} For a  cascade or semicascade $(X,f)$  the
following are equivalent.
\begin{itemize}
\item[(1)]  $(X,f)$ is strongly  transitive.

\item[(2)] For every  nonempty, open set $U \subset X$ and every point $x \in X$, there
exists $n \in \N$ such  that $x \in f^n(U)$.

\item[(3)] For every $x \in X$,  the backward orbit $\mathcal{O}^\leftarrow(x)$ is dense in $X$.

\end{itemize}

If $(X,f)$ is strongly  transitive, then it is topologically transitive.
\end{theorem}

\begin{theorem}  \cite{AANT} \label{1.2} A cascade $(X,f)$ is strongly transitive
if and only if it is minimal.
\end{theorem}

We have the inclusion relationship as given in \cite{AANT}:

\vskip .5cm

\centerline{\scriptsize Mixing $\Rightarrow$ Weak Mixing  $\Rightarrow$ Transitivity $\Leftarrow$ Strongly Transitive $\Leftarrow$ Minimal.}

\centerline{\scriptsize Strongly Product transitive $\Rightarrow$ Weak Mixing.}

\vskip .5cm

We note here a few observations that can be helpful. Let $(X,T)$ be a topologically transitive flow with $X$  metrizable then there exists a $ x\in X$ with $\overline{Tx}=X$ means $(X,T)$ is point transitive. Let $(X,T)$ be a topologically transitive, pointwise almost periodic flow with $X$  metrizable then $(X,T)$ is minimal. Let $\pi:(X,T)\rightarrow (Y,T)$ be an epimorphism of flows and $(X,T)$ be topologically transitive, then $(Y,T)$ is topologically transitive.

 \vskip .5cm

 If $(X,d)$ is a metric space then the flow $(X,T)$ is said to be \emph{equicontinuous at a point $y\in X$} if for every $\epsilon> 0$ there exists a neighborhood $U$ of $y$ such that for every $x\in U$ and every $t\in T$; we have $d(tx,ty)<\epsilon$ and the flow is \emph{equicontinuous} if it is equicontinuous for every point of $X$. A cascade(semicascade) $(X,f)$ is equicontinuous if the family $\lbrace f^n: n\in \Z ( \N)\rbrace$ is equicontinuous on $X$. A flow $(X,T)$ is called sensitive at a point $y\in X$ if for some $\delta> 0$ and for every neighborhood $U$ of $y$ there exists a point $x\in U$ and some $t\in T$ for which  $d(tx,ty)>\delta$ and system is \emph{sensitive} if it is sensitive at every point of $X$. A minimal flow $(X,T)$ is either equicontinuous or sensitive \cite{AY}.

 \vskip .5cm

 If there is a dense set of equicontinuity points in $X$ then the flow $(X,T)$ is called \emph{almost equicontinuous} (AE) see \cite{AAB1, AAB2}. A transitive flow $(X,T)$ is either almost equicontinuous or sensitive \cite{AAB1}.

The flow $(X,T)$ is called \emph{hereditarily almost equicontinuous}(HAE) if every closed subsystem of $(X,T)$ is also \emph{almost equicontinuous} \cite{ES}.

A flow $(X,T)$ is called locally equicontinuous (LE) if $\forall\ x\in X$, the subflow $(\overline{Tx},T)$ is  equicontinuous \cite{ELBW3}. The cascade $(X,f)$, with $X =\{(r,\theta): 0 \leq r \leq 1; \theta \in \R\}$ and $f(r, \theta) = (r, r+\theta)$ is locally equicontinuous but not almost equicontinuous.

A flow $(X,T)$ is called \emph{weakly almost periodic} (WAP) if for every $g\in \mathcal{C}(X)$(the space of continuous complex valued functions with the topology of pointwise convergence), the set $\overline{\lbrace g\circ t: t \in T\rbrace}$ is relatively compact in $\mathcal{C}(X)$ \cite{D, ELN, MST}. There are examples of locally equicontinuous systems that are not  weakly almost periodic systems mentioned in \cite{ELBW3}.

\centerline{ In general,  WAP $\subseteq$ LE $\subseteq$ AE.}

A flow $(X,T)$ is \emph{non sensitive} (NS)  if for every $\epsilon> 0$ there exists a neighborhood $U$ in $X$ such that for every $x,y \in U$ and every $t\in T$; we have $d(tx,ty)<\epsilon$. $(X,T)$ is called \emph{hereditarily non sensitive} (HNS) if every closed subsystem of $(X,T)$ is also \emph{non sensitive} \cite{GMES}.

Essentially non sensitive systems are those systems that have a non empty set of equicontinuity points. A transitive flow $(X,T)$ is \emph{non sensitive} (NS)  if and only if it is \emph{almost equicontinuous} \cite{GMES}. Though in the non transitive case we can have a non sensitive system which is not almost equicontinuous. And it is still not known if hereditary non senstive systems differ from hereditary almost equicontinuous systems and if the class of metrizable  hereditary non senstive and hereditary almost equicontinuous systems is
closed under factors and countable products     \cite{GMES}.

\vskip .5cm

 A pair $\{x,y\}$ is called \emph{proximal} if  there exists a net $\lbrace t_{i}\rbrace\subseteq T$ with $\lim\limits_i t_{i}x=\lim\limits_i t_{i}y$. The collection of proximal pairs in $X\times X$ will be denoted by $P(X)$ and a flow $(X,T)$ is \textit{proximal} if every pair $(x,y)\in X\times X$ is proximal ; i.e., $P(X)= X\times X$. For a cascade(semicascade) a pair $\{x,y\}$ is called \emph{proximal} if there exists a sequence $\lbrace n_{i}\rbrace\subseteq \Z ( \N)$ with $\lim\limits_i f^{n_i}(x)=\lim\limits_i f^{n_i}(y)$.

 From the definition, we can see that the relation of proximality is reflexive and symmetric but it is not transitive in general and so not always an equivalence relation.

 \vskip .5cm

 We recall a very important result in this direction;

  \emph{Auslander-Ellis Theorem \cite{AUS, ELL}}: Let $(X,T)$ be a flow and $x\in X$. Then  there exists an almost periodic point  $y\in X$  such that $(x,y)\in P(X)$.

  In fact, this theorem says more. One can specify the proximal almost periodic point. That is, if $M$ is a minimal set in the orbit closure of $x$, then there is an $x' \in M$ with $x$ and $x'$ proximal.

\vskip .5cm

We define $\Delta = \{(x, x) : x \in X \} \subset X \times X$, the \emph{diagonal} of $X \times X$. Thus $ (x, y) \in \Delta$ means that $x = y$. The flow $(X,T)$ is a \textit{distal flow} when the only proximal pairs are of the form $(x,x)\in X\times X$; i.e., $P(X) = \Delta $. Distality is preserved under homomorphisms and taking products.

Let $\pi:(X,T)\rightarrow (Y,T)$ be a homomorphism of flows. We say that $\pi$ is a proximal (distal) homomorphism if whenever $\pi(x)=\pi(y)$, the pair $x$ and $y$ is proximal (distal). And $(X,T)$ is called proximal (distal) extension of $(Y,T)$.

Let $(X,T)$ be topologically transitive and distal. Then $(X,T)$ is minimal.

\vskip .5cm

For a flow $(X,T)$, a point $x \in X$ is called \emph{recurrent}  if  the orbit of $x$ returns to its neighbourhood infinitely often. Usually this infinitely often is described in terms of some  \emph{`admissible set'}. A point $x \in X$ is called \emph{recurrent}  if for any $\epsilon > 0$ there is an admissible  $A \subset T$ such that  $d(t(x), x) < \epsilon \ $ for all $t \in A$,  and \emph{uniformly recurrent}  if and only if the orbit closure
$\overline{Tx}$ is a minimal set.

For a cascade $(X,f)$, a point  $x \in X$ is called periodic if there exists a $n \in \N$ such that $f^n(x) = x$, \emph{recurrent}  if there exists a sequence $n_k \nearrow \infty$  such that  $f^{n_k}(x) \to x$ i.e. the set $N(x,U) = \{ n \in \N: f^n(x) \in U\}$ is infinite for any neighbourhood $U$ of $x$,  and \emph{uniformly recurrent} if for any neighbourhood $U \ni x$, the set $N(x,U) = \{ n \in \N: f^n(x) \in U\}$ is syndetic(i.e. with bounded gaps). Note that periodic points and almost periodic points are uniformly recurrent. The set of all recurrent points in $X$ is denoted as $R_f(X)$.

The \emph{omega limit set} of a point $x\in X$, $\omega(x) = \lbrace y\in X: y=\lim\limits_k f^{t_{k}}x$ for some sequence $ (t_{k})\ in\ \N\rbrace$, and  is a non empty closed $f$-invariant set. A point $x$ is  recurrent if $x\in\omega(x)$. A point $x\in X$ is called \emph{non-wandering} if for every open set $U\subset X$ with $x\in U$, there exists a $n\in \N$ such that $f^n(U)\cap U \neq\emptyset$. The set of all non-wandering points in $X$ is denoted as $\Omega_f(X)$.

\vskip .5cm

The concept of recurrence is strengthened by the notion of rigidity. These concepts of \emph{weak rigidity, rigidity and uniform rigidity} in topological dynamics were first defined by Glasner and Maon \cite{RIG} for  cascades. A cascade $(X,f)$ is called \emph{weakly rigid} if for any $x_1, \ldots, x_n \in X$ and $\epsilon > 0$, there is a $m \in \N$ such that $d(x_i, f^m(x_i)) < \epsilon,$ $i = 1, \ldots, n$. This essentially means that  each n-tuple $(x_1, \ldots , x_n)$ in the product system $(X^n, f^{(n)})$ for each $n \in \N$ is a recurrent point. The cascade $(X, f)$ is called \emph{$n$-rigid} if the later property is satisfied for some $n \in \N$,  \emph{rigid} if there is  a sequence $n_k \nearrow \infty$ such that $f^{n_k}x \to x$  for all $x \in X$,  and
\emph{uniformly rigid} if the sequence $n_k \nearrow \infty$ is such that $f^{n_k} \to e$ uniformly on $X$. Here $e$ is the identity mapping on $X$.

 \vskip .5cm

\centerline{ \it  Uniformly Rigid $\Rightarrow$ Rigid $\Rightarrow$ Weakly Rigid}

\vskip .5cm

There are examples of uniformly rigid, rigid but not uniformly rigid and weakly rigid but not rigid cascades discussed in \cite{RIG}.

\begin{theorem} \cite{AAB1} Let (X, f) be a topologically transitive system. Then  $(X,f)$ is almost equicontinuous implies that  $(X,f)$ is uniformly rigid. \end{theorem}

\begin{theorem} \cite{ELBW1}  A topologically transitive system without isolated points
which is not sensitive is uniformly rigid. \end{theorem}

\subsection{Symbolic Dynamics}

\emph{Symbolic Dynamics} originated as a method to study general dynamical systems. The distinct feature in symbolic dynamics is that space and time are both discretized. Each state is associated with a symbol and the evolution of the system is described by an infinite sequence of symbols.  Related theory can be read from \cite{LIN}.

In general for the alphabet set $\mathcal{A}$ with $|\mathcal{A}| < \infty$, consider $\mathcal{A}^{\mathbb{Z}}$ then

 $\rho(x,y)= \left\{
               \begin{array}{ll}
                 2^{-(k+1)}, & \hbox{if  $x\neq y$ {and $k$ is maximal so that} $x_{[-k, k]}=y_{[-k, k]}$;} \\
                 0, & \hbox{$x=y.$}
               \end{array}
             \right.$

is a metric on $\mathcal{A}^{\mathbb{Z}}$ and the cylinder set $C_{k}(u)=\lbrace x\in X: x_{[k, k+|u|-1]}= u\rbrace;\ i.e,\ C_{k}(u)$ is the set of points in which the block $u$ occurs starting at position k. These cylinder sets are open sets and $C_{-n}(x_{[-n,n]})= B_{2^{-(n-1)}}(x)$. Hence if $x\in C_{k}(u)$ and $n$= max$\lbrace |k|, |k+|u|-1|\rbrace$, then $B_{2^{-(n-1)}}(x)\subset C_{k}(u)$. Cylinder sets are basic open set of the shift space.
If $|u|=2m+1$, set $C_{-m}(u)=[u]$. We shall usually talk of cylinders $[u]$ where the central word in the bi-infinite sequence is $u$.

For an alphabet set $\mathcal{A}$, the full $\mathcal{A}$-shift $(\mathcal{A}^{\mathbb{Z}},\sigma)$ is a dynamical system where $\mathcal{A}^{\mathbb{Z}}= \lbrace x=(x_{i})_{i\in \mathbb{Z}}: x_{i}\in \mathcal{A}\ \forall\ i\in \mathbb{Z}\rbrace$ and $\sigma:\mathcal{A}^{\mathbb{Z}}\rightarrow \mathcal{A}^{\mathbb{Z}}$ called \emph{shift map} is defined as $\sigma(x)=y$, where $y_{i}= x_{i+1}$. If $\mathcal{A}=\lbrace 0,1,2,...,r-1\rbrace$, then it $(\mathcal{A}^{\mathbb{Z}},\sigma)$ is called full \emph{r-shift}.

\vskip .5cm

If $x\in \mathcal{A}^{\mathbb{Z}}$ and $w$ is a block over $\mathcal{A}$, then $w$ occurs in $x$ if there
are indices $i$ and $j$ so that $w = x_{[i,j]}$. Note that the empty block $\varepsilon$ occurs in every $x$, since $\varepsilon = x_{[1,0]}$. Let $\mathcal{F}$ be a collection of blocks over $\mathcal{A}$, which
we will think of as being the \emph{forbidden blocks}. For any such $\mathcal{F}$, define $X_{\mathcal{F}}$ to be the subset of sequences in $\mathcal{A}^{\mathbb{Z}}$ which do not contain any block in $\mathcal{F}$.  In general, the collection $\mathcal{F}$ is  countably infinite.

A \emph{shift space} is a subset $X$ of a full
shift $\mathcal{A}^{\mathbb{Z}}$ such that $X = X_{\mathcal{F}}$ for some collection $\mathcal{F}$ of forbidden blocks over $\mathcal{A}$. Any closed invariant subset of a shift space is called a \emph{subshift}.

Let $X\subset \mathcal{A}^{\Z}$ and $\mathcal{B}_n(X)$ denote
the set of all $n-$blocks that occur in points in $X$ then the
collection $\mathcal{B}(X)=\bigcup\limits_{n=0}^{\infty} \mathcal{B}_n(X)$ is called the language of  $X$.

In the other way, a shift space $X$ can be defined as $X=X_{\mathcal{F}}$ where $\mathcal{F}= \mathcal{B}(X)^{c}$. If the set $\mathcal{F}$ can be taken to be finite then the $X_{\mathcal{F}}$ is called a \emph{subshift of finite type(SFT).}

We recall a standard example of subshift of finite type. The \emph{Golden Mean Shift} $X \subset \{0,1\}^{\Z}$ is the set of all binary sequences with no two consecutive $1$'s. Here $X=X_\mathcal{F}$ , where $\mathcal{F}=\{11\}$.

For subshifts $X$ and $Y$ (not necessarily on the same alphabet), we say that $Y$ factors on $X$ and write $X \rightarrow Y$ if there is a factor map from $X$ onto $Y$. A subshift is called \emph{sofic} if it is a factor of a subshift of finite type.

Again a standard example of a sofic shift is the \emph{Even Shift}. The Even Shift $Y \subset \{0,1\}^{\Z}$ is the set of all binary sequences so that between any two $1$'s there are an even number of $0$'s. One can take for $\mathcal{F}$ the collection $\{10^{2n+1}1:n \geq 0\}$. The Even Shift is a factor of the Golden Mean Shift.

And $X \stackrel{Per}{\rightarrow} Y$ means that the period of any periodic point of $X$ is divisible
by the period of some periodic point of $Y$. We say that $X\to Y$ if $X$ factors onto $Y$.

A shift space $X$ is irreducible if for every ordered pair
of blocks $u, v \in \mathcal{B}(X)$ there is a $w \in \mathcal{B}(X)$ so that $uwv \in \mathcal{B}(X)$.

Let $X$ be a shift space. The \emph{entropy} $h(X)$ of $X$ is defined by $h(X)=\lim\limits_{n\rightarrow\infty}\frac{1}{n} log|\mathcal{B}_n(X)|$.

\begin{theorem} \cite{B} \label{boy} [Boyle's factor theorem] Suppose $X$ and $Y$ are irreducible subshifts of finite type or irreducible sofic shifts such that
$X \stackrel{Per}{\rightarrow} Y$ and $h(X)>h(Y)$. Then $X \rightarrow Y$.\end{theorem}

\vskip .5cm

For us $\emph{2}=\lbrace 0,1 \rbrace$, then full shift over $\emph{2}$ denoted as $(\emph{2}^{\mathbb{Z}},\sigma)$ is called \emph{\textbf{2-shift}}.

The \emph{\textbf{2-shift}} is a mixing system with a dense set comprising of periodic points  of all periods.

A continuously differentiable map $f : \mathbb{S}^1 \to \mathbb{S}^1$ is called expanding if
$|f^{'}(x)| > 1$ for all $x \in \mathbb{S}^1$. We  define the degree of $f: \mathbb{S}^1 \to \mathbb{S}^1$ , denoted by $deg(f)$ to be the number(of preimages)$|f^{-1}(x)|$, for any $x \in \mathbb{S}^1$. Now for  expanding maps $f ,g : \mathbb{S}^1 \to \mathbb{S}^1 $,  $deg( f \circ g) = deg( f )deg(g)$. An expanding map $f : \mathbb{S}^1 \to \mathbb{S}^1$ is mixing. It is an easy exercise that if $f : \mathbb{S}^1 \to \mathbb{S}^1$ is an expanding map of degree $2$ then $f$ is a factor of the \emph{\textbf{2-shift}}.

\section{Dynamics of Induced Systems}

Let $X$ be a compact metric space. Let $2^{X}$ denote the set of all non empty closed subsets of $X$ i.e. $2^X=\ \lbrace A\subset X : \overline{A}=A,\ A\neq\emptyset \rbrace$. We denote $\mathcal{F}_1(X)=\lbrace \lbrace x \rbrace : x\in X\rbrace \equiv X$,  $ \ \mathcal{F}_n(X)=\lbrace F\in 2^X: \mathrm{|F|} = n \rbrace$ and $\mathcal{F}(X)= \bigcup \limits_{n=1}^{\infty}\mathcal{F}_n$ is the collection of all finite subsets of $X$. Define $\langle V_{1},V_{2},...,V_{k}\rangle= \lbrace A\in 2^{X}: A\subset \bigcup \limits_{i=1}^k V_{i}$ and $A \cap V_{i}\neq\emptyset$ \ for $1 \leqslant i \leqslant k \rbrace$, for any finite collection $\lbrace V_{1},V_{2},...,V_{k} \rbrace$ of subsets of $X$ and let  $\mathcal{B}= \lbrace \langle V_{1},V_{2},...,V_{k}\rangle : k=1,2,3,.. $and $V_{i}$ is open subset of $X$ for $1 \leqslant i \leqslant k \rbrace$. Then $\mathcal{B}$ is base for the topology $\nu$ on $2^{X}$ called the \emph{Vietoris Topology}. When $X$ is compact Hausdorff space then the space $2^{X}$ endowed with Vietoris Topology is also compact Hausdorff and $\overline{\mathcal{F}(X)}=2^X$.

For the metric space $(X,d)$ we have the \emph{Hausdorff metric} $H_d$ defined on $2^X$. For $ A, B \in 2^{X},\ H_d(A,B) = \max \left\lbrace \sup \lbrace d(a,B) : a \in A \rbrace, \sup \lbrace d(b, A) : b \in B \rbrace\right\rbrace $ where $d(a,B)= min \lbrace d(a,b): b \in B\rbrace$. Thus, $H_d(A, B) < \epsilon$ if and only if
each set is in the open $\epsilon$ neighbourhood of the other, or, equivalently, each point of $A$ is within $\epsilon$ of a point in $B$ and vice versa.

For a compact metric space, Hausdorff metric topology and Vietoris topology are equivalent.

Let $X$ be a compact Hausdorff space and $A_{i}\rightarrow A$ and $B_{i}\rightarrow B$ in $(2^{X},\nu)$ and $A_{i}\subset B_{i}\ \forall\ i$ then $A\subset B$.

Let $h:X\rightarrow Y$ be a continuous map then $h_*: 2^{X}\rightarrow 2^{Y}$, defined as $h_*(A)= h(A)$  is well defined because of compactness of $X$ and is continuous.

We refer the interested reader to look into \cite{IN,EM} for more details.

\subsection{For Cascades and Semicascades}

We consider our systems to be cascades or semicascades in this subsection.

\vskip .5cm

Let $X$ be a compact metric space and $f: X \to X$ be a homeomorphism or a continuous map. The cascade (semicascade) $(X,f)$ induces the cascade (semicascade) $(2^X,f_*)$.

We recall a few important results for induced systems from \cite{AN,PUN}:
\begin{enumerate}

\item Let $(X,f),(Y,g)$ be cascades and $\pi: X\rightarrow Y$ be a homomorphism of cascades then $\pi_*: 2^{X} \rightarrow 2^{Y}$ defined as $A \mapsto \pi(A)$ defines the action on $2^{X}$ and $\pi_{\ast}$ is homomorphism of cascades.

\item Let $X$ be a compact Hausdorff space and $i: X\rightarrow 2^{X}$ such that $i(x)=\lbrace x \rbrace$, then $i$ is a continuous  embedding.

\item If $(X,f)$ is distal and $M$ is minimal subset of $2^{X}$ then the subsystem $(M, f_*)$ is distal.

\item  For a surjective flow $(X,f)$, $(2^{X},f_*)$  is distal if and only if  $(X,f)$ is equicontinuous if and only if  $(2^{X},f_*)$  is equicontinuous.

\item Let $(X,f)$ be a cascade (semicascade) and $(2^{X},f_*)$ be the induced cascade (semicascade) then $(X,f)$ is weakly mixing  if and only if $(2^{X},f_*)$ is topological transitive if and only if $(2^{X},f_*)$ is weakly mixing .

\item Suppose $X$ is complete, separable metric space  with no isolated points and $(X,f)$ is weakly mixing then for  a transitive point $C$ of the induced system $(2^X,f_*)$;

\begin{enumerate}
\item $f^nC\cap f^mC = \emptyset$ for $m \neq n \in \mathbb{N}$.
\item The set $C$ is nowhere dense in $X$.
\item The set $C$ has infinitely many components.
\end{enumerate}
\end{enumerate}

\vskip .5cm

  We note that in $(2^X, f_{\ast})$ as transitivity is equivalent to weakly mixing and so a  transitive  $(2^X, f_{\ast})$ is always sensitive and never equicontinuous. So it is interesting to look into the cases of equicontinuity when $(2^X, f_{\ast})$ is not transitive. We discuss the cases for almost equicontinuity and local equicontinuity for the induced system $(2^X, f_{\ast})$.

\begin{lemma}
Suppose $X$ is a compact Hausdorff space and $Y\subset X$ such that $\overline{Y}=X$ then $\overline{\mathcal{F}(Y)}= 2^X$.
\end{lemma}
\begin{proof}
Consider the basic open set $\langle V_{1},V_{2},...,V_{k}\rangle$ in $2^X$ where each $V_i$ is an open set in $X$. Now since $Y$ is dense in $X$, so for each $V_i$, there is a $y_i\in Y$ such that $y_i\in V_i$. So the set $F= \lbrace y_1, y_2, \ldots, y_k \rbrace \in \langle V_{1},V_{2},...,V_{k}\rangle$ and since $F\in \mathcal{F}(Y)$, $\mathcal{F}(Y)$ is dense in $2^X$.
\end{proof}
\begin{theorem} \label{TH 2.4}
$(X,f)$ is almost equicontinuous if and only if $(2^X,f_{\ast})$ is  almost equicontinuous.
\end{theorem}
\begin{proof}
Suppose $(X,f)$ is almost equicontinuous. So there exists $Y\subset X$ such that $\overline{Y}=X$ and each point in $Y$ is point of equicontinuity. Now let $A=\lbrace a_1, a_2,\ldots, a_k \rbrace\in  \mathcal{F}(Y)$ be any finite set. Since each member of $Y$ is point of equicontinuity, so for each $a_i\in A$, for every $\epsilon>0$ there is an open set $U_i$ in $X$ such that $d(f^n(a_i), f^n(u))<\epsilon$ for each $n\in \mathbb{N}$ and for each $u\in U_i$.

Now consider the open set $\langle U_1, U_2,\ldots,U_k\rangle \ni A$ in $2^X$ and a $B\in \langle U_1, U_2,\ldots,U_k\rangle$. For  given $\epsilon>0$, for $n\in \mathbb{N}$;
\begin{center}
$H_d(f_{\ast}^n(A),f_{\ast}^n(B))= \max \left\lbrace \sup \lbrace d(f^n(a),f^n(B)) : a \in A \rbrace, \sup \lbrace d(f^n(b), f^n(A)) : b \in B \rbrace\right\rbrace$.
\end{center}
Since for each $b\in B$, there is some $U_i$ such that $b\in U_i$, $d(f^n(a_i), f^n(b))<\epsilon$ for each $n\in \mathbb{N}$.  Since $d(f^n(a_i),f^n(B))= inf\lbrace d(f^n(a_i), f^n(b)):b\in B \rbrace < \epsilon$ and for each $b\in B$, $d(f^n(b), f^n(A))= inf \lbrace d(f^n(b), f^n(a_i)) : a_i \in A \rbrace<\epsilon$, which means $H_d(f_{\ast}^n(A),f_{\ast}^n(B))<\epsilon$ for each $n\in \mathbb{N}$ and $\forall\ B\in \langle U_1, U_2,\ldots,U_k\rangle$. Thus $A$ is the point of equicontinuity in $2^X$. So each member of $\mathcal{F}(Y)$ is the point of equicontinuity in $2^X$ and since $\mathcal{F}(Y)$ is dense in $2^X$, $(2^X, f_{\ast})$ is almost equicontinuous.

 Conversely, suppose $(2^X, f_{\ast})$ is almost equicontinuous and $\mathcal{A}\subset 2^X$ is the dense set of equicontinuous points. Then for given $\epsilon > 0$, $\exists\ \delta > 0$ such that $\forall\ A\in \mathcal{A}$, and $B\in 2^X$
 \begin{center}
 $H_{d}(A, B)< \delta \Rightarrow H_{d}(f_{\ast}^n(A), f_{\ast}^n(B)) < \epsilon\ \forall\ n\in \N$.
\end{center}
Pick an $a\in A$ and $y\in X$ with $d(a,y)< \delta$. Then, \begin{center}
$H_{d}(f_{\ast}^n(A), f_{\ast}^n(A \cup \lbrace y \rbrace ) < \epsilon\ \forall\ n\in \N$.
\end{center}
This implies that $d(f^n(a), f^n(y))<\epsilon\ \forall\ n\in \N$, and so $a\in A$ is an equicontinuity point for $(X,f)$.

Thus $\bigcup \limits_{A\in \mathcal{A}}\lbrace a : a\in A \rbrace$ is a dense set of equicontinuity points  in $X$ and so $(X,f)$ is almost equicontinuous.
\end{proof}

We  observe the following  for locally equicontinuous systems;

\begin{theorem}
If $(2^X,f_{\ast})$ is locally equicontinuous $\Rightarrow (X,f)$ is locally equicontinuous.
\end{theorem}
\begin{proof}
The proof here is trivial since $x\longmapsto \lbrace x \rbrace$ is an isometry from $X \longmapsto 2^X$.
\end{proof}

However the converse of the above theorem need not to be true as can be seen in the below example;

\begin{example} \label{EX 3.1} Consider $X= \lbrace (r,\theta): 0\leq \theta \leq 2\pi,\ r\in \lbrace 1- \frac{1}{2^n}:n\in \N\rbrace  \cup \lbrace 0,1 \rbrace \rbrace $ and $f(r,\theta)=(r, \theta + 2\pi r)$ defined on $X$.

For this cascade $(X,f)$, we see that each orbit closure is an equicontinuous system, in fact $f$ is identity on the circle $r=1$. Hence $(X,f)$ is locally equicontinuous.

We now consider the induced system $(2^X, f_{\ast})$. Let $A= \lbrace (r,0): r\in \lbrace 1- \frac{1}{2^n}: n\in \N  \rbrace \cup \lbrace 0,1 \rbrace \rbrace$ and consider $\overline{\mathcal{O}(A)}\subset 2^X $.

For a fixed $m\in \N$,
%we see that \begin{center}
%$f^{2^m}(1- \frac{1}{2^m},0)= \lbrace 1- \frac{1}{2^m}, 2^m 2\pi (1- \frac{1}{2^m}))=(1- \frac{1}{2^m},0)$.
%\end{center}
$f(r, \theta)= (r, \theta + 2\pi r)$, $f^2(r, \theta)= (r, \theta + 4\pi r), \ldots, f^k(r, \theta)= (r, \theta + 2k\pi r)$ and for any $k \leq m$,
\begin{center}
$f^{2^m}(1- \frac{1}{2^k},0)= \lbrace 1- \frac{1}{2^k}, 2^m 2\pi (1- \frac{1}{2^k}))=(1- \frac{1}{2^k},0)$.
\end{center}
Consider the sequence $\lbrace A, f(A), f^2(A), f^{2^2}(A),\ldots, f^{2^n}(A),\ldots\rbrace$ in $2^X$. We note that this sequence is not a Cauchy sequence in $2^X$ as for any $k\in \N$, $H_{d}(f^{2^k}(A), f^{2^{k+1}}(A))> 1$.
\begin{center}
Note that $f^{2^k}(1- \frac{1}{2^{k+1}},0)= \lbrace 1- \frac{1}{2^{k+1}}, 2^k 2\pi (1- \frac{1}{2^{k+1}}))=(1- \frac{1}{2^{k+1}},\pi)$ whereas $f^{2^{k+1}}(1- \frac{1}{2^{k+1}},0)= (1- \frac{1}{2^{k+1}}, 0)$.
\end{center}
Since $2^X$ is compact, this sequence will have a convergent subsequence, and so there exists a sequence $n_{i}\nearrow \infty$ and $B\in 2^X$ such that $f^{2^{n_i}}(A)\longrightarrow B$. So $B\in \overline{\mathcal{O}(A)}$ in $2^X$.

We claim that $A$ is not a point of equicontinuity for $\overline{\mathcal{O}(A)}$. Given $\delta > 0$ there exists $n_{r}> n_{t}$ s.t.
\begin{center}
$H_{d}(B, f^{2^{n_r}}(A))< \frac{\delta}{2}$
and $H_{d}(B, f^{2^{n_t}}(A))< \frac{\delta}{2}$.
\end{center}
\begin{center}
Thus $H_{d}(f^{2^{n_r}}(A), f^{2^{n_t}}(A))< \delta$.\end{center} But
\begin{center}
$f^{2^{n_t}}(1- \frac{1}{2^{n_t}},0)= (1- \frac{1}{2^{n_t}},0)$ \text{and} \\
$f^{2^{n_t}}(1- \frac{1}{2^{n_r}},0)= (1- \frac{1}{2^{n_r}},\frac{2\pi}{2^{n_{r}-n_{t}}})$
\end{center} Hence
\begin{center}
$f^{2^{2n_t}}(1- \frac{1}{2^{n_t}},0)= (1- \frac{1}{2^{n_t}},0)$ \text{and} \\
$f^{2^{2n_t}}(1- \frac{1}{2^{n_r}},0)=(1- \frac{1}{2^{n_r}},\frac{4\pi}{2^{n_{r}-n_{t}}})$
\end{center}
So for some $k$, \begin{center}
$H_{d}(f^{2^{kn_t}}(f^{2^{n_r}}(A)), f^{2^{kn_t}}(f^{2^{n_t}}(A)))>1$.
\end{center}
Hence $A$ is not a point of equicontinuity in ($\overline{\mathcal{O}(A)}, f_{\ast})$. So $(2^X,f_{\ast})$ is not locally equicontinuous.
\end{example}

\begin{remark} From the above example, we also note that $(X,f)$ is hereditary almost equicontinuos but $(2^X,f_{\ast})$ is not hereditary almost equicontinuos since $(\overline{\mathcal{O}(A)}, f_{\ast})$ is a subsystem of  which fails to be almost equicontinuous.

However $(2^X,f_{\ast})$ is hereditary almost equicontinuos $\Rightarrow (X,f)$ is hereditary almost equicontinuos with the proof being obvious.
\end{remark}

And for weakly almost periodic (WAP) flows, we can say:

\begin{proposition} \label{wap} Suppose $(2^X,f_{\ast})$ is a weakly almost periodic (WAP) flow then $(X,f)$ is also weakly almost periodic (WAP).
\end{proposition}
\begin{proof}  Since $(2^X,f_{\ast})$ is WAP, $ \{\hat{g}\circ f_{\ast}^{n}: n \in \Z\}$ is relatively compact in the weak topology on $\mathcal{C}(2^X)$ for all $\hat{g} \in \mathcal{C}(2^X)$. Now each $g \in \mathcal{C}(X)$ induces a $\hat{g} \in \mathcal{C}(2^X)$ given as $\hat{g}(A) = \max \{ |g(a)|: a \in A\}$. And  $\hat{g}\circ f_{\ast}^{n}(\{x\}) = g \circ f^n(x), \ \forall n \in \Z, \ \forall x \in X$. Thus $ \{g\circ f^n: n \in \Z\}$ is relatively compact in the weak topology on $\mathcal{C}(X)$ for all $g \in \mathcal{C}(X)$. Thus, $(X,f)$ is WAP. \end{proof}

\begin{remark} The converse here is not true, and we shall look into that later. \end{remark}

\vskip .5cm

\begin{theorem}
 $(X,f)$ is uniformly rigid if and only if $(2^X,f_{\ast})$ is  uniformly rigid if and only if $(2^X, f_{\ast})$ is rigid.
\end{theorem}
\begin{proof}
$(2^X,f_{\ast})$ is  uniformly rigid trivially implies that $(X,f)$ is uniformly rigid and $(2^X,f_{\ast})$ is also  rigid.

Since $(X,f)$ is uniformly rigid then there is a sequence $\lbrace n_{k} \rbrace \nearrow \infty$ such that $f^{n_{k}}\longrightarrow e = I_{X}$ uniformly. We can easily see by the definition of $f_{\ast}$, that $f_{\ast}^{n_{k}}\longrightarrow e = I_{2^X}$ uniformly. Hence $(2^X,f_{\ast})$ is also uniformly rigid.

Let $(2^X,f_{\ast})$ be rigid. Then there is a sequence $\lbrace n_{k} \rbrace \nearrow \infty$ such that $f_*^{n_{k}}A \longrightarrow A$, $\ \forall A \in 2^X$. This gives the convergence of  $f_*^{n_{k}}\longrightarrow e = I_{2^X}$ over compacta and hence uniformly, since $2^X$ is a compact metric space. This proves that $(2^X,f_{\ast})$ is uniformly rigid.

\end{proof}

 \begin{remark}  Though  $(2^X,f_{\ast})$ is rigid vacuously implies that $ (X,f)$ is  rigid, we see in the below example taken from \cite{RIG} that the converse does not hold.

\begin{example}  Consider $X= \lbrace (r,\theta): 0\leq \theta \leq 2\pi,\ r\in \lbrace 1- \frac{1}{2^n}:n\in \N\rbrace  \cup \lbrace 0,1 \rbrace \rbrace $ and $f: X \to X$ defined as

$$f(r,\theta)= \left\{\begin{array}{ll}
                    (r, \theta + 2\pi(1-r)), & \hbox{$r \neq 1$;} \\
                    (r, \theta), & \hbox{$r=1$.}
                    \end{array}
                    \right.$$

This cascade $(X,f)$ is rigid as $f^{2^k} (x) \to x$ for all $x \in X$. But $(X,f)$ is not uniformly rigid. And so it is easy to see that $(2^X,f_{\ast})$ will not be rigid.
\end{example}
\end{remark}

With a different proof Li, Oprocha, Ye and Zhang \cite{LPYZ} show that:

\begin{theorem} \cite{LPYZ} Let $(X,f)$ be a cascade. The following are equivalent:

(1) $(X,f)$ is uniformly rigid;

(2) $(2^X,f_*)$ is uniformly rigid;

(3) $(2^X,f_*)$ is rigid;

(4) $(2^X,f_*)$ is weakly rigid. \end{theorem}

The cascade (semicascade) $(2^X, f_*)$ induces the cascade (semicascade) $(2^{2^X}, f_{**})$, which induces the cascade (semicascade) $(2^{2^{2^X}}, f_{***})$, $\ldots$, which induces $({2^{\cdot^{\udots}}}^{2^X}, f_{* \ldots *})$, $\ldots$. And we have,

\begin{theorem}
 $(X,f)$ is weakly mixing if and only if $(2^X,f_{\ast})$ is  weakly mixing if and only if $(2^{2^X}, f_{**})$ is weakly mixing $\dots$ if and only if $({2^{\cdot^{\udots}}}^{2^X}, f_{* \ldots *})$, $\ldots$ is weakly mixing $\ldots$.
\end{theorem}

\begin{theorem}
 $(X,f)$ is mixing if and only if $(2^X,f_{\ast})$ is  mixing if and only if $(2^{2^X}, f_{**})$ is mixing $\dots$ if and only if $({2^{\cdot^{\udots}}}^{2^X}, f_{* \ldots *})$ is mixing $\ldots$.
\end{theorem}

\begin{theorem}
 $(X,f)$ is equicontinuous (almost equicontinuous) if and only if $(2^X,f_{\ast})$ is  equicontinuous (almost equicontinuous) if and only if $(2^{2^X}, f_{**})$ is equicontinuous (almost equicontinuous) $\dots$ if and only if $({2^{\cdot^{\udots}}}^{2^X}, f_{* \ldots *})$ is equicontinuous (almost equicontinuous) $\ldots$.
\end{theorem}

\begin{theorem} \label{3.8}
 $(X,f)$ is uniformly rigid if and only if $(2^X,f_{\ast})$ is  uniformly rigid if and only if $(2^{2^X}, f_{**})$ is uniformly rigid $\dots$ if and only if $({2^{\cdot^{\udots}}}^{2^X}, f_{* \ldots *})$ is uniformly rigid $\ldots$.
\end{theorem}

\begin{theorem}
 $(2^X,f_{\ast})$ is  transitive if and only if $(2^{2^X}, f_{**})$ is transitive $\dots$ if and only if $({2^{\cdot^{\udots}}}^{2^X}, f_{* \ldots *})$ is transitive $\ldots$.
\end{theorem}

\vskip .5cm

\subsection{For flows}

Suppose $(X,T)$ is a  \emph{flow}, where $T$ is a discrete, Abelian topological group and $X$  a compact metric space.

Now for the cascade $(X,f)$, the induced dynamics of $(2^X,f_*)$ has been studied in \cite{AN, PUN} and the references therein. Almost all of this study with almost same results can also be carried out for the flow $(X,T)$ with $T$ mostly Abelian, though there has been no account of such a study in literature so far. We briefly summarize some of the results in our case and give some proofs for the sake of completion.

We note that most of the results listed below will also hold if $T$ is a monoid.

\vskip .5cm

Let $(X,T)$ be a  flow then $\pi^t: X\rightarrow X$ induces a map $\pi^t_*: 2^{X}\rightarrow 2^{X}$, inducing the flow $(2^X,T)$. We prove a few important dynamical properties of the induced dynamics of $(2^X,T)$ here:

\begin{enumerate}

\item If $(X,T)$ is distal and $M$ is minimal subset of $2^{X}$ then the subsystem $(M, T)$ is distal.

[We refer the reader to the next section for reference to the algebraic tools used in this proof.]

\vskip .2cm

This easily follows since if $(u \circ A, v \circ A)$ is a proximal pair in $(M,T)$ for $A \in M$ and $u \circ A \neq v \circ A$ with $u, v \in \beta T$, then there exists $a_1,a_2 \in A$ with $ua_1 \neq va_2$ such that $(ua_1, va_2)$ is a proximal pair in $(X,T)$.

\vskip .2cm

\item If $(X,T)$ is  distal then each almost periodic point in the induced system
$(2^X,T)$  is  distal and these points are dense in $2^X$.

\vskip .2cm

This can be seen since the almost periodic points in $2^X$ are distal by the above observation. Since every point of $X$ is distal, and finite sets are dense in $2^X$, the almost periodic points  in $2^X$ are dense.

\vskip .2cm

\item As observed in \cite{AN},   $(X,T)$ distal need not imply that $(2^X,T)$ is distal.

\vskip .2cm

\item     $(X,T)$ is equicontinuous if and only if  $(2^{X},T)$  is equicontinuous if and only if $(2^{X},T)$  is distal.

\vskip .2cm

It is simple to see that $(X,T)$ being equicontinuous implies that $(2^X, T)$ is also equicontinuous and that further implies that $(2^X,T)$ is distal. We will show that $(2^X,T)$ if distal implies that $(X,T)$ is equicontinuous.

For this we consider the  \emph{regionally proximal relation} $RP \subset X \times X$ for $(X,T)$.
We say that $(x,y) \in RP$ if
there are nets $\{x_n\}$ and $\{y_n\}$ in $X$ with $x_n \to x$ and $y_n \to y$, $z \in X$
and net $\{t_n\}$ in $T$ such that $({t_n}(x_n),{t_n}(y_n))\to (z,z)$ in $X \times X$.  It is
known  that $(X,T)$ is equicontinuous if and only if $RP$ is the identity relation in $X \times X$.

Suppose  $(X,T)$ is not equicontinuous.
Then there are $x \neq y \in X$ with  $(x,y) \in RP$. Then
we have nets $x_n \to x$, $y_n \to y$ and $t_n \in T$ with
$({t_n}(x_n),{t_n}(y_n))\to (z,z)$ for some $z \in X$.
Let $C=\{x_1,x_2, \dots\} \cup \{x\}$ and $D=\{y_1,y_2, \dots\} \cup \{y\}$. We
can always take $C \cap D=\emptyset$. Since  $(2^X,T)$ is distal, the orbit closure
of $(C,D)$ in $2^X \times 2^X$ should be minimal. Let $(C',D')$ be the limit point of the set $\{ ({t_n}(C ,D) \}$ in $2^X \times 2^X$. But $z \in C' \cap D'$, and so we have $C' \cap D' \neq \emptyset$. Now if the $T \times T$ orbit closure of $(C,D)$ were minimal, then $(C,D)$ would be in
the $T \times T$ orbit closure of $(C',D')$. But that would imply that  $C \cap D \neq \emptyset$, a contradiction.

\vskip .2cm

We note that this does not require  $T$ to be discrete and so would be true for any Abelian $T$.

\vskip .2cm

\item $(X,T)$ is almost equicontinuous if and only if $(2^X,T)$ is  almost equicontinuous. The proof being similar to the one for Theorem \ref{TH 2.4}.

    \vskip .2cm

\item When $T$ is Abelian, $(X,T)$ is weakly-mixing if and only if $(2^{X},T)$ is topological transitive if and only if $(2^{X},T)$ is weakly-mixing.

\vskip .2cm

For Abelian $T$, we use a characterization due to Karl Petersen \cite{PET} - $(X, T)$ is weakly mixing if and only if given nonempty
open subsets $A$ and $B$ of $X$ there is $t \in T$ such that $tA \cap A \neq \emptyset$ and
$tA \cap B \neq \emptyset$. This clearly gives that $(2^X, T)$ is weakly mixing implies that $(2^X,T)$ is topologically transitive implies that $(X,T)$ is weakly mixing. All we need to see is that $(X,T)$ is weakly mixing implies that $(2^X,T)$ is weakly mixing.

 Let $N^\pi(U,V) = \{ t \in T: t(U) \cap V \neq \emptyset\}$.  We refer to  \cite{AK, AAG} for the proof of  the beautiful \emph{Furstenberg Intersection Lemma} (c.f. \cite{F}) - for non empty, open $U_1, U_2, V_1, V_2 \subset X$, there exists $U_3, V_3 \subset X$ such that
 \begin{equation*}\label{1.5}
 N^\pi(U_3,V_3) \quad \subset \quad  N^\pi(U_1,V_1) \cap  N^\pi(U_2,V_2). \hspace{2cm}
 \end{equation*}

And again we refer to \cite{AK, AAG} for the proof that for Abelian $T$, $(X,T)$ is weakly mixing if and only if the product system $(X^n, T)$ is weakly mixing  for all $n \in \N$ if and only if $\{N^\pi(U,V): U, V \subset X$ are non empty, open $\}$ is a filter.

This gives us that $(X,T)$ is weakly mixing implies $(X^k,T)$ is topologically transitive  for all $k \in \N$ implies $ \bigcap \limits_{i=1}^n N^\pi(\langle U^i_{1},U^i_{2},...,U^ i_{k}\rangle, \langle V^i_{1},V^i_{2},...,V^i_{k}\rangle) \neq \emptyset$ for non empty, open  $U^i_{1}, U^i_{2}, \ldots , U^i_{k},  V^i_{1}, V^i_{2}, \ldots, V^i_{k} \subset X$ for any $k$ and $i = 1, \ldots, n$, implying that $(2^X, T)$ is weakly mixing.

We refer to \cite{AAG} to look into conditions when this would also be true for nonabelian $T$.

\vskip .2cm

\item Suppose  $(X,T)$ is weakly mixing with Abelian $T$. Then $(2^X,T)$ is topologically transitive and so admits a $G_\delta$ dense set of transtive points in $2^X$. For  any such transitive point $C$ of the induced system $(2^X,T)$, we have;

\begin{enumerate}
\item[(a)] $t'C \cap tC = \emptyset$ for $t' \neq t \in T$.
\item[(b)] The set $C$ is nowhere dense in $X$.
\item[(c)] The set $C$ has infinitely many components.
\item[(d)] Every point of $C$ is a transitive point for $(X,T)$.
\end{enumerate}

\vskip .2cm

For (a) we observe that if $t_1(C) \cap t_2(C) \neq \emptyset$, then there exists $c_1, c_2 \in C$ with $t_1(c_1) = t_2(c_2)$.  For $x \in X$  there
exists a net $k_n$ in $T$ such that $ {k_n}(C)$ converges
to $\{ x \}$. Hence, $\{ t_{i}({k_n}(c_i) ) \}$ converges to $t_i(x)$ for $i = 1,2$.  Since $T$ is abelian,
$ t_i({k_n}(c_i)) = {k_n}(t_i(c_i)) $ for $i = 1,2$ and  the two limits are the same
 as $t_1(c_1) = t_2(c_2)$. Thus, $t_1(x) = t_2(x)$ for all $x \in X$. Since
$x$ was arbitrary, $t_1 = t_2$.

Infact, each of the sets in $\{ t(C) : t \in T \}$ are pairwise disjoint.

\vskip .2cm

For (b) we note that if $C$ had a nonempty interior then since $X$ is perfect there would exist $U, V \subset C$ disjoint non empty, open subsets
of $C$.  Then by transitivity there is a  $t \in T$  such that $C \cap t(C) \not= \emptyset$. But  this would contradict (a).

\vskip .2cm

For (c) we see that for any set $\{ x_1, \ldots , x_n \}$  of $n$ distinct points in $X$ with
$2\epsilon < \min \{ d(x_i, x_k):$ for $i \not= k = 1,\ldots,n \}$, there
exists $t \in T$ such that $t(C)$ is $\epsilon$ close to $\{ x_1, \ldots , x_n \}$.  Then the $\epsilon$ ball centered at each
$x_i$ gives a partition of $C$ into $n$ non empty, open sets and hence the number of components of $C$ is at least $n$.

\vskip .2cm

For (d) we again recall the argument in (a) that there
exists a net $\lbrace t_n \rbrace$ in $T$ such that $ {t_n}(C)$ converges
to $\{ x \}$ for all $x \in X$. Hence, each $c \in C$ is a transitive point in $X$.

\end{enumerate}

The cascade $(2^X, T)$ induces the cascade $(2^{2^X}, T)$, which induces the cascade $(2^{2^{2^X}}, T)$, $\ldots$, which induces $({2^{\cdot^{\udots}}}^{2^X}, T)$, $\ldots$. And we have,

\begin{theorem}
 For Abelian $T$, $(X,T)$ is weakly mixing if and only if $(2^X,T)$ is  weakly mixing if and only if $(2^{2^X}, T)$ is weakly mixing $\dots$ if and only if $({2^{\cdot^{\udots}}}^{2^X}, T)$ is weakly mixing $\ldots$.
\end{theorem}

\begin{theorem}
 For Abelian $T$, $(X,T)$ is mixing if and only if $(2^X,T)$ is  mixing if and only if $(2^{2^X}, T)$ is mixing $\dots$ if and only if $({2^{\cdot^{\udots}}}^{2^X}, T)$ is mixing $\ldots$.
\end{theorem}

\begin{proposition} \label{wap} Suppose $(2^X,T)$ is a weakly almost periodic (WAP) flow then $(X,T)$ is also weakly almost periodic (WAP).
\end{proposition}
\begin{proof}  Since $(2^X,T)$ is WAP, $ \{\hat{f}\circ \pi^t_*: t \in T\}$ is relatively compact in the weak topology on $\mathcal{C}(2^X)$ for all $\hat{f} \in \mathcal{C}(2^X)$. Now each $f \in \mathcal{C}(X)$ induces a $\hat{f} \in \mathcal{C}(2^X)$ given as $\hat{f}(A) = \max \{ |f(a)|: a \in A\}$. And  $\hat{f}\circ \pi^t_*(\{x\}) = f \circ \pi^t (x), \ \forall t \in T, \ \forall x \in X$. Thus $ \{f\circ \pi^t: t \in T\}$ is relatively compact in the weak topology on $\mathcal{C}(X)$ for all $f \in \mathcal{C}(X)$. Thus, $(X,T)$ is WAP. \end{proof}

\begin{theorem}
 $(X,T)$ is equicontinuous (almost equicontinuous) if and only if $(2^X,T)$ is  equicontinuous (almost equicontinuous) if and only if $(2^{2^X}, T)$ is equicontinuous (almost equicontinuous) $\dots$ if and only if $({2^{\cdot^{\udots}}}^{2^X}, T)$ is equicontinuous (almost equicontinuous) $\ldots$.
\end{theorem}

\vskip .5cm

\section{Function Spaces}
Any study on enveloping semigroups cannot be complete without a proper background of function spaces. In this section we discuss some known facts about some function spaces and try to construct some more results about various topologies on functions spaces.

\begin{definition}
For any topological space $X$ with $C\subset X$, a compact subset, and $U\subset X$, an open subset,  the set $$(C,U)=\lbrace f:X\rightarrow X: f(C)\subset U  \rbrace$$
 forms a sub-basic open set in the \emph{compact-open  topology} on $X^X$ and is abbreviated as \emph{$c$-topology}.
\end{definition}

\begin{definition}
For a topological space $X$, if $x\in X$ and $U\subset X$  is open then the set $$(x,U)=\lbrace f:X\rightarrow X: f(x)\in U  \rbrace$$ forms a sub-basic open set of the \emph{point open topology} on $X^X$ and is abbreviated by \emph{$p$-topology}.
\end{definition}

We recall some important results, the reader is encouraged to look for more in \cite{DUG, WIL}

\begin{enumerate}
\item The $p$-topology is always contained in the $c$-topology.

 \item For a compact set $C\subset X$ and open set $U\subset X$, $\overline{(C,U)}\subset (C,\overline{U})$.

\item For each $x_0\in X$, let $c_{x_0}:X \rightarrow X$ be the constant map $x\rightarrow x_0$. The map $j: X\rightarrow X^X$ given by $x\rightarrow c_x$ is a homeomorphism of $X$ onto a subspace of $X^X$; thus $X$ can always be embedded in $X^X$.

\item Let $X_0\subset X$ be a subspace of $X$; then $X_{0}^X$ is homeomorphic to the subspace $S_0=\lbrace f\in X^X:f(X)\subset X_0 \rbrace \subset X^X$.

\item Let $X$ be a metric space and $X^X$ has the $c$-topology. If $\mathcal{F}\subset X^X$ is equicontinuous family on $X$ then $\overline{\mathcal{F}}\subset X^X$ is also equicontinuous.
\end{enumerate}
See \cite{DUG, WIL} for more details. We recall some important theorems here.

\begin{theorem}{\cite{DUG, WIL}}\textbf{(Arzela-Ascoli)}
Let $(X,d)$ be a metric space and $\mathcal{F}\subset X^X$ is a family satisfies;
\begin{enumerate}
\item $\mathcal{F}$ is equicontinuous on $X$.
\item $\overline{\lbrace f(x):f \in \mathcal{F} \rbrace}$ is compact for each $x$.
\end{enumerate}
Then $\overline{\mathcal{F}}$ is compact and equicontinuous on $X$.
\end{theorem}

\begin{theorem}\cite{DUG, WIL}
Let $(X,d)$ be a metric space. A sequence $\lbrace f_n \rbrace$ in $X^X$ converges to an $f\in X^X$ uniformly on every compact subset if and only if $f_n \rightarrow f$ in the $c$-topology of $X^X$.
\end{theorem}

\vskip .5cm

\begin{theorem} \cite{DUG, WIL}
On an equicontinuous family $\mathcal{F}$, the compact-open topology reduces to point-open topology.
\end{theorem}

\vskip .5cm

Now $\mathcal{C}(X,X)\subset X^X$ and $\mathcal{C}(2^X,2^X)\subset {(2^X)}^{2^X}$ the set of all continuous self maps on $X$ and $2^X$ respectively, we establish some inter-relations between them.

\begin{theorem} \label{cpe}
The space $\mathcal{C}(X,X)$ with the $c$-topology is embedded in the space $\mathcal{C}(2^X,2^X)$ in the $p$-topology.
\end{theorem}
\begin{proof}
Define the map $i:(\mathcal{C}(X,X),c)\rightarrow (\mathcal{C}(2^X,2^X),p)$ as $i(f)=f_{\ast}$ where $f_{\ast}(A)=f(A)$. As $\mathcal{F}_1(X)\equiv X$, we can easily see that $i$ is one-one.

Now let $(A,\langle U_1, U_2,\ldots,U_k\rangle)$ be a sub-basic open set in $(\mathcal{C}(2^X,2^X),p)$. We check whether $i^{-1}((A,\langle U_1, U_2,\ldots,U_k\rangle))\cap \mathcal{C}(X,X)$ is open in the $c$-topology of $\mathcal{C}(X,X)$.

\textbf{Claim:} $i^{-1}((A,\langle U_1, U_2,\ldots,U_k\rangle))\cap \mathcal{C}(X,X)= (A,U)\setminus \bigcup_{j=1}^n \lbrace f\in \mathcal{C}(X,X): f(A)\cap U_j= \emptyset\rbrace$ in $\mathcal{C}(X,X)$, where $U=\bigcup_{i=1}^n U_i$.

For some $j$, let $\overline{\lbrace f\in \mathcal{C}(X,X): f(A)\cap U_j= \emptyset\rbrace}$.
Then there is a net $\{f_{n}\}$ in $\lbrace f\in \mathcal{C}(X,X): f(A)\cap U_j=\emptyset\rbrace$ such that $f_n \longrightarrow h$ in the $c$-topology of $\mathcal{C}(X,X)$. Then $f_n(A)\longrightarrow h(A)$ and since for each $n$, $f_n(A)\cap U_j = \emptyset \Rightarrow h(A)\cap U_j = \emptyset \Rightarrow\lbrace f\in \mathcal{C}(X,X): f(A)\cap U_j= \emptyset\rbrace$ is closed set. Therefore $\bigcup_{j=1}^n \lbrace f\in \mathcal{C}(X,X): f(A)\cap U_j= \emptyset\rbrace$ is a closed set in $c$-topology of $\mathcal{C}(X,X)$. So $(A,U)\setminus \bigcup_{j=1}^n \lbrace f\in \mathcal{C}(X,X): f(A)\cap U_j= \emptyset\rbrace$ is open.

Now for any $g\in i^{-1}((A,\langle U_1, U_2,\ldots,U_k\rangle))\cap \mathcal{C}(X,X) \Rightarrow i(g)=g_{\ast}\in ((A,\langle U_1, U_2,\ldots,U_k\rangle))\cap \mathcal{C}(X,X)\Rightarrow g_{\ast}(A)\subset \bigcup_{i=1}^n U_i$ and $g_{\ast}(A)\cap U_j \neq \emptyset  \Leftrightarrow   g\in (A,U)\setminus \bigcup_{j=1}^n \lbrace f\in \mathcal{C}(X,X): f(A)\cap U_j= \emptyset\rbrace$.

Therefore $i$ is continuous. Hence $\mathcal{C}(X,X)$ is embedded in $\mathcal{C}(2^X,2^X)$ and the compact-open topology of $\mathcal{C}(X,X)$  coincides with the point-open topology of $\mathcal{C}(2^X,2^X)$.
\end{proof}

We now discuss some results relating the space $\mathcal{C}(X)$ of continuous complex valued functions on a topological space $X$ and the corresponding space $\mathcal{C}(2^X)$.
\begin{lemma}
The map $i:\mathcal{C}(2^X)\rightarrow \mathcal{C}(X)$ defined as $i(f)=f|_{\mathcal{F}_1(X)}$ is a continuous map.
\end{lemma}
\begin{proof}
In the $p$-topology of $\mathcal{C}(X)$, let $W=(x,(a,b))$ be a basic open set in $\mathcal{C}(X)$ then $i^{-1}(W)= (\lbrace x \rbrace, (a,b))$ which is open in the $p$-topology of $\mathcal{C}(2^X)$. So $i$ is a continuous map.
\end{proof}

\vskip .5cm

Not every $\alpha \in \mathcal{C}(2^X, 2^X)$ is induced  from some element in $\mathcal{C}(X,X)$. Those that do are called inducible. We note the following due to Janusz J. Charatonik and Wlodzimierz J. Charatonik \cite{CHA} for inducible mappings;

\begin{definition}\cite{CHA}
Let $X$ and $Y$ be compact Hausdorff spaces then $\alpha \in \mathcal{C}(2^X, 2^Y)$ is said to be \emph{inducible} if there is an $f\in \mathcal{C}(X,Y)$ such that $f_{\ast}=\alpha$.
\end{definition}

Given two mappings $g_1,g_2\in \mathcal{C}(2^X,2^Y)$, there is a relation $\prec$ defined as $g_1\prec g_2$ if and only if $g_1(A)\subset g_2(A)$ for all $A\in 2^X$. This relation is always reflexive, transitive and the symmetry exists if $g_1=g_2$.

\begin{theorem}\cite{CHA}  \label{cha}
Let $X$ and $Y$ be two compact Hausdorff spaces. A mapping  $g:2^X \rightarrow 2^Y$ is inducible if and only if each of the following three conditions are satisfied:
\begin{enumerate}
\item $g(\mathcal{F}_1(X))\subset \mathcal{F}_1(Y)$;
\item $A\subset B$ implies $g(A)\subset g(B)$ for every $A,B\in 2^X$.
\item $g$ is minimal with respect to the order $\prec$ i.e if there is an $f$ such that $f\prec g$ then $f=g$.
\end{enumerate}
\end{theorem}
\vskip .5cm

\section{Algebra of Topological Dynamics and Enveloping Semigroups}

For this section we rely mostly on the basics as given in \cite{AK, AUS, REL, EL, ELL, F, JOIN, ES, HIN}. Though this study was initiated by Robert Ellis and most of this was developed in his work \cite{ELLIS}, we refrain from quoting this source since we have not read that work. We mention here that the same work has been rewritten as \cite{ELL} which turns out to be one of our primary sources in this section.

\subsection{Stone-$\check{C}$ech compactification $\beta T$ of discrete $T$}

 As noted earlier, the Stone-$\check{C}$ech compactification $\beta T$ of $T$ is determined as:

(a) $T\subset \beta T$, with $\beta T$ compact, Hausdorff

(b) $\overline{T}= \beta T$

(c) if $\tau:T\rightarrow Z$ be any function where $Z$ is compact Hausdorff space then there exists a unique continuous extension $\hat{\tau}$ of $\tau$, given as $ \hat{\tau}:\beta T\rightarrow Z$.

Let $T$ be a discrete group or monoid with the identity element $e$, so that $T$ is provided with an associative binary operation: $T\times T\rightarrow T$ as $(s,t)\mapsto st$, then the left multiplication map $L_{t}: \beta T\rightarrow \beta T$ as $p\mapsto tp$ is continuous for all $t\in T$ and right multiplication $R_p(q)= qp$ is also continuous for all $p\in \beta T$. Thus we see that the group or semigroup structure of $T$ can be extended to $\beta T$ and left $L_{t}$ and right multiplication $R_{p}$ are continuous. Also if $p,q\in \beta T$ such that $pq\in T$ then $p,q \in T$. So, if $p\in \beta T$ and it has an inverse $q$ i.e. $pq= e$ then  $e = pq \in T \ \Rightarrow\  p,q\in T$. So, the elements which has an inverse must be from $T$, not from $\beta T$.

The theory of \emph{Filters} is used for the construction of Stone-$\check{C}$ech compactification $\beta T$ of $T ( \neq \emptyset)$. We recall that for a given set $S$ with the partial ordering $\subseteq$, a family $F$ of subsets of $S$ is called a \emph{filter} if
(i) $\emptyset\notin F$, (ii) if $A, B \in F$ then $A\cap B\in F$, and (3) if $A\in F$ and $A\subseteq B$ then $B\in F$.
The maximal filter under $\subseteq$ is called an \emph{ultrafilter}. The structure of $\beta T$ is based on \textit{ultrafilters}.
\begin{center}
$\beta T = \lbrace \mathcal{U}: \mathcal{U} \text{ is an ultrafilter on T}\rbrace $.
\end{center}

and for $\emptyset\neq A\subset T$, the hull of A is defined as  $h(A)=\lbrace \mathcal{U} \in \beta T : A\in \mathcal{U}\rbrace$.

Let $t \in T$. Then the collection $h(t)= \lbrace A : t\in A\subset T\rbrace$ is an ultrafilter on $T$ and $\tau=\lbrace \Gamma \subset \beta T: \text{for every}\ \mathcal{U} \in \Gamma,\ \exists\ A \in \mathcal{U}\ \text{such that}\ h(A)\subset \Gamma\rbrace$ is a topology on $\beta T$, $\mathcal{B}=\lbrace h(A): A\subset T\rbrace$ is base for $\tau$ and $h(A)$ is the closure of a subset $A$ in $\beta T$.

\vskip .5cm

Also $\beta T$ is compact Hausdorff space and so $(\beta T, T)$ is a flow.

\vskip .5cm
Now for any $\mathcal{U} \in \beta T$ and $t\in T,\ t \mathcal{U}  = \lbrace tA: A\in \mathcal{U}\rbrace = \lbrace A: t^{-1}A\in \mathcal{U} \rbrace$. Since $R_{t}(A)= At$ and $L_{t}(A)= tA$. So, $ t \mathcal{U}=\lbrace tA: L_{t}^{-1}(A)\in \mathcal{U}\rbrace$. So, $A\in t \mathcal{U} \Rightarrow L_{t}^{-1}(A)\in \mathcal{U} \Rightarrow t\in \lbrace s: L_{s}^{-1}(A)\in \mathcal{U} \rbrace \Rightarrow \lbrace s: L_{s}^{-1}(A)\in \mathcal{U} \rbrace \in h(t)$ and vise-versa. So,\ $A\in \mathcal{U} t\Leftrightarrow t\in \lbrace s: L_{s}^{-1}(A)\in \mathcal{U} \rbrace \Leftrightarrow \lbrace s: L_{s}^{-1}(A)\in \mathcal{U} \rbrace \in h(t)$.

From this point of view, we can generalize the concept of \textit{product} of two ultrafilter by identifying $t\equiv h(t)$ is an ultrafilter. So, we can define $A\in \mathcal{U}\mathcal{V} \Leftrightarrow \lbrace s: L_{s}^{-1}(A)\in \mathcal{V} \rbrace \in \mathcal{U}$.

So, for $T$ to be a semigroup, $t\in T,\ A\subset T$ and $\mathcal{U} \in \beta T,\ tA=L_{t}(A)= \lbrace ta: a\in A \rbrace$. Then $\mathcal{U} \ast A = \lbrace s: L_{s}^{-1}(A)\in \mathcal{U} \rbrace$ and if $T$ is a semigroup and $\mathcal{U},\mathcal{V}\in \beta T$ and\ $\omega= \lbrace A\subset T: \mathcal{U} \ast  A \in \mathcal{V}\rbrace$. Then $\omega$ is an ultrafilter on $T$.

So, the \textit{product} of ultrafilters is defined to make it a semigroup. For $\mathcal{U},\mathcal{V}\in \beta T$, \textit{Product} is defined as $\mathcal{U}\mathcal{V}= \lbrace A\subset T:\mathcal{U} \ast A \in \mathcal{V} \rbrace$ and the result
$A\ast(\mathcal{U}\mathcal{V})= (A\ast\mathcal{U})\ast\mathcal{V}$ proves that for $\mathcal{U},\mathcal{V},\mathcal{W}\in \beta T \Rightarrow (\mathcal{U}\mathcal{V})\mathcal{W}=\mathcal{U}(\mathcal{V}\mathcal{W})$, making $\beta T$ a semigroup.

 Thus for the semigroup $\beta T$ and for $t\in T,\ \mathcal{U}\in \beta T,\ R_{\mathcal{U}}:\beta T\rightarrow \beta T$ and $L_{t}:\beta T\rightarrow \beta T$ are right and left multiplications then both are continuous.

\vspace{ .5cm}

If $(X,T)$ is a point transitive flow, there always exists an epimorphism $\psi: (\beta T,T) \rightarrow (X,T)$ means every point transitive flow is homomorphic image of $\beta T$.

Since $(X,T)$ is point transitive then $\exists$ $x_0 \in X$ such that $\overline{Tx_0}= X$. Define a map $\eta:T \rightarrow X$ defined by $\eta(t)=tx_0$, then there is a unique continuous extension $\hat{\eta}(p) =px_0 \ \forall\ p \in \beta T.\qquad [because \ \overline{T}=\beta T]$

 Since $(X,T)$ is a flow, so, $(t t')x_0 = t(t'x_0)$. So, by the uniqueness of extensions, we get; $(tp)x_0 =t(px_0)$. Now for $p\in \beta T$ and $t\in T$, $\hat{\eta}(tp)=(tp)x_0=t(px_0)=t \hat{\eta}(p)$. So $\hat{\eta}$ is a flow homomorphism.

\subsection{Enveloping Semigroups $E(X)$ for flows $(X,T)$}

 Since $T$ acts on $X$ then every point of $T$ gives a mapping  on the space $X$. Therefore the space $X^X$ also plays an important role in the theory. Since $X$ is compact, $X^X$ - the space of all maps from $X$ to $X$, is also compact Hausdorff space endowed with the topology of pointwise convergence. We denote $f(x)$ as the image of $x$ under $f$ and $fg$ is the composition of $f$ and $g$ and $\rho: X^X \times X^X\rightarrow X^X$ is the map defined as $\rho(f,g)=fg \ \forall\ f,g\in X^X$, which provides the semigroup structure on $X^X$. Now for every $f \in X^X,\ \rho^{f}: X^X \rightarrow X^X$ defined as $\rho^{f}(g) = gf$  and  $\rho_{f}: X^X \rightarrow X^X$ as $\rho_{f}(g)=fg$ are continuous. Thus $X^X$ is a semigroup.

If $X^X\times X\rightarrow X$ is given as $(f,x)\mapsto f(x)$ then $(f_1f_2,x)\mapsto (f_1f_2)(x)=f_1(f_2(x))=(f_1,(f_2(x))=(f_1,(f_2,x))$ and $(id,x)\mapsto id\cdot x=x$, which shows that $X^X$ acts on $X$.

If T is any discrete subgroup of $X^X$ which consists entirely of continuous maps, then $(X,T)$ is a flow.
Conversely, if $\pi(X,T)$ is a flow then consider the set $A=\lbrace\pi^t: t\in T\rbrace$ where $\pi^t:X\rightarrow X$ defined as $\pi^t(x)=tx$ with each  $\pi^t$ continuous and $\pi^{t_1}\pi^{t_2}(x)=\pi^{t_1}(t_2 x)= t_1(t_2x)= (t_1 t_2)x=\pi^{t_1t_2}(x)$ and $\pi^{t^{-1}}(x)=t^{-1}x=(\pi^t)^{-1}(x)$, so $\pi^{t^{-1}}=(\pi^{t})^{-1}$.
So, $A$ is a discrete subgroup of $X^X$ consisting of continuous maps.

Now the map $T\to X^X$ defined as $t\to \pi^{t}$ has a continuous extension $\Phi_{X}:\beta T \rightarrow X^X$ and $\Phi_{X}(\beta T)=\overline{\lbrace\pi^{t}:t\in T\rbrace}=E(X, T) = E(X)$ is a subgroup of $X^X$ called the \emph{Enveloping semigroup} of the flow $(X,T)$. We give a simple example here:

\begin{example} \label{EX 5.1}

Consider the cascade $([0,1],f)$ where $f(x)=x^2$.
Enveloping semigroup $E(X)$ here is just the two point compactification of $\Z$ as $f^{n} \longrightarrow g_{1}$ and $f^{-n} \longrightarrow g_{2}$ as $n \longrightarrow \infty$, where

\centerline{$g_{1}(x)= \left\{
             \begin{array}{ll}
               0, \ &  x\in [0,1);  \\
               1, & x = 1.
             \end{array}
           \right.$
and $g_{2}(x)= \left\{
                 \begin{array}{ll}
                   0, & x = 0; \\
                   1, &  x\in (0,1].
                 \end{array}
               \right.$}

\end{example}

Enveloping semigroups are the most essential tools in the algebraic theory of topological dynamics. Their algebraic structure leads to characterization of some important dynamical properties of the flow $(X,T)$. We briefly recall some results:

\begin{itemize}

\item  The action $T\times E(X)\rightarrow E(X)$ as $(t,p)\mapsto \pi^{t}p$, where $\pi^{t}=L_{t}:E(X)\rightarrow E(X)$ as $L_{t}(q)=tq$ and $E(X)$ is point transitive.

     \item $\Phi_{X} :\beta T\rightarrow E(X)$ is both a flow and a semigroup homomorphism.

      \item The map $\psi :E(X)\rightarrow X$ as $p\mapsto px$ is a flow homomorphism for all $x\in X$.

       \item The map $\Phi_{\beta T}: \beta T\rightarrow E(\beta T, T)$ is an isomorphism.

       \item Let $\phi:(X,T)\rightarrow (Y,T)$ be a homomorphism of flows, then $\phi(px)=p\phi(x)$ for all $x\in X$ and $p\in \beta T$.

\item Let $f:(X,T)\rightarrow (Y,T)$ be a surjective flow homomorphism. Then there exist a map $\theta:E(X)\rightarrow E(Y)$ defined as $\theta(a)= \phi_{Y}(p)$ where $p \in \beta T$ with $\phi_{X}(p)= a$.

  (a) $\theta$ is surjective and continuous.

  (b) $\theta(pq)=\theta(p)\theta(q) \forall\ p,q\in E(X)$ so that $\theta$ is both flow and a semingroup homomorphism.

  (c) If $\psi:E(X)\rightarrow E(Y)$ is a homomorphism with $\psi\phi_{X}=\phi_{Y}$ then $\theta=\psi$.

%\begin{proof}
%Let $a\in E(X)$ and define $\theta(a)=\Phi_{Y}(p)$ where $p\in \beta_{T}$ such that $\Phi_{X}(p)=a$.
%\\ \textbf{$\theta$ is well defined}: Suppose $p,q\in \beta T$ such that $\Phi_{\beta T}=\Phi_{X}(q)$ and Let $y\in Y$ \\ Since $f$ is onto, so $\exists\quad x\in X$ such that $f(x)=y$, \\then $y\Phi_{Y}(p)=f(x)\Phi^{Y}(p)=f(x\Phi_{X}(p))=f(x\Phi_{X}(p))=f(x)\Phi_{Y}(q)=y\Phi_{Y}(q)$ \\ so $\Phi_{Y}(p)=\Phi_{Y}(q)$. \\ So, if $a=b\Rightarrow \theta(a)=\theta(b)$ \\ Also $\theta\Phi_{Y}(p)=\Phi_{Y}(p)\qquad$ [$\because$ by definition]
%\\ So, the top half is commutative............................
%\end{proof}

 \item If $\pi: T\times X\rightarrow X$ is a continuous action then we can define $\pi_{1}:T\times (X\times X)\rightarrow X\times X$ by
\begin{center}
$\pi_{1}(t,(x,y))=(\pi(t,x),\pi(t,y))=(tx,ty)$.
\end{center}
So, in view of last result, $E(X)\cong E(X\times X)$.

\item For the flow $(X,T)$ if $x_{0}\in X$ is such that $\overline{Tx_{0}}=X$ and $\theta:E(X)\rightarrow X$ be defined by $\theta(p)=px_{0} \ \forall\ p\in E(X)$ then $\theta$ induces an isomorphism $E(E(X))\cong E(X)$.
\end{itemize}

\vskip .5cm

We observe that even if $X$ is a countable space,  $E(X)$ can be uncountable.

\begin{example} \label{EX 5.2}

For each $n\in \N$, let $X_{n}=\lbrace (r,\theta)=(\frac{1}{2^n},\frac{2k\pi}{2^n}(mod\ 2\pi)): k=0,1,2,\ldots\rbrace$
and $X= \bigcup\limits_{n\in \N} X_{n} \bigcup \lbrace (0,0) \rbrace$ as a subspace of $\mathbb{R}^2$.

\begin{center}

Define $f:X\rightarrow X$ as $f(r,\theta)=(r, \theta + 2\pi r(mod\ 2\pi))$.

\end{center}

For any $s\in \N$,

\begin{center}

$f^{s}(\frac{1}{2^n},\theta)=(\frac{1}{2^n}, \theta + \frac{2s\pi}{2^n}(mod\ 2\pi))$

\end{center}

Let $s$ be a 2-adic integer. Suppose $s= 2^{s_1}+2^{s_2}+ \ldots +2^{s_r}$ then
$f^{s}(\frac{1}{2^n},\theta)=(\frac{1}{2^n},\theta + 2\pi( \frac{2^{s_1}+2^{s_2}+ \ldots +2^{s_r}}{2^n})(mod\ 2\pi))=(\frac{1}{2^n},\theta + 2\pi(\frac{1}{2^{n-{s_1}}}+\frac{1}{2^{n-{s_2}}}+\ldots+\frac{1}{2^{n-{s_r}}})(mod\ 2\pi))$.

Let $a = \ldots 10101 = 1+4+16+ \ldots$ be a 2-adic integer.
Then for the function $f_{a}$ defined as $f_{a}(r,\theta) = (r,\theta + 2k\pi a_{n}(mod\ 2\pi))$, where $a_{2k} = \frac{1}{2^2}+\frac{1}{2^4}+\frac{1}{2^6}+\ldots + \frac{1}{2^{2k-2}}+\frac{1}{2^{2k}}$ and $a_{2k+1}= \frac{1}{2}+\frac{1}{2^3}+\frac{1}{2^5}+\ldots + \frac{1}{2^{2k-1}}+\frac{1}{2^{2k+1}}$, we see that $f_a$ will be a member of $E(X)$ corresponding to $a$.

In general for any 2-adic integer $q$ where $q= \lbrace x_{n}
\rbrace_{n\in \N}$ as $x_{n}=0,1$,
the corresponding map $f_{q}\in E(X)$ will be defined as
$f_{q}(r,\theta)= (r, \theta + 2k\pi a_{n})$
where $a_{n}= \frac{1}{2^n} \sum\limits_{n\in \N} \ x_{k} 2^k$ and $r=
\frac{1}{2^n}$.

Hence the enveloping semigroup $E(X)$ is isomorphic to the group of 2-adic integers.
Observe that $X$ is a countable space but $E(X)$ is uncountable.
\end{example}

\begin{example} \label{EX 5.3} Consider the cascade $(X,f)$, where $X= \lbrace (r,\theta): 0\leq \theta \leq 2\pi,\ r\in \lbrace 1- \frac{1}{2^n}:n\in \N\rbrace  \cup \lbrace 0,1 \rbrace \rbrace $ and $f(r,\theta)=(r, \theta + 2\pi r)$. For any $k\in \N$, $f^{k}(1- \frac{1}{2^n},\theta)=(1- \frac{1}{2^n},\theta + 2k\pi(1- \frac{1}{2^n}))=(1- \frac{1}{2^n}, \theta - \frac{2k\pi}{2^n})$.
$f_{a}\in E(X)$ corresponding to the 2-adic integer
$a = \ldots 10101 = 1+4+16+ \ldots$ can be described as follows: $f_{a}(r,\theta) = (r,\theta - 2k\pi a_{n})$, where
$a_{2k} = \frac{1}{2^2}+\frac{1}{2^4}+\frac{1}{2^6}+\ldots + \frac{1}{2^{2k-2}}+\frac{1}{2^{2k}}$ and $a_{2k+1}= \frac{1}{2}+\frac{1}{2^3}+\frac{1}{2^5}+\ldots + \frac{1}{2^{2k-1}}+\frac{1}{2^{2k+1}}$. Also any integer $p$ will be equivalent to a 2-adic terminates at zeroes. So $p$ will be associate with the map $f^{p}$ from the enveloping semigroup. In general for any 2-adic integer $q$ where $q= \lbrace x_{n} \rbrace_{n\in \N}$ as $x_{n}=0,1$, the corresponding map $f_{q}\in E(X)$ will be defined as $f_{q}(r,\theta)= (r, \theta - 2k\pi a_{n})$ where $a_{n}= \frac{1}{2^n} \sum\limits_{n\in \N} \ x_{k} 2^k$ and $r= 1- \frac{1}{2^n}$.

%$\frac{2^{2k}-1}{3\cdot2^{2k}} \longrightarrow \frac{1}{3}$, $a_{2k+1} = \frac{2^{2k+2}-1}{3 \cdot 2^{2k+1}}\longrightarrow \frac{2}{3}$.
It's enveloping semigroup $E(X)$ is isomorphic to the group of 2-adic integers.
\end{example}

\begin{example} \label{EX 2.2} Consider the cascade $(X,f)$, where $X= \lbrace (r,\theta): 0\leq \theta \leq 2\pi,\ r\in \lbrace 1- \frac{1}{3^n}:n\in \N\rbrace  \cup \lbrace 0,1 \rbrace \rbrace $ and $f(r,\theta)=(r, \theta + 2\pi r)$.

\vskip .5cm

For any $k\in \N$, $f^{k}(1- \frac{1}{3^n},\theta)=(1- \frac{1}{3^n},\theta + 2k\pi(1- \frac{1}{3^n}))=(1- \frac{1}{3^n},\theta - \frac{2k\pi}{3^n})$.
$f_{a}\in E(X)$ corresponding to the 2-adic integer
$a = \ldots 10101 = 1+9+27+ \ldots$ can be described as follows: $f_{a}(r,\theta) = (r,\theta - 2k\pi a_{n})$, where
$a_{2s} = \frac{1}{3^2}+\frac{1}{3^4}+\frac{1}{3^6}+\ldots + \frac{1}{3^{2s-2}}+\frac{1}{3^{2s}}$ and $a_{2s+1}= \frac{1}{3}+\frac{1}{3^3}+\frac{1}{3^5}+\ldots + \frac{1}{3^{2s-1}}+\frac{1}{3^{2s+1}}$. By the same procedure in last example, it's enveloping semigroup is isomorphic to the group of 3-adic integers. In general for a prime $p$, and cascade $(X,f)$, where $X= \lbrace (r,\theta): 0\leq \theta \leq 2\pi,\ r\in \lbrace 1- \frac{1}{p^n}:n\in \N\rbrace  \cup \lbrace 0,1 \rbrace \rbrace $ and $f(r,\theta)=(r, \theta + 2\pi r)$ is a self-map on $X$, it's enveloping semigroup $E(X)$ is isomorphic to the group of p-adic integers.
\end{example}

\underline{QUESTION:} What can we say about $E(X)$ for the cascade $(X,f)$, when $X= \lbrace (r,\theta): 0\leq \theta \leq 2\pi,\ r\in \lbrace 1- \frac{1}{j^n}:n\in \N, j =2,3 \rbrace  \cup \lbrace 0,1 \rbrace \rbrace $ and $f(r,\theta)=(r, \theta + 2\pi r)$ is a self-map on $X$.

In particular, what can we say about $E(X)$ for the skew product $(X,f)$, when $X= \lbrace (r,\theta): 0\leq \theta \leq 2\pi,\ r\in [0,1]\} $ and $f(r,\theta)=(r, \theta + 2\pi r)$ is a self-map on $X$.

\subsection{Ideals in $E(X)$}

Let $(X,T)$ be a flow and $E(X)$ be its enveloping semigroup. Then a non empty subset $I\subset E(X)$ is called the \emph{left ideal} if $E(X)\cdot I\subset I$; i.e., $\alpha \in I$, $p\in E(X) \Rightarrow p \alpha \in I$.  Throughout this article an \emph{ideal} for us  is  a left ideal.

$I$ is a \emph{minimal ideal} if and only if $I$ is closed in $E(X)$ and  does not contain any left ideal as a proper subset. For a minimal ideal $I$, the flow $(I,T)$ is minimal. Again Zorn's Lemma guarantees that for any flow $(X,T)$, if $I$ is any ideal in $E(X)$ then $I$ contains a minimal ideal. We recall an important lemma due to K. Nakamura here:

\begin{lemma}\cite{NAMA}
If $G$ is a compact semigroup for which one-sided multiplication $x\rightarrow xx_0$ is continuous, then $G$ contains an idempotent.
\end{lemma}
\begin{proof}
Let $A$ be a minimal subset of $G$ satisfying $(i)AA\subset A$, $(ii)A$ is compact. Since $G$ itself satisfy these properties, so such a family will be non empty and Zorn's lemma will guarantee the existence of a minimal set of this kind. Take any $u\in A$. $Au \subseteq A$ is compact and, moreover $Au.Au \subset Au$ whence $Au=A$. So for some $v\in A, vu=u$. Let $A^{\prime}= \lbrace v: vu=u \rbrace$. The set $A^{\prime}$ is non empty, compact, and $A^{\prime} A^{\prime} \subset A^{\prime}$. Thus $A^{\prime}=A$, and hence $u\in A^{\prime}$, and so $u^2=u$.
\end{proof}

Since every continuous map from a compact space to a Hausdorff space is a closed map. So, consequently if $G$ is a compact Hausdorff semigroup such that the maps $R_{x}: G\rightarrow G$ given as $R_{x}(y)= yx $ are continuous $\forall\ x\in G$ then $\exists$ an \emph{idempotent}  $u^2 = u \in G$.

Thus every minimal ideal contains an idempotent. The idempotents in a minimal ideal has an important role.

\vspace{0.5cm}

Let $(X,T)$ be a flow and $I\subset E(X)$ be a minimal ideal in $E(X)$ then;

\begin{itemize}

\item The set $J$ of idempotents of $I$ is non empty.

\item An idempotent $u\in E(X)$ is said to be minimal if $u$ is contained in some minimal ideal $I\subseteq E(X)$.

\item $pv=p$ for all $v\in J$ and $p\in I$.

\item $vI$ is a group with identity $v$ for all $v \in J$.

\item $\lbrace vI: v\in J\rbrace$ is a partition of $I$, and

\item If we set $G=uI$ for some $u\in J$, then $I=\biguplus\lbrace vG: v\in J\rbrace$, where $\biguplus$  denotes  disjoint union.

\item We define an equivalence relation $\sim $ on $J$, the set of idempotents in $E(X)$ as:

\begin{center}
$u\sim v\Leftrightarrow uv=u$ and $vu=v$.
\end{center}
and  so if $u\sim v$ then we say $u$ and $v$ are equivalent. If $I,K\subseteq E(X)$ are minimal ideals in $E(X)$ and $u^{2}=u\in I$ is an idempotent then $\exists$ a unique idempotent $v\in K$ with $uv=v$ and $vu=v$.

\end{itemize}

We note an important result proved by Ellis  \cite{ELLIS} and given as an exercise in \cite{ELL}. We include the proof here since we need these ideas for some of our results later.
\begin{theorem}\label{3.1}
In a flow $(X,T)$, if $I,K\subset E(X)$ are two minimal ideals then $(I,T)$ and $(K,T)$ are isomorphic.
\end{theorem}
\begin{proof}
We know that every minimal ideal of $E(X)$ contains an idempotent. We say that that for idempotents $u$ and $v$, $u \sim v$ if and only if $uv=v$ and $vu=u$.

\textbf{Claim}: For two minimal ideals $I,K$ in $E(X)$ and for an idempotent $u\in I$, there is a unique idempotent $v\in K$ such that $u \sim v$.

For $u\in I$, consider the set $Ku=\lbrace ku: k\in K \rbrace$. Since $u \in I$ and $I$ is an ideal $\Rightarrow ku\in I$. Therefore $Ku\subset I$ and we can see that $Ku$ is also an ideal in $E(X)$. So $Ku=I$.

Now consider $N= \lbrace k\in K: ku=u\rbrace$. Since $Ku=I$ and $u\in I$. So there is a $k\in K$ such that $ku=u \Rightarrow N\neq \emptyset$. Also if $k_{1}, k_{2}\in N$ then $k_{1}k_{2}u=u\Rightarrow k_{1}k_{2}\in N$. So $N^{2}\subset N$. Also as $k\in N\Leftrightarrow ku=u \Leftrightarrow k\in R_{u}^{-1}(u)$. Therefore $N=R_{u}^{-1}(u) \cap K$.

Since $N^{2}\subset N$, let $\Sigma = \lbrace S \subset E(X) :  \emptyset \neq S , S^2 \subset S\rbrace$. Since $N \in \Sigma$, $\Sigma \neq \emptyset$ and by Zorn's lemma there exists a minimal element, $S$ of $\Sigma$ when the latter is ordered by inclusion.

 Let $s \in S$, $Ss = R_{s}(S)$ is a closed subset of $S$ because $R_{s}(S)$ is closed. $SsSs \subset SSs \subset Ss \subset S$, whence $Ss = S$ by the minimality of $S$.

Let $R = \lbrace t \in S : ts = s\rbrace$ and since $Ss=S$, $R\neq \emptyset$ and $R^2 \subset R = R_{s}^{-1}(s)= \overline{R}$. So $R=S$ therefore $s^2 = s$, which means there is a $v\in N$ such that $v^{2}=v$. So $vu=u$. By interchanging the role of $I$ and $K$, there will be a $w\in I$ such that $wv=v$.

Now since $u\in I$ and $u^{2}=u$, so $Iu=I$ as $I$ is minimal ideal. So for any $p\in I$, $p=\alpha u$, $\alpha \in I$. Therefore $pu=\alpha uu=\alpha u=p$. Here $w\in I$. So $w=wu=wvu=v=u \Rightarrow w=u$. So there is a $v\in K$ such that $uv=v$ and $vu=u$.

Suppose there is another $\eta \in K$ such that $u\eta=\eta$ and $\eta u = u$ then $\eta=u \eta=vu \eta=v \eta$. So $\eta \in vK$. Since ${\eta}^{2}= \eta$, $\eta \in \eta K \Rightarrow \eta \in vK \cap \eta K$ but $vK$ and $\eta K$ will be disjoint. Therefore $\eta=v$. Hence there is a unique $v\in K$ such that $uv=v$ and $vu=u$. So claim is true.

Now define $R_{v}:(I,T) \rightarrow (K,T)$ as $R_{v}(p)= pv$.

\textbf{Claim:} $R_{v}$ is an isomorphism.

If $R_{v}(p)= R_{v}(q)$ then $pv=qv \Rightarrow pvu =qvu \Rightarrow pu = qu \Rightarrow p=q$. Therefore $R_{v}$ is injective.

Now let any $\alpha \in K$. Since $K\subset E(X)$ is minimal, $\overline{\mathcal{O}(v)}= K$. So there is a net $\lbrace t_{i} \rbrace$ in $T$ such that $t_{i}v \longrightarrow \alpha \Rightarrow t_{i}uv \longrightarrow \alpha $ because $uv=v$. Now $\lbrace t_{i}\rbrace \in E(X)$, there is a subnet $\lbrace t_{i_{k}} \rbrace$ such that $t_{i_{k}} \rightarrow p$ where $p \in E(X) \Rightarrow t_{i_{k}}u \rightarrow pu$ and $pu\in I$.

$\alpha= \lim\limits_{i\rightarrow \infty} t_{i_{k}}uv=puv$ where $pu\in I$. So $\alpha= R_{v}(pu)$. Therefore $R_{u}$ is surjective. Since right multiplication in $E(X)$ is continuous and homomorphism, $R_{u}$ is an isomorphism.

\end{proof}

We again consider Example \ref{EX 5.1}:

\begin{example} \label{EX 5.5}

Consider the cascade $([0,1],f)$ where $f(x)=x^2$.
Enveloping semigroup $E(X)$ here is just the two point compactification of $\Z$ as $f^{n} \longrightarrow g_{1}$ and $f^{-n} \longrightarrow g_{2}$ as $n \longrightarrow \infty$, where

\centerline{$g_{1}(x)= \left\{
             \begin{array}{ll}
               0, \ &  x\in [0,1);  \\
               1, & x = 1.
             \end{array}
           \right.$
and $g_{2}(x)= \left\{
                 \begin{array}{ll}
                   0, & x = 0; \\
                   1, &  x\in (0,1].
                 \end{array}
               \right.$}

Note that here both $g_1$ and $g_2$ are idempotents. $f(g_1) = g_1$ and $f(g_2) = g_2$. Also $g_1(g_2) = g_2$ and $g_2(g_1) = g_1$.

So this cascade $(X,f)$ has two minimal ideals $I_1 = \{g_1\}$ and $I_2 = \{g_2\}$ each comprising of the singleten an idempotent.

\end{example}

We will be mostly dealing with  the study of cascades or semicascades and so for  $T=\Z$ or $T=\N \cup \lbrace 0 \rbrace$. We prove some of  the above properties for cascades, for which we have not seen any proof in print.
\begin{lemma} Let $\phi:(X,f)\rightarrow (Y,g)$ be a homomorphism then there exists $\Phi:E(X)\rightarrow E(Y)$ s.t. $\Phi(pq)=\Phi(p)\Phi(q)$ and $\phi(px)=\Phi(p)\phi(x)$. \end{lemma}

$\Phi$ is a continuous surjection and so a semigroup homomorphism. We note that
\xymatrixrowsep{12mm}
\xymatrixcolsep{15mm}
\begin{align*}
\xymatrix{
& \beta T  \ar[ld]_{\Phi_X} \ar[dr]^{\Phi_Y} \\
  E(X)   \ar[rr]^{\Phi}
%& \mathcal{F}(V) \ar[ld] \ar[dd]\\
 & & E(Y)   \\
 X \ar[rr]_{\phi}
 %& \mathcal{G}(V) \ar[ld]\\
 & & Y
}
\end{align*}

where $\Phi_{X}:\beta \Z \rightarrow X^X$ s.t. $\Phi_{X}(\beta \Z)=E(X)$ and $\Phi_{Y}:\beta \Z \rightarrow Y^Y$ s.t $\Phi_{Y}(\beta \Z)=E(Y)$.

Since $E(X)= \overline{\lbrace f^{n}: n\in \Z \rbrace}\subset X^X$ and $\phi f = g \phi$, if

\begin{align*}
& f^{n_i}\longrightarrow p\ in\ E(X) \\
i.e. \ & f^{n_i}(x)\longrightarrow p(x)\ \forall\ x\in X\\
i.e. \ & \phi(f^{n_i}(x))\longrightarrow \phi(p(x))\ \forall\ x\in X\\
i.e. \ & g^{n_i}(\phi(x))\longrightarrow \phi(p(x))\ \forall\ x\in X
\end{align*}

\centerline{Define $\Phi: E(X)\rightarrow E(Y)$ as $\phi(px)=\Phi(p)\phi(x)$.}

We first note that $\Phi$ is surjective.

Let \begin{align*}
& q\in E(Y)\\
i.e. \ & g^{n_{i}}\longrightarrow q\ in\ Y^Y\\
i.e. \ & g^{n_{i}}(y)\longrightarrow q(y)\ \forall\ y\in Y
\end{align*}
Since $\phi$ is surjective, $\exists\ x\in X$ s.t $\phi(x)=y\ \forall\ y\in Y$ and so \begin{align*}
& g^{n_{i}}(\phi(x))\longrightarrow q(\phi(x))\ \forall\ x\in X\\
i.e. \ &\phi f^{n_{i}}(x)\longrightarrow q\phi(x)\ \forall\ x\in X
\end{align*}
By passing to subsequence if necessary
\begin{align*}
& f^{n_i}\longrightarrow p\ in\ E(X) \\
i.e. \ & f^{n_i}(x)\longrightarrow p(x)\ \forall\ x\in X\\
i.e. \ & \phi f^{n_i}(x)\longrightarrow \phi p(x)\ \forall\ x\in X
\end{align*}
Thus $\phi(px)=q\phi(x)=\Phi(p)\phi(x)\ \forall\ x\in X \Rightarrow \Phi(p)=q$.

We now prove that $\Phi$ defined thus is continuous. Suppose
\begin{align*}
& q_n \longrightarrow p \ in\ E(X) \\
i.e. \ & q_n(x) \longrightarrow p(x)\  \forall\ x\in X\\
i.e. \ & \phi(q_n(x)) \longrightarrow \phi(p(x))\ in\ Y \ \forall\ x\in X\\
i.e. \ & \Phi(q_n)\phi(x)\longrightarrow \Phi(p)\phi(x)\ in\ Y \ \forall\ x\in X \ \text{or} \ \phi(x)\in Y\\
i.e. \ & \Phi(q_n)\longrightarrow \Phi(p)\ in\ E(Y)
\end{align*}
So $\Phi$ is continuous.

\begin{theorem} Let $(X,f)$ and $(Y,g)$ be cascades for which there exists a homomorphism $\Phi:E(X)\rightarrow E(Y)$.
If $I\subset E(X)$ is a minimal ideal then $\Phi(I)$ is a minimal ideal in $E(Y)$.
\end{theorem}
\begin{proof}
Since $I$ is minimal ideal in $E(X)$, $I$ is minimal subset of the flow $(E(X),T)$. Take any $\alpha\in E(Y)$ and $\Phi(r)\in \Phi(I)$ where $r\in I$. For any $y\in Y$, since $\phi$ is a factor map, there will be an $x\in X$ such that $\phi(x)=y$ then $\alpha \Phi(r) (\phi(x))= \alpha \phi(rx)$.

Now since $\alpha\in E(Y)$ there is a sequence $\lbrace n_i \rbrace$ such that $g^{n_i}\longrightarrow \alpha \Rightarrow g^{n_i}(y)\longrightarrow \alpha(y) \Rightarrow g^{n_i}(\phi(x))\longrightarrow \alpha(\phi(x))$. Since $\phi$ is homomorphism of flows, $g^{n_i}(\phi(x))= \phi(f^{n_i}(x))$. By passing to a subsequence, there will be a $p\in E(X)$ such that $f^{n_i}\longrightarrow p$. Since $\phi$ is continuous, $\phi(f^{n_i}(x))\longrightarrow\phi(p(x))$. So $\alpha(\phi(x))= \phi(p(x))$.

So $\alpha(\phi(rx))= \phi(p(rx))= \Phi(pr)\phi(x)$. So $\alpha \Phi(r) (\phi(x))= \Phi(pr)(\phi(x))$ for any $x\in X$. Therefore $\alpha \Phi(r) = \Phi(pr)$. Since $r\in I$ and $I$ is an ideal in $E(X)$, $pr\in I \Rightarrow \Phi(pr)\in \Phi(I)$. Hence $\alpha \Phi(r)\in \Phi(I)\Rightarrow \Phi(I)$ is an ideal of $E(Y)$. Since $I$ is minimal subset of $E(X)$, $\Phi(I)$ will also be minimal subset of $E(Y)$. Hence $\phi(I)$ is minimal ideal in $E(Y)$.
\end{proof}

\subsection{Minimality and Enveloping Semigroups}

We recall the universality of a minimal subset $M$ of $\beta T$.

\vskip .5cm

  Let $M$ be a minimal subset of $\beta T$,  $((\beta T,T)\cong E(\beta T,T))$ and $(X,T)$ be a minimal flow.
Then there exists an epimorphism $\lambda:M\rightarrow X$.

\vskip .5cm

Suppose $(X,T)$ is a minimal flow then for any $x\in X,\ \overline{Tx}$ is closed invariant set and by minimality of X, $\overline{Tx}=X$.  Hence $(X,T)$ is pointwise almost periodic. We list some results:

\begin{itemize}
\item Let $(X,T)$ be a flow with  $x\in X$ and $I\subseteq E(X)$ be a minimal ideal in $E(X)$  then following are equivalent;

(a) $x$ is almost periodic point of $X$.

(b) $\overline{Tx}=Ix=\lbrace px: p\in I\rbrace$ and

(c) there exist $u^{2}=u\in I$ such that $ux=x$.

\item For the flow $(X,T)$  and $x,y\in X$, the following are equivalent;

 (a) There exists a net $\lbrace t_{i}\rbrace\subseteq T$ with $\lim\limits_i t_{i}x=\lim\limits_i t_{i}y$.

 (b) $\overline{T(x,y)}\cap \Delta\neq \emptyset$ where $\Delta=\lbrace(x,x):x\in X\rbrace$.

  (c) There exists $p\in \beta T$ with $px=py$.

   (d) There exists $r\in E(X)$ such that $rx=ry$.

    (e) there exists a minimal  ideal $I\subseteq E(X)$ with $rx=ry$ $\forall\ r\in I$.

    (f) there exists a minimal  ideal $K\subseteq \beta T$ with $qx=qy$ $\forall\ q\in K$.

 for $x,y\in X$ the pair $(x,y)$ is \textit{proximal} if it satisfies anyone of these equivalent condition.

\item For an idempotent $u\in \beta T$ and $x\in X$,  since $\overline{T}=\beta T$ $\ \exists$ a net $\lbrace t_{i} \rbrace \subseteq T$ such that
$t_{i}\rightarrow u  \Rightarrow t_{i}x  \rightarrow  ux$ and so  $\lim\limits_i t_{i}(x,ux) = \lim\limits_i (t_{i}x, t_{i}ux)=(ux, uux)=(ux, ux)$ i.e. $\lim\limits_i(t_{i}x)=\lim\limits_i t_{i}(ux)$. Hence $x$ and $ux$ are proximal.

\item Let $(X,T)$ be a minimal flow, then;

 (a) $P(X)= \lbrace (x,wx): x\in X, \text{ \textit{w} is minimal idempotent in }  \beta T  \rbrace$.

  (b) $P(X)= \lbrace (x, wx) : x\in X,\text{ \textit{w} is minimal idempotent in } E(X) \rbrace$.

and a flow (not necessarily minimal) is distal if and only if $ux=x$  for all $x\in X$ and for all idempotents $u\in \beta T$.

%\begin{theorem}
%Let $X$ and $Y$ be minimal flows and $\pi: X\rightarrow Y$ be a homomorphism then $\pi(P(X))= P(Y)$ and $\pi(xP(X))= \pi(x)P(Y)\ \forall\ x\in X$
%\end{theorem}
%\begin{theorem}
%\end{theorem}

\item Let $(X,T)$ be a flow and $\Phi: \beta T \rightarrow E(X)$ be the canonical map. The following are equivalent;

 (a) $X$ is distal.

 (b) $e$ is the only idempotent in $E(X)$.

  (c) $E(X)$ is a group.

   (d) $E(X)= \Phi(M)$, where $M$ is the minimal subset of $\beta T$ and

   (e) $E(X)$ is minimal.
\end{itemize}

The following theorem with the sketch of the proof is given in \cite[p. 53, Th. 6]{AUS}. We give the details of the proof here;
\begin{theorem} \label{TH EMF}
Let $(X,T)$ be an equicontinuous minimal flow,  and let $x_{0}\in X$. Let $F=F_{x_0}=\lbrace p\in E(X): px_{0}=x_0\rbrace$, then $F$ is a closed subgroup of $E(X)$. T acts on the space of left cosets $\mathcal{F}=\lbrace qF: q\in E\rbrace$ by $t(qF)=(tq)F\ (t \in T)$ and the flow $(E / F, T)$ is isomorphic with $(X,T)$. If $T$ is Abelian then $(E(X),T) \cong (X,T)$.
\begin{proof}
Since $(X,T)$ is minimal equicontinuous flow and hence distal, $E(X)$ will be a group. Now if $p_1, p_2 \in F$ then $p_{1}(x_0)=x_0$ and $p_{2}(x_0)=x_0$. So, $p_{1}p_{2}(x_0)=p_{1}(x_0)=x_0\ \Rightarrow p_{1}p_{2} \in F$.

Also if $p_{1}(x_0)=x_0\ \Rightarrow\ x_0=p_{1}^{-1}(x_0)\ \Rightarrow p_{1}^{-1}\in F$. Therefore $F$ is subgroup of $E(X)$. Now $\mathcal{F}=\lbrace qF: q\in E\rbrace$ is set of left cosets determined by $F$. For $t_1, t_2 \in T$, $(t_1t_2)(qF)= ((t_1t_2)q)F= (t_1(t_2q))F = t_1(t_2q)F $. Therefore T acts on the space of right cosets $\mathcal{F}=\lbrace qF: q\in E\rbrace$ by $t(qF)=(tq)F$.

Now we define the natural map $\Psi:(E / F, T)\rightarrow (X,T)$ defined by $\Psi(qF)=q(x_0)$.

\textbf{$\Psi$ is a flow homomorphism:} For $t\in T$ and $qF\in \mathcal{F}$, then $\Psi(t(qF))=\Phi((tq)F)=(tq)(x_0)=t(q(x_{0}))=t\Psi(qF)$.

So, $\Psi$ is a flow homomorphism.

\textbf{$\Psi$ is injective:}
If $\Psi(qF)=\Psi(sF)\ \Rightarrow q(x_o) = s(x_0)\ \Rightarrow s^{-1}q(x_0)= x_0\ \Rightarrow s^{-1}q\in F\ \Rightarrow qF=sF$. Which means $\Psi$ is one-one.

\textbf{$\Psi$ is surjective:} For any $x\in X$, since $X$ is minimal. So, $X= \overline{Tx_{0}}$. There exist a net $\lbrace t_{i} \rbrace$ in $T$ such that $x=\lim\limits_i t_i x_0$ and since $t\equiv \pi^{t}$ and $\overline{T}=E$ then $lim_i t_i = r$ for some $r\in E(X)$. Therefore $x=r(x_0)=\Phi(rF)$. So, $\Psi$ is surjective also and hence $(E / F, T)\cong (X,T)$.

Further if $T$ is abelian then for any $q\in E(X)$ and $t\in T$, there is a net $\lbrace s_i \rbrace$ in $T$ such that $\lim\limits_i s_i=q$. So, $qt=\lim\limits_i s_i t= \lim\limits_i t s_i= tq$. Therefore for $p\in F$ and $t\in T,\ p(tx_0)=pt(x_{0})=tp(x_{0})=tx_{0}$. Which means $p$ is the identity map on $Tx_{0}$.

Now for any $y\in X=\overline{Tx_{0}},$ there is a net $\lbrace k_{i} \rbrace$ in $T,\ y=\lim\limits_i k_{i}x_0,$ since $p$ is identity, hence continuous on $Tx_0\ \Rightarrow p(y)=p(\lim\limits_i k_{i} x_{0})=\lim\limits_i p(k_{i} x_{0})=\lim\limits_i k_{i} x_{0} = y$. Therefore $p$ is identity on whole $X\ \Rightarrow F=\lbrace e \rbrace \ \Rightarrow E/F=E$. Hence $(E(X),T) \cong (X,T)$.
\end{proof}
\end{theorem}

%
%[[The minimal subsets of the flow $(\beta T, T)$ (all are isomorphic) coincide with the minimal  ideals of the semigroup $\beta T$. These are universal minimal flows - every minimal flow is a homomorphic image. We fix a universal minimal flow $(M, T)$, and let $J(M)$ denote the
%set of idempotents in M. Then $J(M)$ is non empty; indeed, if $(X, T)$ is a minimal
%flow and $x \in X$, there is a $u \in J(M)$ such that $ux = x$. Now fix $u \in J(M)$, and let $G = uM$. Then $G$ is a group (with respect
%to the semigroup operation on $M$) which can be identified with \emph{the group of
%flow automorphisms} of $(M, T)$ via left multiplication. $G$ can be provided
%with a compact $T_1$ (but not Hausdorff) topology, with respect to which
%multiplication is (separately) continuous, and inversion is continuous.]]

\subsection{Circle Operator and Induced Systems}

Again, for any topological group $T$, $T$ acts on the induced system $2^X$ as $tA=\lbrace ta: a\in A\rbrace$ for any $t\in T$ or we can write $\pi^t(A)=\lbrace \pi^t(a)=ta :a\in A\rbrace$ and the \emph{circle operator} is an action of $\beta T$ on the induced system $2^{X}$. Let $(X,T)$ be a
flow(not necessarily a cascade), $\emptyset\neq A= \overline{A}\subset X$ and
 $p\in \beta T$. We define the circle operation of $\beta T$ on $2^X$ by

 \centerline{$p\circ A= \lbrace x\in X: \exists$ net $ \lbrace a_{i}\rbrace \subset A$,  and  $\lbrace t_{i} \rbrace \subset T$ with $t_{i}\rightarrow p$  and $t_{i}a_{i}\rightarrow x\rbrace$.}

\vskip .5cm

Let $(X,T)$ be a flow, $\emptyset\neq A\subset X$ and $t\in T$ and $p,q\in \beta T$ then;

\begin{itemize}

\item[(a)] $t \circ A = t \overline{A}$, considered as a subset of $2^X$.

\item[(b)] $p \circ A = \bigcap \lbrace\overline{(N \cap T)A}: \text{N is a neighborhood of $p$ in}\ \beta T \rbrace$, again considered as a subset of $2^X$.

\item[(c)] $pA\subset p\circ A=\overline{p\circ A}$.

\item[(d)] $p\circ\overline{A} = p\circ A$.

\item[(e)]  $p \circ (q \circ A)=(pq)\circ A$.

\end{itemize}

So, these properties shows that the $\beta T$ acts on $2^{X}$ with circle operator. Also if $\pi: (X, T)\rightarrow (Y,T)$ is a flow homomorphism then

\begin{itemize}

\item[(a)] $\pi(p\circ A)=p\circ \pi(A)$ for all $\emptyset\neq A\subset X$ and all $p\in \beta T$.

\item[(b)] $p\circ \pi^{-1}(y) \subset \pi^{-1}(py)$ for all $y\in Y$ and $p\in \beta T$ and if $\pi$ is open map and $(Y,T)$ is minimal then  $p\circ \pi^{-1}(y)= \pi^{-1}(py)$ for all $p\in M$, $M$ is minimal ideal of $\beta T$ and $ \pi^{-1}(p\circ B)= p\circ \pi^{-1}(B)$ for all $B\in 2^{Y}$ and $p\in M$.
\end{itemize}

Identify $t\in T$ with the map $t\longrightarrow tx$. So without loss of generality $T$ can be considered as a subset of $X^X$ and the enveloping semigroup $E(X)$ is its closure in $X^X$. In fact as shown by Ellis,  there exists a continuous map $\Psi: \beta T\rightarrow X^X$ which is an extension of $\psi : T \longrightarrow X^X$ where $\Psi(\beta T) = E(X)$.

Since $\beta T$ acts on both $X$ and $2^X$, the action of $p\in \beta T$ on $X$ is given as $x\longrightarrow px$ where $px=\lim t_{i}x$ whereas the action of $p\in \beta T$ on $2^X$ is given as $A\longrightarrow p\circ A$ where $p\circ A = \lbrace x\in X: t_{i}\longrightarrow p, \lbrace a_i \rbrace \subset A, t_{i}a_{i}\longrightarrow x\rbrace= \lim t_{i}A$ in $2^X$ where $t_{i}\longrightarrow p$ in $\beta T$.

\vskip .5cm

\emph{In the  study of induced systems, we shall see that computing $E(2^X)$ and looking into $pA\in 2^X$ for $p\in E(2^X)$ is relatively easy than looking into $p\circ A$ for $p\in \beta T$.}

\section{Periodicity and Recurrence in Enveloping Semigroups}

In this section, we will be  working only with cascades $(X,f)$.

\vskip .5cm

\subsection{Recurrence in Enveloping Semigroups}

Before we get into details with recurrence, we look into this simple observation:

\begin{proposition} \label{P 6.2}
Let $(X,f)$ be a cascade. If $e$ is an accumulation point of $E(X)$  then $\mathcal{R}_f(X)=X,$ where $\mathcal{R}_f(X)$ is the set of recurrent points in $X$.
\end{proposition}
\begin{proof}
Let $e$ be an accumulation point  of $E(X)$.

There is a sequence $\lbrace n_{k} \rbrace$ in $\Z$ such that $f^{n_k}\longrightarrow e$. Since $E(X)$ has topology of pointwise convergent. So, $f^{n_k}(x)\longrightarrow x$ for all $x\in X$.

Hence $x\in \mathcal{R}_f(X)\Rightarrow\ X=\mathcal{R}_f(X)$.
\end{proof}

\begin{remark} Note that this gives us a bit more than the recurrence of all $x \in X$. \end{remark}

%\begin{proposition}
%If $(X,f)$ is a cascade and $\omega(x)=X \ \forall\ x\in X$ then $f$ is surjective on $E(X)$.
%\end{proposition}
%\begin{proof}
%Since  $\omega(x)=X \ \forall\ x\in X$ then for any $x_0\in X,\ f(x_0) \in \omega(x_0),$ there is a sequence $\lbrace n_k \rbrace$ in $\Z$ such that $f^{n_k}(x_0)\longrightarrow x_0$. Since this is true for every $x\in X$ and $E(X)$ has topology of pointwise convergence.
%
%Thus $\ f$ will be a limit point of $E(X)$. Hence $f$ is surjective on $E(X)$.
%\end{proof}

Every minimal ideals contains a minimal idempotent. Minimal idempotents play an important role in the algebraic theory of topological dynamics.

But they are also important entities in the topological dynamics of enveloping semigroups.

\begin{theorem} \label{4.01}
Let $(X,f)$ be a cascade and $u\in E(X)$ be an idempotent then $u$ is a recurrent point in $(E(X),f)$.
\end{theorem}
\begin{proof}
Since $u\in E(X)$, there is a sequence $\lbrace n_i \rbrace$ such that $f^{n_i}\longrightarrow u\Rightarrow f^{n_i}u\longrightarrow uu=u$ because $u$ is an idempotent. So $f^{n_i}u\longrightarrow u$ shows that $u$ is a recurrent point in $(E(X),f)$.
\end{proof}

Since every point in a minimal ideal is almost periodic,

\begin{corollary}
Every element in the minimal ideal of $E(X)$ is recurrent.
\end{corollary}

\subsection{Periodic points for Enveloping Semigroups}

$E(X)$ can have periodic points. We first look into some examples.

\vskip .5cm

We again consider Example \ref{EX 5.1}:

\begin{example}

Consider the cascade $([0,1],f)$ where $f(x)=x^2$. Enveloping semigroup $E(X)$ here is just the two point compactification of $\Z$ as $f^{n} \longrightarrow g_{1}$ and $f^{-n} \longrightarrow g_{2}$ as $n \longrightarrow \infty$, where

\centerline{$g_{1}(x)= \left\{
             \begin{array}{ll}
               0, \ &  x\in [0,1);  \\
               1, & x = 1.
             \end{array}
           \right.$
and $g_{2}(x)= \left\{
                 \begin{array}{ll}
                   0, & x = 0; \\
                   1, &  x\in (0,1].
                 \end{array}
               \right.$}

Note that  $f(g_1) = g_1$ and $f(g_2) = g_2$. Thus $g_1$ and $g_2$ are fixed points in $E(X)$.
\end{example}

\begin{example} \label{EX 4.2}
$X = [-1,1]$ and $f:X\rightarrow X$ defined as $f(x)= -x^{3}$ then $f^{2n}\longrightarrow g$, $f^{2n+1}\longrightarrow h$, $f^{-2n}\longrightarrow m$ and $f^{-2n-1}\longrightarrow n$ where

\centerline{$g(x)= \left\{
             \begin{array}{ll}
               0, \ &  x\in (-1,1)  \\
               1, & x = 1 \\
             -1, & x =-1.
             \end{array}
           \right.$
and $h(x)= \left\{
                 \begin{array}{ll}
                   0, & x \in (-1,1) \\
                   1, &  x = -1\\
                  -1, &  x=1.
                 \end{array}
               \right.$}

\centerline{$m(x)= \left\{
             \begin{array}{ll}
               0, \ &  x=0  \\
               1, & x \in (0,1] \\
              -1, & x \in [-1,0].
             \end{array}
           \right.$
and $n(x)= \left\{
                 \begin{array}{ll}
                    0, \ &  x=0  \\
                    1, & x \in [-1,0) \\
                   -1, & x \in (0,1].
                 \end{array}
               \right.$}

Here $mm=m$, $nn=m$, $mg=g$, $gm=m$, $ng=h$, $gn=n$, $hh=g$, $hn=m$, $mh=h$,  $hhh=gh=h$, $hg=h$, $gg=g$, $fg=h$, $fh=g$, $fm=n$, $fn=m$.

So $E(X)= \lbrace f^n: n\in \Z\rbrace \cup \lbrace g,h,m,n \rbrace$.
\end{example}

\begin{remark} We note that in Theorem \ref{4.01} we showed that idempotents are recurrent points in $E(X)$.  As seen in the above example, $h^3 = h$ and $h$ is not an idempotent but still $h$ is recurrent. As discussed earlier every point in the minimal ideal is recurrent in $E(X)$, but such a point need not be an idempotent.\end{remark}

Enveloping semigroups can admit periodic points of any period, in case of cascades. We see an example:

\begin{example} Let $n \in \N$.

Let $X_{n} = \lbrace n \rbrace \times \Z^{\ast} \times \lbrace 1,2,\ldots, n\rbrace$ where $\Z^{\ast}= \Z \cup \lbrace \infty \rbrace$ is one point compactification of $\Z$ and $f_{n}: X_{n} \rightarrow X_{n}$ as $f_{n}(n, k, l )=(n, k+1, (l+1) \mod n )$. In the cascade $(X_n, f_n)$, $f_{n}^{ni}(n, k, l)= (n, k+ni, l) \ \forall i$. So for each $(n, k, l )\in X_n$, $f_{n}^{ni}(n, k, l)\longrightarrow (n, \infty, l)$. So $E(X_n)= \lbrace f_{n}^{k}: k\in \Z\rbrace \cup \lbrace p_n, f_np_n, \ldots, f_n^{n-1}p_n \rbrace$ where $p_{n}(n, k, l)= (n, \infty, l)$. Also $f_{n}^{n}p_{n}(n, k, l)= f_{n}^{n}(n, \infty, l)=(n, \infty, l)\Rightarrow f_{n}^{n}p_{n}= p_{n}$. Therefore $p_{n}$ is periodic point of period $n$. Also $p_{n}p_{n}=p_{n}$. Therefore for each $n\in \N$, there is a cascade $(X_{n}, f_{n})$ such that $E(X_{n})$ contains a periodic point of period $n$ which is an idempotent.
\end{example}

\begin{example}
Let $X_{n}$ as in example above and $X= \bigcup\limits_{i=1}^{n} X_{i}$ and $f:X\rightarrow X$ defined as $f(n, k, l )=f_{n}(n, k, l )$. So $f_{n}^{ni}(n, k, l)= (n, k+ni, l)$ for each $n$ and each $i$. Now consider lcm$(1,2,\ldots,n)= s$ then $f^{si}(n, k, l )= (n, k+si, l)$. Therefore $f^{si}(n, k, l )\longrightarrow (n, \infty, l)$. So $E(X)= \lbrace f^{k}: k\in \Z\rbrace \cup \lbrace p \rbrace$ where $p(n, k, l)= (n, \infty, l)$. Also $pp=p$ and $f^{s}(p)=p$.
\end{example}

We can observe that $E(\bigcup\limits_{i=1}^{n} X_{i}) \neq E(X_1)\cup E(X_2)\cup \ldots \cup E(X_n)$ but for any $1\leq i \leq n$, there is a factor map $h:(E(X),f)\rightarrow (E(X_i), f_{i})$ as $h(f^k)=f_{i}^{k}$ and $h(p)= p_{i}$.

\begin{proposition}
Let $(X_i, f_i), \ i = 1, \ldots k$ be cascades. Let $X = \bigcup \limits_{i=1}^k X_i$ with $f: X \to X$ defined as $f(x) = f_i(x)$ for $x \in X_i$ (assuming that $f_i(x) = f_j(x)$ when $x \in X_i \cap X_j$). Then $(X,f)$ is a cascade and $E(X)$ factors onto $E(X_i), \ i = 1, \ldots, k$.

Let $\overline{X} = \prod \limits_{i=1}^k X_i$ with $\overline{f}: X \to X$ defined as $\overline{f}(\prod \limits_{i=1}^k x_i) = \prod \limits_{i=1}^k f_i(x_i)$ for $x_i \in X_i$. Then $(\overline{X},\overline{f})$ is a cascade and $E(\overline{X})$ factors onto $E(X_i), \ i = 1, \ldots, k$.
\end{proposition}

\begin{theorem} \label{4.1}
Let $(X,f)$ be a cascade. If $p\in E(X)$ is a periodic point then $\mathcal{O}(p)$ is a minimal ideal in $E(X)$.
\end{theorem}
\begin{proof}
Let $I= \mathcal{O}(p)$, we note that $|I|< \infty$ because $p$ is a periodic point of the cascade $(E(X),f)$.

We prove that for any $q\in E(X)$, $qI=I$. Since $I$ is a periodic orbit, it will be enough to show that $qp\in I$. Let $q\in E(X)$, then there is a sequence $\lbrace n_i \rbrace$ such that $f^{n_i}\longrightarrow q$ in $E(X)$. Thus $f^{n_i}p\longrightarrow qp$ and since $I=\mathcal{O}(p)$ is finite, we have $qp\in I$ making $I$ a minimal ideal.
\end{proof}
\begin{corollary}
Let $(X,f)$ be a cascade. If there is a periodic point $p\in E(X)$ with period $n$ then all periodic points in $E(X)$ have period dividing $n$.
\end{corollary}
\begin{proof}
By Theorem \ref{4.1}, $\mathcal{O}(p)$ is a minimal ideal for each periodic point $p\in (E(X),f)$ and by Theorem \ref{3.1} all minimal ideals are isomorphic. Hence all periodic point will have the same period.
\end{proof}

\begin{proposition}
In a cascade $(X,f)$ with $E(X)$ infinite,  there are no surjective periodic points in $E(X)$.
\end{proposition}
\begin{proof}
Suppose there is a surjective periodic point $p\in E(X)$ with period $n$ then $f^{n}p=p$ and since $p$ is surjective,  $f^{n}(x)=x$ for each $x\in X$. So $f^{n}= I_{X}$, where $I_{X}$ is the identity map on $X$. Hence $E(X)$ contains at most $n$ number of elements and hence finite, which contradicts that $E(X)$ is infinite. So there is no surjective periodic point.
\end{proof}

\begin{theorem} For a cascade $(X,f)$, if its enveloping semigroup $E(X)$ has periodic points then these periodic points are finite, and hence the minimal ideals in $E(X)$ are finite.

In particular, if the period of these periodic points is $n$ then there can be atmost $2n$ periodic points.
\end{theorem}
\begin{proof}
We have already seen that all periodic points will have the same period say $n$. Suppose there are infinitely many periodic points $p_{1}, p_{2}, \ldots, p_{k},\ldots$. For every $p_{l}$, there is a sequence $\lbrace f^{m^{l}_{i}} \rbrace$ such that $f^{m^{l}_{i}}\longrightarrow p_{l}$. Since $\emph{Per}(p_{l})=n$, $f^{m^{l}_{i}+kn} \longrightarrow p_{l}$ for all $k\in \N$.

Consider the $n+1$ points $p_1, \ldots, p_{n+1}$ and let open $U_l \ni p_l$ such that $U_l \cap U_j = \emptyset$ for $l \neq j$, and suppose that ${m^{l}_{i}} > 0$. Since $f^{m^{l}_{i}}\longrightarrow p_{l}$, there is an $N_l \in \N$ such that $f^{m^{l}_{i}}\in U_{l}$ for $m^l_i \geq N_l$.

Let $N \ggg N_l, \ \forall l, \ 1 \leq l \leq n+1$, such that $f^N \in U_1$. Then $f^{N+kn} \in U_1, \ \forall k > 0$. Then there are $n-1$ choices among $N+1, \ldots, N+(n-1)$ to be uniquely contained in the $n$ open sets $U_2, \ldots, U_{n+1}$. This contradicts that $f^{m^{l}_{i}+kn} \longrightarrow p_{l}$ for all $k\in \N$, $1 \leq l \leq n+1$.

If ${m^{l}_{i}} < 0$, then there is an $N_l \in \N$ such that $f^{m^{l}_{i}}\in U_{l}$ for $m^l_i \leq -N_l$.

In all possibilities $|Per(f)| \leq 2n$.

\end{proof}

\section{Finite Minimal Ideals in Enveloping Semigroup and Proximal Relations}

Since $(X,T)$ is a flow,  $(E(X),T)$ is also be a flow and we can talk about the proximal relation in $(E(X),T)$ also. The following proposition tells about the relation between $E(X)$ and $E(E(X))$ in some given conditions;

\begin{proposition}\cite{ELL}\label{5.2}
For a point transitive flow $(X,T)$, $E(E(X))\cong E(X)$.
\end{proposition}

\begin{corollary} For any flow $(X,T)$, $(E(X),T)$ is always point transitive and so $E(E(E(X))) \cong E(E(X))$. \end{corollary}

\begin{theorem} \label{TH 7.1}
Suppose $(X,T)$ is point transitive flow then $(X,T)$ is distal if and only if $(E(X),T)$ is distal.
\end{theorem}
\begin{proof}
Since $(X,T)$ is point transitive then by proposition \ref{5.2}, $E(E(X))\cong E(X)$.

Now if $(X,T)$ is distal then $P(X)= \Delta$. Let $(p,q)\in P(E(X))$ then $\alpha p = \alpha q$ for all $\alpha \in I$, where $I\subset E(E(X))\cong E(X)$ is minimal ideal. Which gives $\alpha p(x)=\alpha q(x)$ for each $x\in X \Rightarrow (p(x),q(x))\in P(X)\Rightarrow p(x)=q(x)$. Since this is true for each $x\in X$, $p=q$. Therefore $P(E(X))=\Delta$ and hence $(E(X),T)$ is distal.

Conversely suppose $(E(X),T)$ is distal flow then $E(E(X))$ will be a group. So $E(X)\cong E(E(X))$ is a group, which gives that $(X,T)$ is distal.
\end{proof}

\begin{remark} We note that if  $(X,T)$ is  not distal $E(X)$ contains a non-trivial idempotent $u$, and
the pair $(u,e)$ is proximal. \end{remark}

\begin{remark}
We note that if the   flow $(X,T)$ is not transitive then also we can have $E(E(X))\cong E(X)$. And in that case if $(X,T)$ is distal(not distal), then $(E(X),T)$ is also distal(not distal). We recall Example \ref{EX 5.1} here:

\vskip .5cm

Consider the cascade $([0,1],f)$ where $f(x)=x^2$.
Enveloping semigroup here is just the two point compactification of $\Z$ as $f^{n} \longrightarrow g_{1}$ and $f^{-n} \longrightarrow g_{2}$ as $n \longrightarrow \infty$, where

$g_{1}(x)= \left\{
                         \begin{array}{ll}
                           0, & \hbox{$x\in [0,1)$;} \\
                           1, & \hbox{$x = 1$.}
                         \end{array}
                       \right.$
and \   \
$g_{2}(x)= \left\{
                 \begin{array}{ll}
                   0, & x = 0; \\
                   1, &  x\in (0,1].
                    \end{array}
                 \right.$

Note that  $f(g_1) = g_1$ and $f(g_2) = g_2$. Thus $g_1$ and $g_2$ are fixed points in $E(X)$.

Now we determine $E(E(X))$ of $(X,f)$. For that we need to look into the limit functions of $\lbrace f^n : n\in \Z\rbrace$ in $E(X)^{E(X)}$ with the topology of pointwise convergent. As $f^{n} \longrightarrow g_{1}$ and $f^{-n} \longrightarrow g_{2}$ in $X^X$. So $f^{n} \longrightarrow p$ in $E(X)^{E(X)} \Leftrightarrow f^{n}(\alpha) \longrightarrow p(\alpha)$ for each $\alpha\in E(X)$.

We can observe that $f^{n} \longrightarrow h$ and $f^{-n} \longrightarrow k$ in $E(X)^{E(X)}$, where $h,k: E(X)\rightarrow E(X)$ are defined as $h(f^n)=g_1$, $h(g_1)=g_1$, $h(g_2)=g_2$ and $k(f^n)=g_2$, $k(g_1)=g_1$, $k(g_2)=g_2$.

 For any $k\in \Z$, $\lim\limits_{n\rightarrow \infty} f^{n}(f^k) = g_1(f^k)= g_1 = h(f^k)$.

Also $\lim\limits_{n\rightarrow \infty} f^{n}(g_1) = g_1(g_1)=g_1=h(g_1)$ and $\lim\limits_{n\rightarrow \infty} f^{n}(g_2) = g_1(g_2)=g_2=h(g_2)$. So by topology of pointwise convergent of $E(X)^{E(X)}$, $f^{n} \longrightarrow h$.

In the same way $k\in \Z$, $\lim\limits_{n\rightarrow \infty} f^{-n}(f^k)=g_2(f^k)=g_2=k(f^k)$ in $E(X)^{E(X)}$.

Also $\lim\limits_{n\rightarrow \infty} f^{-n}(g_1) = g_2(g_1)=g_1=k(g_1)$ and $\lim\limits_{n\rightarrow \infty} f^{-n}(g_2) = g_2(g_2)=g_2=k(g_2)$. So by topology of pointwise convergent of $E(X)^{E(X)}$, $f^{-n} \longrightarrow k$ in $E(X)^{E(X)}$.
 Therefore $E(E(X))= \lbrace f^n:n\in \Z\rbrace \cup \lbrace h,k \rbrace$ which is isomorphic to $E(X)$.

\end{remark}

\begin{remark} We can also have the situation that $(X,T)$ is not a point transitive flow,  yet $(X,T)$  and  $(E(X),T)$ are both distal.

For example consider $X = \{(r, \theta): r = 1, 2$ and $ 0 \leq \theta \leq 2\pi \}$ and let  $f: X \to X$ be defined as $f(r, \theta) = (r, \theta + 2 \pi \alpha  \mod(2 \pi))$, where $\alpha$ is irrational. Then the enveloping semigroup $E(X)$ is isomorphic to an irrational rotation. And here $(X,f)$  and  $(E(X),f)$ are both distal.

\end{remark}

\begin{theorem}
For a point transitive flow $(X,T)$, if $P(X)$ is an equivalence relation then $P(E(X))$ is also an equivalence relation.
\end{theorem}
\begin{proof}
Since we know that proximal relation is always reflexive and symmetric, so we need to check only transitivity. Suppose $(p,q),(q,r)\in P(E(X))$ then $\alpha p = \alpha q$ and $\beta q = \beta r$ for all $\alpha \in I, \beta \in J$, where $I,J\subset E(E(X))\cong E(X)$ are minimal ideals. Now for each $x\in X$, $\alpha p(x) = \alpha q(x)$ and $\beta q(x) = \beta r(x)$ for all $\alpha \in I, \beta \in J$, which means $(p(x),q(x)),(q(x),r(x))\in P(X)$ and since $P(X)$ is an equivalence relation, $(p(x), r(x))\in P(X) \Rightarrow \gamma p(x)=\gamma r(x)$ for all $\gamma \in K$, where $K \subset E(X)\cong E(E(X))$ is minimal ideal. Since $\gamma p(x)=\gamma r(x)$ is true for all $x\in X$, so $\gamma p = \gamma r$ for all $\gamma \in K \Rightarrow (p,r)\in P(E(X))$. Therefore $P(E(X))$ is transitive and hence equivalence relation.
\end{proof}

We recall the below from \cite{EL}:

\begin{theorem} \cite{EL} \label{er}
For a flow $(X,T)$, the proximal relation is an equivalence relation if and only if $E(X)$ contains a unique minimal ideal.
\end{theorem}

\begin{example}
Let $X=\lbrace (r,\theta):r\in \lbrace 1,2 \rbrace, \theta\in \R  \rbrace$. Define $f:X\rightarrow X$ as $f(r,\theta)=(r, \theta+\alpha)$ where $\alpha$ is an irrational multiple of $2\pi$. We note that $f(X)=X$ and represents the irrational rotation on both the circles comprising $X$.

Here $E(X)$ is isomorphic to an irrational rotation and so has a unique minimal ideal. Also the proximal relation $P(X)=\Delta$ and vacuously  an equivalence relation.
\end{example}

\begin{example} \label{EX 7.2}

Let $X= \{ x^n : x^n = (x^n_i)$ such that $x^n_i = 1$ if $i = n$ and $x^n_i = 0$ otherwise, $n \in \Z \} \cup \{\overline{0}\} \ \subset \emph{2}^{\mathbb{Z}}$, where $\overline{0}$ is the  sequence of all $0$s. We observe that $X$ is closed and invariant under the right shift operator $\sigma$, and so $(X, \sigma)$ is a subshift of the $\emph{2}$-shift. Note that $(X, \sigma)$ is not minimal and contains $\{\overline{0}\}$ as the unique minimal subset.

For any $x \in X$ observe that $\sigma^k(x) \longrightarrow \overline{0}$ and $\sigma^{-k}(x) \longrightarrow \overline{0}$ as $k \longrightarrow \infty$. Thus if $p: X \to X$ is the constant map $p(x) = \overline{0}$, then $\sigma^k \longrightarrow p$ and $\sigma^{-k} \longrightarrow p$ pointwise as $k \longrightarrow \infty$. So in this case $E(X) = \{ \sigma^k: k \in \Z\} \cup \{p\}$ and so $I = \{p\}$ is the unique minimal ideal in $E(X)$.

Also the proximal relation $P(X)=\Delta$ and vacuously  an equivalence relation.
\end{example}

Both our examples above have a unique minimal ideal and hence the proximal relation is an equivalence relation. We note that the example above is weakly almost periodic (WAP). We have some interesting observations for this.

\vskip .5cm

\begin{proposition}\cite{ELN}\label{5.1}
For a weakly almost periodic (WAP) flow $(X,T)$, if $I \subset E(X)$ is a minimal ideal then $(I,T)$ is an equicontinuous flow.
\end{proposition}

\begin{proposition} \label{thm}
For a weakly almost periodic (WAP) flow  $(X,T)$, $E(X)$ has a unique minimal ideal which has a unique idempotent.
\end{proposition}
%\begin{proof}
%Since $(X,T)$ is WAP, each member of $E(X)$ is continuous. So both right multiplication and left multiplication on $E(X)$ are continuous.
%
%We first prove that  $E(X)$ has a unique minimal  ideal.
%
%Let $I$ be a minimal  ideal in $E(X)$ and let there be two idempotents $u,v\in I$. Then $Iu=Iv=I$. Let $p\in I$ be any element then $p=ru$ and $p=sv$ where $r,s \in I$. So $pu=ruu=ru=p=sv=svv=pv$ which means $u$ and $v$ are proximal to each other. By proposition \ref{5.1}, $(I,T)$ is an equicontinuous flow. Since $(I,T)$ is minimal also, it is a distal flow. So there are no proximal pairs in $I$. Therefore there is unique idempotent say $u$ in $I$.
%
%If there is another minimal ideal $J$ in $E(X)$ then by the same argument, $J$ contains a unique idempotent say $v$.
%
%Also $Iu = I$ and $Jv = J$ are groups with identity $u$ and $v$ respectively. Now consider the set $JI=\lbrace qp:  q\in J, p\in I \rbrace$ then $JI\subset I\cap J$ since both $I$ and $J$ are ideals. Now $vu \in JI \subset J$ and $J$ is group with identity $v$, and so there is an inverse $q$ of $vu$ in $J$. So $q(vu)=v$. Thus $vu=qvuu=qvu^{2}=qvu=v$. So $v=vu\in JI \subset I \Rightarrow v\in I$. Since $I$ contains the unique idempotent $u$, $u=v$. Since every group has a unique identity, $I=J$.
%
%Thus by Theorem \ref{er}, $P(X)$ is an equivalence relation.
%\end{proof}

We have the proximal relation as an equivalence relation also for weakly almost periodic flows. We look into the results from both \cite{EL, ELN}, and conclude.

\begin{proposition}
For a weakly almost periodic (WAP) flow  $(X,T)$, $P(X)$ is an equivalence relation
\end{proposition}

\vskip .5cm

We note that usually $P(X)$ is not an equivalence relation, and hence $E(X)$ need not contain a unique minimal ideal.

\vskip .5cm

We recall an important class of flows here called \emph{PI flows}, studied first by Glasner \cite{SGP}. \emph{PI flows} are basically minimal flows modulo a proximal extension that are obtained successively from the trivial flow using equicontinuous and proximal extensions. For more on PI flows we recommend \cite{AUS, SGP}.

Ellis, Glasner and Shapiro \cite{EGS1} prove that for a minimal $(X,T)$ if $E(X)$ contains only finitely many minimal ideals then $(X,T)$ is a PI flow. An improvement to this result was given by McMahon \cite{M}, who showed that in a  minimal $(X,T)$ if $E(X)$ contains countably many minimal ideals then $(X,T)$ is a PI flow.
Recently, Glasner \& Glasner \cite{GG} improve McMahon’s result and show that a minimal $(X,T)$
whose enveloping semigroup $E(X)$ contains less than $2^{2^{\omega}}$
 minimal  ideals is PI. They also give an example to show that the converse is not true.

\vskip .5cm

We note that for a non minimal flow or cascade we cannot have any version similar to PI. Since minimal ideals in the semigroup $E(X)$ are actually minimal subsets of the flow $(E(X),T)$, it will be interesting to know the dynamical synonyms for $E(X)$ to have just finitely many minimal ideals or countably many minimal ideals in the non-minimal case.

\vskip .5cm

We look into an example with finitely many minimal ideals;

\begin{example}
Let $X=\lbrace (r,\theta):1 \leq r \leq 2, \theta\in [0,2\pi)  \rbrace$ and $h:X\rightarrow X$ be defined as $h(r,\theta)=(1+(r-1)^2, \theta+2\pi\alpha)$ where $\alpha$ is an irrational constant. We note that $h$ is a homeomorphism with $h^{-1}(r, \theta)=(1+ (r-1)^{\frac{1}{2}}, \theta - 2\pi \alpha)$.

Consider the cascade $(X,h)$. We note that

$(1,\theta) \overset{\text{h}}{\longrightarrow} (1, \theta+2\pi\alpha) \overset{\text{h}}{\longrightarrow} (1, \theta+4\pi\alpha)\overset{\text{h}}{\longrightarrow}\ldots  \overset{\text{h}}{\longrightarrow} (1, \phi)\overset{\text{h}}{\longrightarrow}\ldots$

$(2,\theta) \overset{\text{h}}{\longrightarrow} (2, \theta+2\pi\alpha) \overset{\text{h}}{\longrightarrow} (2, \theta+4\pi\alpha)\overset{\text{h}}{\longrightarrow}\ldots  \overset{\text{h}}{\longrightarrow} (2, \phi)\overset{\text{h}}{\longrightarrow}\ldots$

$(\frac{3}{2},\theta) \overset{\text{h}}{\longrightarrow} (\frac{5}{4}, \theta+2\pi\alpha) \overset{\text{h}}{\longrightarrow} (\frac{17}{6}, \theta+4\pi\alpha)\overset{\text{h}}{\longrightarrow}\ldots  \overset{\text{h}}{\longrightarrow} (1, \phi)\overset{\text{h}}{\longrightarrow}\ldots$

$\ldots\ldots\ldots\ldots\ldots\ldots\ldots\ldots\ldots\ldots\ldots\ldots\ldots\ldots\ldots\ldots\ldots\ldots\ldots\ldots\ldots\ldots$

Also

$(1,\theta) \overset{h^{-1}}{\longrightarrow} (1, \theta-2\pi\alpha)  \overset{h^{-1}}{\longrightarrow} (1, \theta-4\pi\alpha) \overset{h^{-1}}{\longrightarrow}\ldots   \overset{h^{-1}}{\longrightarrow} (1, \xi) \overset{h^{-1}}{\longrightarrow}\ldots$

$(2,\theta) \overset{h^{-1}}{\longrightarrow} (1, \theta-2\pi\alpha)  \overset{h^{-1}}{\longrightarrow} (2, \theta-4\pi\alpha) \overset{h^{-1}}{\longrightarrow}\ldots   \overset{h^{-1}}{\longrightarrow} (1, \xi) \overset{h^{-1}}{\longrightarrow}\ldots$

$(\frac{3}{2},\theta) \overset{h^{-1}}{\longrightarrow} (1+ (\frac{3}{2})^{\frac{1}{2}}, \theta-2\pi\alpha)  \overset{h^{-1}}{\longrightarrow} (1+(1+(1-\frac{1}{2^{\frac{1}{2}}})^{\frac{1}{2}}), \theta-4\pi\alpha) \overset{h^{-1}}{\longrightarrow}\ldots   \overset{h^{-1}}{\longrightarrow} (2, \xi) \overset{h^{-1}}{\longrightarrow}\ldots$

$\ldots\ldots\ldots\ldots\ldots\ldots\ldots\ldots\ldots\ldots\ldots\ldots\ldots\ldots\ldots\ldots\ldots\ldots\ldots\ldots\ldots\ldots$

Let $E(X)$ be the enveloping semigroup for $(X,h)$ then $E^{\ast}=E(X)\setminus\lbrace h^{n}: n\in \Z \rbrace= \lbrace h_{1\theta}: 0 \leq \theta \leq 2\pi \rbrace \cup \lbrace h_{2\theta}: 0 \leq \theta \leq 2\pi \rbrace$

where \begin{center}
$h_{1\theta}(r,\phi)= \left\{
             \begin{array}{ll}
               (1, \theta+\phi), \ &  1\leq r <2, 0\leq \phi < 2\pi  \\
               (2,\theta+\phi), & r = 2, 0\leq \phi < 2\pi.
             \end{array}
           \right.$
           \end{center}
 and
 \begin{center}
 $h_{2\theta}(r, \phi)= \left\{
             \begin{array}{ll}
               (1, \theta-\phi), \ & r=1 , 0\leq \phi < 2\pi  \\
               (2,\theta-\phi), & 1< r \leq 2, 0\leq \phi < 2\pi.
             \end{array}
           \right.$
 \end{center}

 We observe that $h^n\longrightarrow h_{1\theta}$ for some $0\leq \theta \leq 2\pi$ with $h_{1\theta}$ having a dense orbit in the set of all $h_{1\alpha}$ for $0\leq \alpha \leq 2\pi$ and $h^{-n}\longrightarrow h_{2\lambda}$ for some $0\leq \lambda \leq 2\pi$ with $h_{2\lambda}$ having a dense orbit in the set of all $h_{2\gamma}$ for $0\leq \gamma \leq 2\pi$ in $X^X$.

Note that

\begin{align*}
hh_{1\theta} &= h_{1(\theta+ 2\pi \alpha)}\\
hh_{2\theta} &= h_{2(\theta- 2\pi \alpha)}\\
h_{1\theta}h_{1\theta} &= h_{1\theta}\\
h_{2\theta}h_{2\theta} &= h_{2\theta}\\
h_{2\theta}h_{1\theta} &= h_{1\theta}\\
h_{1\theta}h_{2\theta} &= h_{2\theta}\\
h_{1\theta}h_{1\phi} &= h_{1(\theta+\phi)}\\
h_{2\theta}h_{2\phi} &= h_{2(\theta+\phi)}\\
h_{1\theta}h_{2\phi} &= h_{2(\theta+\phi)}\\
h_{2\theta}h_{1\phi} &= h_{1(\theta+\phi)}\\
\end{align*}

Let \begin{center}

$I_{1}= \lbrace h_{1\theta}: 0\leq \theta \leq 2\pi\rbrace$\\
$I_{2}= \lbrace h_{2\theta}: 0\leq \theta \leq 2\pi\rbrace$
\end{center}
Thus $E(X)$ has two minimal ideals $I_{1}$ and $I_{2}$.

\end{example}

\begin{remark} We note that for the above example $E(E(X))\ncong E(X)$.

\vskip .2cm

In order to compute $E(E(X))$, observe that $h^n\longrightarrow h_{1\theta}$ for some $0\leq \theta \leq 2\pi$ and $h^{-n}\longrightarrow h_{2\theta}$ for some $0\leq \theta \leq 2\pi$ in $X^X$ and we claim that $h^n\longrightarrow H$ and $h^{-n}\longrightarrow K$ in $E(X)^{E(X)}$ where for each $k\in \Z$,

$$H(h^k) = h_{1(\theta+2k\pi \alpha)}, H(h_{1\theta}) = h_{1\theta}, H(h_{2\theta}) = h_{2\theta},$$

$$K(h^k) = h_{2(\theta-2k\pi \alpha)}, K(h_{1\theta}) = h_{1\theta}, K(h_{2\theta}) = h_{2\theta}.$$

since

\begin{center}
$h^n(h^k) \to  h_{1\theta}(h^k)= h_{1(\theta+2k\pi \alpha)}= H(h^k)$,

$h^n(h_{1\theta}) \to h_{1\theta}(h_{1\theta})= h_{1\theta}= H(h_{1\theta})$,

$h^n(h_{2\theta}) \to h_{1\theta}(h_{2\theta}) = h_{2\theta}= H(h_{2\theta})$.
\end{center}

and

\begin{center}
$h^{-n}(h^k) \to h_{2\theta}(h^k)= h_{2(\theta-2k\pi \alpha)}= K(h^k)$,

$h^{-n}(h_{1\theta}) \to h_{2\theta}(h_{1\theta})= h_{1\theta}= K(h_{1\theta})$,

$h^{-n}(h_{2\theta}) \to h_{2\theta}(h_{2\theta}) = h_{2\theta}= K(h_{2\theta})$.
\end{center}

\vskip .2cm

Therefore $E(E(X))= \lbrace h^{n}: n\in \Z \rbrace \cup \lbrace H,K \rbrace$. Here $H,K$ are  not continuous because the sequence $h^n\longrightarrow h_{1\theta}$ but $\lim\limits_{n\rightarrow \infty}H(h^n)=\lim\limits_{n\rightarrow \infty}h_{1(\theta + 2n\pi\alpha)} = h_{1(\theta + \phi)}\neq H(h_{1\theta})$. Also as $h^{-n}\longrightarrow h_{2\theta}$ but $\lim\limits_{n\rightarrow \infty}K(h^{-n})=\lim\limits_{n\rightarrow \infty}h_{2(\theta - 2n\pi\alpha)} = h_{2(\theta - \psi)}\neq K(h_{2\theta})$.

Thus $E(E(X))$ is the two point compactification of $\Z$.

We further note that since $E(E(X))$ is point transitive, and so $E(E(E(X))) \cong E(E(X))$.

\end{remark}

\vskip .5cm

We observe that for a flow $(X,T)$, $P(X)= \bigcup\limits_{\alpha}R_{\alpha}$, where $R_{\alpha}= \lbrace (x,y): px=py\  \forall\  p\in I_{\alpha}, \ I_{\alpha}\subset E(X)\ \text{is minimal ideal in E(X)} \rbrace$. Suppose $(x,y), (y,z)\in P(X)\cap R_{\alpha}$ then $px=py$ and $py=pz$ for each $p\in I_{\alpha}$. So $px=pz$ for all $p\in I_{\alpha}$. Therefore $(x,z)\in P(X)\cap R_{\alpha}$. Therefore   proximal relation is transitive and hence an equivalence relation on $R_{\alpha}$.

We note that each $R_{\alpha} \subset X\times X$ need not be closed, open or even pairwise disjoint, and hence do not form a partition of $P(X)$.

\begin{definition}
A system $(X,T)$ is called finitely proximal if there exists $n \in \N$ such that $P(X)=\bigcup\limits_{i=1}^{n}R_{i}$, where the proximal relation is an equivalence relation on each $R_{i}$. %where $R_{i}= \lbrace (x,y): px=py\  \forall\  p\in I_{i}, I_{i}\subset E(X)\ \text{is minimal ideal in E(X)} \rbrace$.
\end{definition}

\begin{theorem}
For the flow $(X,T)$, $E(X)$ has finitely many minimal ideals $\Leftrightarrow\ (X,T)$ is finitely proximal.
\end{theorem}

 We omit the trivial proof.

\vskip .5cm

\begin{theorem}
A factor of a finitely proximal flow is finitely proximal.
\end{theorem}
\begin{proof}
Suppose $\phi: (X,T)\rightarrow (Y,T)$ is factor mapping. Suppose $(X,T)$ is finitely proximal. Then $P(X)= \bigcup\limits_{i=1}^{n}R_{i}$, where $ \lbrace I_{1}, I_{2}, \ldots, I_{n}\rbrace$ are finitely many minimal ideals and \begin{center}
$R_{i}= \lbrace (x,y): px=py\  \forall\  p\in I_{i}, I_{i}\subset E(X)\ \text{is minimal ideal in E(X)} \rbrace$
\end{center}
Since $\phi$ is factor mapping then there is a factor mapping $\theta:E(X)\rightarrow E(Y)$.

Since minimal ideals are minimal subsets of $E(Y)$ and $E(Y)$ is a factor of $E(X)$, $E(Y)$ also has finitely many minimal ideals. So $\lbrace \theta(I_{1}), \theta(I_{2}), \ldots, \theta(I_{n})\rbrace$ are the minimal ideals in $E(Y)$.

 Thus
\begin{center}
$P(Y)= \bigcup\limits_{i=1}^{n}S_{i}$

where $S_{i}= \lbrace (x,y): px=py\  \forall\  p\in \theta(I_{i}), \theta(I_{i})\subset E(Y)\ \text{is minimal ideal in E(Y)} \rbrace$.
\end{center}

Hence $(Y,T)$ is also finitely proximal.
\end{proof}

\begin{theorem}
A proximal extension of a finitely proximal flow is finitely proximal.
\end{theorem}
\begin{proof}
Suppose $\psi: (X,T)\rightarrow (Y,T)$ is a proximal extension and $(Y,T)$ is finitely proximal. So $E(Y)$ will have finitely many minimal ideals. Again there will be an induced extension $\Psi: E(X)\rightarrow E(Y)$. Since $\Psi$ is an onto proximal homomorphism, $E(X)$ will also have finitely many minimal ideals. So $(X,T)$ is also finitely proximal.
\end{proof}
\vskip .5cm

\begin{definition}
A system $(X,T)$ is called countably proximal if  $P(X)=\bigcup\limits_{n \in \N}R_{i}$, where the proximal relation is an equivalence relation on each $R_{i}$. %where $R_{i}= \lbrace (x,y): px=py\  \forall\  p\in I_{i}, I_{i}\subset E(X)\ \text{is minimal ideal in E(X)} \rbrace$.
\end{definition}

We note that a finitely proximal flow is always countably proximal. We can easily state the below and the proofs for them are simple.

\begin{theorem}
For the flow $(X,T)$, $E(X)$ has countably many minimal ideals $\Leftrightarrow\ (X,T)$ is countably proximal.
\end{theorem}

\begin{theorem}
A factor of a countably proximal flow is countably proximal.
\end{theorem}

\begin{theorem}
A proximal extension of a countably proximal flow is countably proximal.
\end{theorem}

\vskip .5cm

\section{Almost Automorphic Points, Almost Equicontinuity, Almost Periodicity and Metrizability in Enveloping Semigroups}

We recall the concept of almost automorphic points. A point $x\in X$ is called \emph{almost automorphic}, if for every net $\{t_i\}_{i\in I}\subset T$ with $t_i x\to y$ for some $y\in X$ it holds that $t_i^{-1} y\to x$. Flows with almost automorphic points were first studied by Veech \cite{Ve1}. We encourage the reader to refer to \cite{AGN} for more details, and recall some results from there. Note that our $T$ need not be Abelian here as in \cite{AGN}.

\begin{proposition}\cite{AGN}
Equicontinuity of the flow $(X,T)$ implies that every point in $X$ is an almost automorphic point.
\end{proposition}

\begin{proposition}\cite{AGN}
Let $(X,T)$ be a flow so that \emph{every} point in $X$ is almost automorphic.
Then the enveloping semigroup $E(X,T)$ is a group, and the operation of group inversion is a continuous mapping from $E(X,T)$ onto $E(X,T)$ with respect to the product topology on $X^X$.
\end{proposition}

\begin{proposition}\cite{AGN}
If $(X,T)$ is transitive and every point is almost automorphic, then it is minimal and equicontinuous.
\end{proposition}

For equicontinuous $(X,T)$ with $T$ not necessarily Abelian, we recall that $E(X)$ is a group of homeomorphisms and so we can call a point $x\in X$ to be  \emph{almost automorphic} when $px = y$ for some $p \in E(X)$ implies $p^{-1}y = x$. Then in $(E(X),T)$, we note that $pq = r$ will imply $p^{-1}r = q$, i.e. if $\{t_i\}_{i\in I}\subset T$ with $t_i q\to r$ for some $r\in E(X)$, then it holds that $t_i^{-1} r \to q$. Thus every point in $E(X)$ will be almost automorphic.

\begin{proposition} For an equicontinuous $(X,T)$, every point in $E(X)$ is an almost automorphic point. \end{proposition}

From the above propositions and since $E(X,T)$ is minimal for minimal $(X,T)$, we can conclude that:

\begin{theorem} If $(X,T)$ is transitive and equicontinuous then $(E(X),T)$ is  equicontinuous.

In particular, $(E(X),T)$ will also be minimal. \end{theorem}

We note that $(E(X),T)$ is minimal and equicontinuous   vacuously gives that $(X,T)$ is minimal and equicontinuous.

\begin{remark} Compare the theorem above with Theorem \ref{TH EMF} from \cite{AUS}. We have extended Auslander's result here for even non Abelian $T$. Thus if $(X,T)$ is minimal and equicontinuous, with $T$ Abelian or non Abelian, we have $(X,T) \cong (E(X),T)$. \end{remark}

\vskip .5cm

We recall some results  \cite{SGM, ES, GMES, GME, ME, ELBW3}:

\begin{itemize}

\item A compact flow $(X,T)$ is \emph{hereditarily almost equicontinuous} if and only if  $E(X)$ is metrizable.

\item For a metrizable \emph{hereditarily almost equicontinuous} flow $(X,T)$, the flow $(E(X), T)$ is again metrizable \emph{hereditarily almost equicontinuous}.

\item For a metrizable \emph{hereditarily non sensitive} flow $(X,T)$, the flow $(E(X), T)$ is again metrizable \emph{hereditarily non sensitive}.

    \item For a transitive flow $(X,T)$, flow $(X,T)$ is \emph{hereditarily almost equicontinuous} if and only if the flow $(E(X), T)$ is  \emph{hereditarily non sensitive}.

\item Let $(X,T)$ be a compact flow, suppose $T_1$ is a subgroup of $T$ and $X_1$ is a closed $T_1$ invariant subset of $X$. If $E(X,T)$ is metrizable then $E(X_1,T_1)$ is also metrizable.

\item Let $(X,T)$ is metric flow then either $E(X)$ is \emph{seperable Rosenthal compact}, hence with cardinality card$E(X)\leq 2^\omega$ or the compact space $E(X)$ contains a homeomorphic copy of $\beta\N$, hence card$E(X)=2^{2^{\omega}}$.

\item For the flow $(X, T)$,  the following conditions are
equivalent:
\begin{enumerate}
\item The flow $(X, T)$ is hereditarily almost equicontinuous (HAE);

\item The flow $(X, T)$  admits a proper representation on a Radon–Nikodym Banach space;

\item The enveloping semigroup $E(X)$ is metrizable.
\end{enumerate}
\end{itemize}

We recall the below results from \cite{AUS, SGM, TM} for minimal flows and give examples to show that these need not hold for non minimal flows. Thus not everything that holds for minimal systems be true in a non-minimal case.

\begin{theorem} \cite{AUS, MST} \label{TH 6.2} Let $(X,T)$ be a minimal flow, and suppose that all elements of $ E(X)$ are continuous. Then $(X,T)$ is equicontinuous.\end{theorem}

We note that this does not hold for non minimal systems. We recall Example \ref{EX 7.2} here:

\begin{example}

Let $X= \{ x^n : x^n = (x^n_i)$ such that $x^n_i = 1$ if $i = n$ and $x^n_i = 0$ otherwise, $n \in \Z \} \cup \{\overline{0}\} \ \subset \emph{2}^{\Z}$, where $\overline{0}$ is the  sequence of all $0$s. $(X, \sigma)$ is a subshift of the $\emph{2}$-shift. Note that $(X, \sigma)$ is not minimal and contains $\{\overline{0}\}$ as the unique minimal subset.

For any $x \in X$ observe that $\sigma^k(x) \longrightarrow \overline{0}$ and $\sigma^{-k}(x) \longrightarrow \overline{0}$ as $k \longrightarrow \infty$. Thus if $p: X \mapsto X$ is the constant map $p(x) = \overline{0}$, then $\sigma^k \longrightarrow p$ and $\sigma^{-k} \longrightarrow p$ pointwise as $k \longrightarrow \infty$. So in this case $E(X) = \{ \sigma^k: k \in \Z\} \cup \{p\}$ and so $I = \{p\}$ is the unique minimal ideal in $E(X)$.

The point $\overline{0}$ there is not a point of equicontinuity. Though $E(X) \subset \mathcal{C}(X,X)$.

\end{example}

\begin{remark}  We note that Example \ref{EX 5.2} above is a weakly almost periodic (WAP) cascade. \end{remark}

\begin{theorem} \cite{SGM}  A metric minimal flow $(X, T)$ is equicontinuous if and only if its enveloping semigroup $E(X)$ is metrizable. \end{theorem}

We note that our Example \ref{EX 7.2} shows that the above statement does not hold for non-minimal systems.   We note that here $(X, \sigma)$ is not equicontinuous, but is almost equicontinuous since $\overline{0}$ here is a point of sensitivity but still $E(X)$ here is metrizable given the $\sup-$metric on $\mathcal{C}(X,X)$.

We recall from \cite{TM} a problem stated there:

\underline{Problem 3.3.}  Is there a nontrivial minimal proximal system with a metrizable enveloping
semigroup?

We note that the example above answers this problem for a non-minimal case.
\vskip .5cm

\begin{theorem} \cite{TM} Let $(X,T)$ be a minimal  flow such that $E(X)$ is
metrizable. Then

1. There is a unique minimal ideal $I \subseteq E(X)  = E(X, T) \cong E(I, T)$.

2. The Polish group $\mathcal{G} = Aut (I, T)$, of automorphisms of the system $(I, T)$
equipped with the topology of uniform convergence, is compact. \end{theorem}

We note that the above does not hold for non-minimal systems. We again consider Example \ref{EX 5.1}:

\begin{example}

Consider the cascade $([0,1],f)$ where $f(x)=x^2$.
Enveloping semigroup $E(X)$ here is just the two point compactification of $\Z$ as $f^{n} \longrightarrow g_{1}$ and $f^{-n} \longrightarrow g_{2}$ as $n \longrightarrow \infty$, where

\centerline{$g_{1}(x)= \left\{
             \begin{array}{ll}
               0, \ &  x\in [0,1);  \\
               1, & x = 1.
             \end{array}
           \right.$
and $g_{2}(x)= \left\{
                 \begin{array}{ll}
                   0, & x = 0; \\
                   1, &  x\in (0,1].
                 \end{array}
               \right.$}

We complete our observation by noting that $E(X)$ here is metrizable. One of the metrics on $E(X)$ can be taken as $\rho(f^k, f^l) = |k - l|$, $\rho(g_1, f^k) = 1/2^k$, $\rho(g_2, f^k) = 2^k$ and $\rho(g_1,g_2)$ defined appropriately by extending the $\rho$ function defined on a dense subset of $E(X) \times E(X)$.

Note that  $f(g_1) = g_1$ and $f(g_2) = g_2$. Thus $g_1$ and $g_2$ are fixed points in $E(X)$. $I_1 = \{g_1\}$ and $I_2 = \{g_2\}$ are two distinct minimal ideals in $E(X)$, though we have $I_1 \cong E(I_1)$ and $I_2 \cong E(I_2)$.

Also $g_1, g_2$ are not continuous on $[0,1]$, and so Polish groups $\mathcal{G}_1 = Aut (I_1, T)$ and $\mathcal{G}_2 = Aut (I_2, T)$, of automorphisms of the system $(I_1, T)$ and $(I_2, T)$
equipped with the topology of uniform convergence, are not compact(they will fail to be equicontinuous).

\end{example}

\vskip .5cm

For `Weakly almost periodic flows'(WAP)  we have the following mainly due to Ellis and Nerurkar \cite{ELN} and Downarowicz \cite{D}. We just note them here.

\begin{enumerate}
\item A flow $(X,T)$ is WAP if and only if $E(X)\subset C(X,X)$.

\item Let $(X,T)$ be a WAP flow and $I$ be the unique minimal ideal in $E(X)$. Then $I=E(X)$ if and only if $ux=x$ for all $x\in X$ where $u\in I$ is the unique minimal idempotent.

\item Let $(X,T)$ be a WAP flow then;

\begin{enumerate}

\item If $(X,T)$ is minimal, then $(X,T)$ is equicontinuous.

\item If $(X,T)$ is distal, then $(X,T)$ is equicontinuous.
\end{enumerate}

\item Let $(X,T)$ be a WAP flow. If there exists a $p\in E(X)$ such that $pt=p$ or $(tp=p)$ for all $t\in T$, then;

\begin{enumerate}
\item $p=u$, $I=\lbrace u \rbrace$, and all minimal sets are singleton, where $u\in I$ is the unique minimal idempotent.

\item If $\lbrace x_0 \rbrace$ is a minimal set in $\overline{Tx}$ then $ux=x_0$.
\end{enumerate}

\item When the acting group $T$ is Abelian, a point transitive WAP flow $(X,T)$ is always
isomorphic to its enveloping semigroup $E(X)$, which in this case is a commutative
semitopological semigroup.

\item For Abelian $T$ the class of all metric, point transitive,
WAP systems coincides with the class of all metrizable, commutative, semitopological
semigroup compactifications of $T$.

\end{enumerate}

In \cite{D} one can find many interesting examples of
WAP but not equicontinuous $\Z$-systems.

\vskip .5cm

\section{Variations of Transitivity for Enveloping Semigroups}

In this section a study of the concepts of the various forms of transitivity of $E(X)$ is considered. We are grateful to Ethan Akin, Eli Glasner and Benjamin Weiss for the wonderful example of a weakly mixing enveloping semigroup \cite{AEB}. This study basically connects the various forms of transitivities of $(E(X),T)$ with various forms of rigidity of $(X,T)$.

\vskip .5cm

We recall that if  $x_0 \in X$ is a transitive point for the flow $(X,T)$, then \emph{the point transitive flow} $(X,T)$ can be written as a \emph{pointed system} $(X, x_0, T)$. And  that the system $(X,T)$ is a factor of $(E(X),T)$ if and only if there is a dense orbit in $X$.

\vskip .5cm

 The orbit of identity $e \in T$ is dense in $E(X)$, making the system $(E(X),T)$ point transitive that can be expressed as the pointed system $(E(X), e, T)$. If $x \in X$, then the evaluation map is a surjection of pointed systems
$evx : (E(X), e, T) \to (\overline{Tx}, x, T)$ taking the enveloping-semigroup onto the orbit closure of $x$. For any non-empty index set $\Lambda$ and the product system $(X^\Lambda , T)$ we can identify $\Delta E(X)^\Lambda$ with $E(X^\Lambda)$. Thus, we have for any $k$-tuple $(x_1, x_2, \ldots , x_k) \in X^k$ - the
pointed system $(\overline{T(x_1, x_2, \ldots , x_k)}, (x_1, x_2, \ldots , x_k), T)$ is a factor of $(E(X)^k, (\underbrace{e,e,\ldots,e)}_{k- \text{times}}, T)$.

\vskip .5cm

\subsection{Transitive Enveloping Semigroups}

In this subsection, we will only work with cascades.

\begin{remark} In a cascade $(X,f)$, we say that $(X,f)$ is one sided topologically
transitive if there exists some $x\in X$ such that $\overline{\lbrace f^n(x): n \geq 0 \rbrace} = X$
and $(X,f)$ is topologically transitive if there exists some $x\in X$ such that $\overline{\lbrace f^n(x): n \in \Z \rbrace} = X$.

\begin{theorem}\cite{WAL}
For a compact metric space $X$ and a homeomorphism $f:X \rightarrow X$, $f$ is one sided topologically transitive if and only if $f$ is topologically transitive and $\Omega_{f}(X)=X$.
\end{theorem}

Hence, when we talk of topological transitivity of any cascade $(X,f)$, we only look into forward orbits i.e. $\lbrace f^n(x): n \geq 0 \rbrace$. But $E(X)$ need not be metrizable - though it is compact. Hence for the transitivity of $E(X)$ we consider the full orbit $\lbrace f^n(x): n \in \Z \rbrace$. \end{remark}

Note that $(E(X),f)$ is always point transitive, but $E(X)$ need not be metrizable and so $(E(X),f)$ need not  always be topologically transitive. Consider the slightly complex appearing definition stated below:

\begin{definition} For any $n,m \in \N$, for all points $x_1, \ldots, x_n, y_1, \ldots, y_m \in X$ and nonempty, open sets $U_1, \ldots, U_n$, $V_1, \ldots, V_m$ in $X$ whenever

\begin{center}
$(x_1, \ldots, x_n) \in \bigcup \limits_{i \in \Z}(f^{-i}(U_1) \times \ldots \times f^{-i}(U_n))$
\end{center}
and
\begin{center}
$(y_1, \ldots, y_m) \in \bigcup \limits_{j \in \Z}(f^{-j}(V_1) \times \ldots \times f^{-j}(V_m))$,
\end{center}

if there exists $k, l  > 0 \in \Z$ such that $f^k(x_i) \in U_i$ and $f^{k+l}(y_j) \in V_j, \ i=1, \ldots, n, j = 1, \ldots, m$ then we call the cascade $(X,f)$ to be \emph{iteratively transitive}.\end{definition}

\vskip .5cm

The condition $(x_1, \ldots, x_n) \in \bigcup \limits_{i \in \Z}(f^{-i}(U_1) \times \ldots \times f^{-i}(U_n))$ and $(y_1, \ldots, y_m) \in \bigcup \limits_{j \in \Z}(f^{-j}(V_1) \times \ldots \times f^{-j}(V_m))$ for any $n,m \in \N$ ensures that the basic open sets $B_1 = [x_1, \ldots, x_n; U_1, \ldots, U_n],$ and $ B_2 = [y_1, \ldots, y_m; V_1, \ldots, V_m]$ in $E(X)$ (considered as a subspace of $X^X$ with the point-open topology) are nonempty. And hence, $(X,f)$ being iteratively transitive means that for every pair of basic open sets $B_1, B_2 \subset E(X)$, there exists $l > 0 $ such that $f^l(B_1) \cap B_2$ is a non-empty set in $E(X)$ making the system $(E(X), f)$ topologically transitive.

\vskip .5cm

\begin{theorem} \label{2.1} For the cascade $(X,f)$ the following are equivalent:

\begin{enumerate}

  \item The cascade $(E(X), f)$  is topologically transitive.
  \item $(X,f)$ is iteratively transitive.
  \item $(X,f)$ is weakly rigid.
  \item The identity $e = f^0$ is not isolated in $E(X)$.
\end{enumerate}

\end{theorem}

\begin{proof}  Note that the definition of weakly rigid implies that  $\forall \ (x_1, \ldots, x_n) \in X^n$, $(x_1, \ldots, x_n)$ is recurrent in $(X^n, f^{(n)})$ for all $n \in \N$.

$(1) \Leftrightarrow (2)$ follows from the definition.

$(2) \Rightarrow (3)$ Let $(X,f)$ be iteratively transitive. Let $x_1, \ldots, x_n \in X$ and $\epsilon > 0$. Let $U_i$ be the $\epsilon-$ball centered at $x_i$ for all $i=1, \ldots, n$. Choose points $y_1, \ldots, y_n \in X$ such that $f^k(y_i) = x_i$ for all $i=1, \ldots, n$ and $k \in \Z$. Now $(y_1, \ldots, y_n) \in \bigcup \limits_{i \in \Z}(f^{-i}(U_1) \times \ldots \times f^{-i}(U_n))$ and so there exists $ l > 0$ such that $f^k(y_i) \in U_i$ and $f^{k+l}(y_i) \in U_i, \ i=1, \ldots, n$. Since $x_i = f^k(y_i)$, this means that $f^l(x_i) \in U_i$ and so $d(x_i, f^l(x_i)) < \epsilon$ for all $i=1, \ldots, n$. Thus $(X,f)$ is weakly rigid.

$(3) \Rightarrow (4)$ Consider a basic open set $B = [x_1, \ldots, x_n; U_1, \ldots, U_n],$ such that $e \in B$. And since $(X,f)$ is weakly rigid, there exists $ k > 0$ such that $f^k(x_i) \in U_i $ i.e. $f^k \neq e \in B$ and so $e$ is not isolated in $E(X)$.

  $(4) \Rightarrow (1)$ Note that since $e \in E(X)$ is not isolated, any $f^k$ is also not isolated ( $f^n \to e \Rightarrow f^{n+k}  \to f^k$). Thus $E(X)$ is a perfect space and since $\mathcal{O}(e)$ is dense in $E(X)$, the system $(E(X), f)$ is topologically transitive.

\end{proof}

\begin{remark} Note that equivalence of (3) and (4) above is proved   as Lemma 6.1 in \cite{RIG}. \end{remark}

\vskip .5cm

\begin{remark} Note that $(X,f)$ is iteratively transitive does not imply $(X,f)$ is topologically transitive. For example consider $X = \{(r, \theta): r = 1, 2$ and $ 0 \leq \theta \leq 2\pi \}$ and let  $f: X \to X$ be defined as $f(r, \theta) = (r, \theta + 2 \pi \alpha  \mod(2 \pi))$, where $\alpha$ is irrational. Then the identity is not isolated in the  enveloping semigroup $E(X)$   and so is transitive under the action induced by $f$. However, $(X,f)$ is not transitive.\end{remark}

\vskip .5cm

\begin{remark} \label{ex} Note that topologically transitive need not imply iterative transitive. Consider the $\emph{2}-$shift $(\emph{2}^{\Z}, \sigma)$. The enveloping semigroup here will be $E(X) = \beta \Z$ (refer to the next section). Now for points $x = 0^\infty 1^{2k+1} 0^\infty, y = 1^\infty 0^{2l+1}1^\infty \in X$  observe that $x \in [1^{2k+1}], y \in [0^{2l+1}]$ but there is no $\eta \neq e$ for which $\eta([x;[1^{2k+1}]]) \cap [y; [0^{2l+1}]] \neq \emptyset$. Thus $(\emph{2}^{\Z}, \sigma)$ is not iteratively transitive or $(\beta \Z, \Z)$ is not  transitive. But $(\emph{2}^{\Z}, \sigma)$ is an interesting transitive (in fact mixing) system.
\end{remark}

\begin{remark} We compare Theorem \ref{2.1} with Proposition \ref{P 6.2} proved earlier. Since $(E(X),f)$ is transitive, $e$ is an accumulation point of $E(X)$ and so $R_f(X) = X$. This still does not give that $(X,f)$ is transitive. Thus when $(E(X),f)$ is transitive, we get a stronger version of recurrence for $(X,f)$, i.e. weak rigidity. \end{remark}

We recall that the cascade $(2^X, f_*)$ induces the cascade $(2^{2^X}, f_{**})$, which induces the cascade $(2^{2^{2^X}}, f_{***})$, $\ldots$, which induces $({2^{\cdot^{\udots}}}^{2^X}, f_{* \ldots *})$, $\ldots$. And so we have by Theorem \ref{2.1} and Theorem \ref{3.8},

\begin{theorem}
 $(X,f)$ is uniformly rigid if and only if the identity $e_*$ is not isolated in $E(2^X,f_{\ast})$  if and only if the identity $e_{**}$ is not isolated in $E(2^{2^X}, f_{**})$  $\dots$ if and only if the identity $e_{* \ldots *}$ is not isolated in $E({2^{\cdot^{\udots}}}^{2^X}, f_{* \ldots *})$  $\ldots$.

 In particular $\ \mathcal{R}_f(X) = X$, $\mathcal{R}_{f_*}(2^X) = 2^X, \ldots $, $\mathcal{R}_{f_{* \ldots *}}({2^{\cdot^{\udots}}}^{2^X}) = {2^{\cdot^{\udots}}}^{2^X},  \ldots$.
\end{theorem}

This improves the main result in \cite{LPYZ} giving all closed subsets recurrent in any induced space.

\vskip .5cm

 \begin{remark} Note that $E(X)$ will be infinite whenever there is a dense orbit in $X$. However, it is possible that $E(X)$ is infinite and even transitive  when all orbits in $X$ are finite, as can be seen in the example below:

\begin{example}
%We recall Example \ref{EX 2.1} here.
Let  $X=\lbrace re^{i\theta}: r\in A, 0\leq\theta < 2\pi\rbrace$ where $A= \lbrace 1/n :n\in\mathbb{N}\rbrace \cup \lbrace 0\rbrace\cup \lbrace 1-1/n: n \in \mathbb{N}\rbrace$. So, X is a countable disjoint union of circles. Now define the map $T:X\rightarrow X$ as $f(r,\theta)= (r, \theta+ 2\pi r(mod(2 \pi)))$.

Hence at  $r=1$, $f=$ Identity, and $f^k(r,\theta)= (r, \theta + 2k\pi r(mod(2 \pi)))$ and $f^{-k}(r,\theta)= (r, \theta - 2k\pi r(mod(2 \pi)))$.

So points on each circle remain on the same circle for each iteration, since the radius remains same and each $(r, \theta)$ is periodic with period $n$ if $r=1/n$ and $r=1-1/n$.

By  definition, $(X,f)$ is weakly rigid and since $X$ is distal so $E(X)$ is a group and $(E(X),f)$ is minimal as a flow.
\end{example}

\end{remark}

\vskip .5cm

\begin{theorem}  \cite{AEB} \label{2.11} A weakly mixing, weakly rigid cascade $(X, f)$ is connected. \end{theorem}

\begin{remark} We observe that the period doubling system on $\mathbb{S}^1$ factors on the $\emph{2}-$shift. Now the $\emph{2}-$shift is not weakly rigid as discussed in Remark \ref{ex}. Hence the connected period doubling system cannot be weakly rigid and so will not have a transitive enveloping semigroup.\end{remark}

\subsection{Stronger forms of transitivity for Enveloping Semigroups}

We will again deal only with cascades here.

What can one say about mixing and weakly mixing of enveloping semigroups. What about the stronger forms of transitivity? By recalling the study and example in \cite{AEB}, these properties are studied here which leads to several questions $-$ that look far from being settled soon.

\vskip .5cm

We recall that a minimal cascade is strongly transitive and hence all backward orbits are dense \cite{AANT}.

Hence the fundamental Theorem  due to Ellis and Auslander \cite{AUS, ELL} can now be extended for cascades:

\begin{theorem}  For a cascade  $(X,f)$ the following are equivalent:

1. $(X,f)$ is distal.

2. $(E(X),f)$ is distal.

3. $(E(X),f)$ is minimal.

4. $E(X)$ is a group.

5. $\mathcal{O}^\leftarrow(p)$ is dense in $E(X)$ for all $p \in E(X)$ or $(E(X), f)$ is strongly transitive.

\end{theorem}

\vskip .5cm

\emph{What can be said about Mixing Enveloping Semigroups?}

By definition $(E(X),f)$ will be mixing if for any nonempty basic open sets $B_1$ and $B_2$ in $E(X)$, there exists a $N \in \N$ such that $f^k(B_1) \cap B_2 \neq \emptyset$, $\forall \ k \geq N$. This means that for any neighbourhood base $B$ at $e$, there is a   $N \in \N$ such that $f^k(B) \cap B \neq \emptyset$, $\forall \ k \geq N$.

In particular, if $B = [x;U]$ for some open $U \ni x$ there exists $p(x) \in U$ such that $f^kp(x) \in U$ and this is true for all $k \geq N$. This implies that $f^k(U) \cap U \neq \emptyset$ for all $k \geq N$ - giving a stronger version of $x \in \Omega_f(X)$ for all $x \in X$.

\begin{definition} A point $x\in X$ is called \emph{essentially non-wandering} if for every open set $U\subset X$ with $x\in U$, there exists  $N \in \N$ such that $f^n(U)\cap U \neq\emptyset$ for all $n \geq N$.
\end{definition}

We note  the following property of systems with all points essentially non-wandering and skip the trivial proof.

\begin{theorem} The property of all points being essentially non-wandering in a cascade is closed under factors, finite products and subsystems. \end{theorem}

\begin{theorem} The cascade $(E(X),f)$ is mixing $\Rightarrow$ for the cascade $(X,f)$ every $x \in X$ is essentially non-wandering . \end{theorem}

\vskip .5cm

If $(E(X),f)$ is (strongly) mixing then it is also transitive and hence $e$ is not isolated in $E(X)$. Thus $(X,f)$ is weakly rigid. Again in the light of Proposition 6.4 in \cite{RIG}, if $(E(X),T)$ is minimal and strongly mixing then it admits only trivial rigid factors. This dampens the possibility of the existence of a mixing $(E(X),T)$.

We are yet to encounter an example of strongly mixing enveloping semigroup. This leads to the obvious question:

\vskip .5cm

\textbf{Question 1:}\textit{ Do Mixing Enveloping Semigroups Exist?}

\vskip .5cm

\emph{What about weakly mixing $(E(X),T)$?}

This question is also difficult and the only study we have in this direction is \cite{AEB}.

\vskip .5cm

If $E(X) \times E(X)$ is transitive, then $E(X)$ must be a perfect space. Note that $(e,e)$ does not have a dense orbit in $E(X) \times E(X)$. Suppose $\mathcal{O}(\alpha, \beta)$ is dense in $E(X) \times E(X)$. Then for some $\eta \in E(X)$, $\eta \alpha = e$ and $\eta \beta = e$. So $\alpha$ and $\beta$ have a same left inverse, and so none of $\alpha$ or $\beta$ can be an element of $\{f^n: n \in \Z\}$ else $\alpha =  \beta \in \{f^n: n \in \Z\}$ which would make the orbit of $(e,e)$ dense in $E(X) \times E(X)$. So essentially both $\alpha, \beta \in E^*(X) = E(X) \setminus \{f^n: n \in \Z\}$. Hence a weakly mixing $(E(X),T)$ need not have a dense orbit.

\vskip .5cm

Some properties of weakly mixing enveloping semigroups are studied in \cite{AEB}, where an excellent example of such a system is also described. Recall:

\begin{theorem} \cite{AEB} For a cascade $(X, f)$ the following are equivalent:

(i) The system $(X, f)$ has a weakly mixing enveloping semigroup.

(ii) For any $k$-tuple $(x_1, x_2, \ldots , x_k) \in X^k$ the pointed system given by $\overline{\mathcal{O}(x_1, x_2, \ldots , x_k)}$ is weakly mixing.
\end{theorem}

This puts a rather strong product recurrence for the weakly mixing $(E(X),f)$.

\vskip .5cm

If a system  is strongly  product transitive then it is weak mixing \cite{AANT}. Since there exists weak mixing enveloping semigroups, the following natural question arises:

\vskip .5cm

\textbf{Question 2:} \textit{For the cascade $(X,f)$, can  $(E(X),f)$ ever be strongly product transitive?}

We cannot conclude anything here as in case of transitivity since $(E(X)^2,f^{(2)})$ need not have a dense orbit.

\vskip .5cm

We get an equivalence of topological transitivity of $(E(X),T)$ and weak rigidity of  $(X,T)$. And we have some stronger form of recurrence of $(X,f)$ for both weak mixing and mixing  $(E(X),f)$. So again the following natural question arises:

\vskip .5cm

\textbf{Question 3:} \textit{Is there any rigidity condition on $(X,f)$ that gives weakly mixing or mixing of $(E(X),f)$?}

\section{Enveloping Semigroup of the Induced Systems}

In this section, we  discuss some properties of enveloping semigroup of the induced system $(2^X,T)$ and the relation between enveloping semingroups $E(X)$ and $E(2^X)$. We note that some of our results are true when the acting group is any discrete, topological group $T$. Though, here we will be mainly concentrating on cascades.

Consider the cascade $(X,f)$. This induces the cascade $(2^X, f_{\ast})$.

Recall $\mathcal{F}_1(X)=\lbrace \lbrace x \rbrace : x\in X\rbrace \equiv X$. Now consider $E(2^X,f_{\ast})\subset (2^X)^{(2^X)}$. Since $f(\lbrace x \rbrace)=\lbrace f(x) \rbrace$. So
\begin{center}
$f^n(\mathcal{F}_1(X))\subset \mathcal{F}_1(X)$ for every $n \in \mathbb{Z}$.
\end{center}

\begin{lemma} \label{10.1}For every $\alpha\in E(2^X,f_{\ast})$, $\alpha(\mathcal{F}_1(X))\subset \mathcal{F}_1(X)$. \end{lemma}

\begin{proof}
Let $\alpha \in E(2^X)$, there is a net $\lbrace n_i \rbrace $ such that $f_{\ast}^{n_i}\longrightarrow \alpha$. For any $x\in X$, $f_{\ast}^{n_i}(\lbrace x \rbrace)\longrightarrow \alpha(\lbrace x \rbrace )$.

Since $f_{\ast}^{n_i}(\lbrace x \rbrace)\in \mathcal{F}_1(X)$. Let $f_{\ast}^{n_i}(\lbrace x \rbrace)= \lbrace y_i \rbrace$. Suppose $\lbrace y_i \rbrace \longrightarrow A$ i.e $\alpha(\lbrace x \rbrace)=A$ where $A$ contains at least two elements.

Let $\langle U_1,U_2,\ldots, U_n\rangle$ be any neighbourhood of $A$ and for any $a\in A$, $U_a$ is any open set containing $a$. Then $\langle U_a,U_1,U_2,\ldots, U_n\rangle$ will also be a neighbourhood of $A$. Since $\lbrace y_i \rbrace \longrightarrow A$, there is an $l$ such that $\lbrace y_i \rbrace \in \langle U_a,U_1,U_2,\ldots, U_n\rangle$ for all $i\geq l$,
means for all $i\geq l$, $\lbrace y_i \rbrace \subset \bigcup_{i=1}^{n} U_i \cup U_a$ and $\lbrace y_i \rbrace \cap U_i \neq \emptyset$ for all $1\leq i\leq n$ and $\lbrace y_i \rbrace \cap U_a \neq \emptyset$. So $\lbrace y_i \rbrace \in \langle U_a \rangle$ for all $i\geq l$. So,
\begin{center}
$\lbrace y_i \rbrace \longrightarrow \lbrace a \rbrace $
\end{center}
This is true for all $a\in A$. Since $E(2^X,f_{\ast})$ is Hausdorff, $A$ is a singleton.
\end{proof}

\begin{lemma} \label{10.2}
For every $\alpha \in E(2^X,f_{\ast})$, $A\subset B\Rightarrow \alpha(A)\subset \alpha(B)$ for any $A,B\in 2^X$.
\end{lemma}
\begin{proof}
Since for any $\alpha \in E(2^X)$, there is a net $\lbrace n_i \rbrace $ such that $f_{\ast}^{n_i}\longrightarrow \alpha$. So if $A\subset B$ then for every $i$, $f_{\ast}^{n_i}(A)=f^{n_i}(A)= \bigcup_{x\in A} \lbrace f^{n_i}(x) \rbrace\subset \bigcup_{x\in B} \lbrace f^{n_i}(x) \rbrace= f^{n_i}(B) = f_{\ast}^{n_i}(B) \Rightarrow \lim\limits_i f_{\ast}^{n_i}(A)\subset \lim\limits_i f_{\ast}^{n_i}(B)\Rightarrow \alpha(A)\subset \alpha(B)$.
\end{proof}

\begin{lemma} \label{10.3}
For any  $\alpha \in E(2^X,f_{\ast})$ and a finite set $A\in 2^X$, $\alpha(A)= \bigcup_{x\in A} \lbrace  \alpha (\lbrace x \rbrace)  \rbrace$.
\end{lemma}
\begin{proof}
Since $\alpha\in E(2^X,f_{\ast})$, there is a net $\lbrace{n_i}\rbrace$ such that $f_{\ast}^{n_i}\longrightarrow \alpha\Rightarrow f_{\ast}^{n_i}(A)\longrightarrow \alpha(A) $ for each $A\in 2^X$.

Let $S_i= f_{\ast}^{n_i}(A)= \bigcup_{x\in A} \lbrace f^{n_i}(x) \rbrace$ and our claim is that,
\begin{center}
$\lim\limits_{i} S_i =  \bigcup_{x\in A} \lbrace  \alpha (\lbrace x \rbrace)  \rbrace$. \end{center}

Let $<U_i,U_2,\ldots, U_n>$ be any open neighbourhood of $\bigcup_{x\in A} \lbrace  \alpha (\lbrace x \rbrace)  \rbrace$.
\begin{center}
$\Rightarrow \bigcup_{x\in A} \lbrace  \alpha (\lbrace x \rbrace)  \rbrace \subset \bigcup_{j=1}^n U_j$ and $\bigcup_{x\in A} \lbrace  \alpha (\lbrace x \rbrace)  \rbrace \bigcap U_j \neq \emptyset$ for all $1\leq j \leq n$.
\end{center}
Now for any $U_k$,  $\exists\ \alpha(\lbrace x_k \rbrace)\in \bigcup_{x\in A} \lbrace  \alpha (\lbrace x \rbrace)  \rbrace$ such that $\alpha(\lbrace x_k \rbrace) \in U_k$.

By the convergence $f_{\ast}^{n_{i}} \longrightarrow \alpha$, $f_{\ast}^{n_i}(\lbrace x_k \rbrace) \longrightarrow \alpha(\lbrace x_k \rbrace)$. So there is an $N_k$ such that $f_{\ast}^{n_i}(\lbrace x_k \rbrace)\in U_k$ for all $i\geq N_k \Rightarrow S_i \bigcap U_k \neq \emptyset $.

Since $A$ is a finite set say $A=\lbrace x_1, x_2,\ldots x_l\rbrace$ and $f_{\ast}^{n_i}(\lbrace x_k \rbrace) \longrightarrow \alpha(\lbrace x_k \rbrace)$ for each $x_k\in A$. So by choosing $N$ such that $N\geq N_k$ for each $1\leq k \leq l$ such that $f_{\ast}^{n_i}(\lbrace x_k \rbrace)\in \bigcup_{j=1}^n U_j\Rightarrow S_i \subset \bigcup_{j=1}^n U_j$ for all $i\geq N$. Which means $\lim\limits_{i} S_i =  \bigcup_{x\in A} \lbrace  \alpha (\lbrace x \rbrace)  \rbrace$. Since $\lim\limits_{i} S_i = \alpha(A)$.

So by the Hausdorffness of $2^X$, $\alpha(A)= \bigcup_{x\in A} \lbrace  \alpha (\lbrace x \rbrace)  \rbrace$.

\end{proof}

\begin{theorem} \label{T 10.1}
Let $(X,f)$ be an equicontinuous cascade. Then every member of $E(2^X,f_{\ast})$ is inducible.
\end{theorem}
\begin{proof}
For a member of $E(2^X, f_{\ast})$ to be inducible, it should satisfy the three conditions mentioned in the Theorem \ref{cha}. In the above Lemmas \ref{10.1} and \ref{10.2}, we have proved the first two conditions.

Now we will prove that for $\alpha\in E(2^X,f_{\ast})$, $\alpha$ in minimal under the relation $\prec$.

First of all, suppose there is an $n\in \Z$ such that $f_{\ast}^n \prec \alpha \Rightarrow f_{\ast}^n(\lbrace x \rbrace)\subseteq \alpha(\lbrace x \rbrace)$. Since both are singleton,

\begin{center}
$f_{\ast}^n(\lbrace x \rbrace)=\alpha(\lbrace x \rbrace)$ and $f_{\ast}^n(\lbrace x \rbrace)= \lbrace f^n(x) \rbrace \Rightarrow \alpha(\lbrace x \rbrace)= \lbrace f^n(x) \rbrace$.
\end{center}
Now for any finite set $A\in 2^X$, by Lemma \ref{10.3}, $\alpha(A)= \bigcup_{x\in A} \lbrace  \alpha (\lbrace x \rbrace)  \rbrace = \bigcup_{x\in A} \lbrace  f^n (\lbrace x \rbrace)\rbrace= f_{\ast}^{n}(A)$.

Now the set of all finite subsets of $X$ is dense in $2^X$, So for any $E \in 2^X$, there is a net $\lbrace F_j \rbrace$ of finite subsets in $2^X$ such that $F_j \longrightarrow E$.

Since $f_{\ast}^{n}$ is continuous $\Rightarrow f_{\ast}^{n}(F_j)\longrightarrow f_{\ast}^{n}(E)$. From the above discussion $f_{\ast}^{n}(F_j)= \alpha(F_j)$ for each $j$. So,

\begin{center}
$\alpha(F_j)\longrightarrow f_{\ast}^{n}(E)$.
\end{center}
Also $\lim\limits_j \alpha(F_j)=\lim\limits_j \lim\limits_i f_{\ast}^{n_{i}} (F_j)$. Since $(X,T)$ is equicontinuous $\Rightarrow (2^X,T)$ is equicontinuous and so the topology of pointwise convergent will coincide with uniform convergence. So we can interchange the limits. So we have

\begin{center}
$\lim\limits_{j} \alpha(F_j)=\lim\limits_j \lim\limits_i f_{\ast}^{n_{i}} (F_j)= \lim\limits_i f_{\ast}^{n_{i}} (\lim\limits_j F_j)= \lim\limits_i f_{\ast}^{n_{i}} (E)= \alpha(E)$.
\end{center}
So, $\alpha(F_j)\longrightarrow \alpha(E)$. Hence again by the Hausdorffness of $2^X$, $\alpha(E)= f_{\ast}^{n}(E)\Rightarrow \alpha= f_{\ast}^{n}$.

Now for any $\beta \prec \alpha$ where $\beta \in E(2^X)\setminus \lbrace f_{\ast}^n:n\in \Z\rbrace$. Then $\beta(\lbrace x \rbrace)\subseteq \alpha(\lbrace x \rbrace)\ \Rightarrow \beta(\lbrace x \rbrace)=\alpha(\lbrace x \rbrace)$ for all $x\in X$. So, for any $A\in 2^X$, there is a sequence $\lbrace F_j \rbrace$ of finite sets in $2^X$ such that $F_j \longrightarrow A$ and since $\alpha, \beta$ are continuous because $(2^X, T)$ is equicontinuous
\begin{center}
$ \alpha(A)=\lim\limits_j \alpha(F_j)=\lim\limits_j \bigcup_{x\in F_j} \lbrace \alpha (\lbrace x \rbrace) \rbrace = \lim\limits_j \bigcup_{x\in F_j} \beta (\lbrace x \rbrace)=\lim\limits_j \beta(F_j)=\beta(A)$.
\end{center}
Hence $\alpha=\beta$.

So each member of $E(2^X,f_{\ast})$ satisfies all the three conditions of  Theorem \ref{cha} for an equicontinuous system. Hence the members of $E(2^X,f_{\ast})$ are inducible.
\end{proof}

\begin{remark} We note that  all the results proved above are true for a flow $(X,T)$ also. \end{remark}

\begin{theorem} \label{T10.2}
For a flow $(X,T)$, where $X$ is compact metric space, there is a continuous flow homomorphism $\Theta: E(2^X)\longrightarrow E(X)$.
\end{theorem}
\begin{proof}
Define the map $\Theta:E(2^X)\rightarrow E(X),$ defined as \begin{center}
$\Theta(\alpha)= \alpha'$ where $\lbrace\alpha'(x)\rbrace =\alpha(\lbrace x \rbrace)\ \forall\ x\in X$.
\end{center}

\textbf{$\Theta$ is well defined:} $\Theta(\pi_{\ast}^t)= \pi^t$ is well defined in very natural way. For $\alpha \in E(2^X)\setminus \lbrace \pi_{\ast}^t:t\in T \rbrace$ there is a net $\lbrace n_k \rbrace$ such that $\pi_{\ast}^{t_{n_k}}\longrightarrow \alpha$. So, $\pi_{\ast}^{t_{n_k}}(A)\longrightarrow \alpha(A)$ for every $A\in 2^X$.

Therefore $\pi_{\ast}^{t_{n_k}}(\lbrace x \rbrace)\longrightarrow \alpha(\lbrace x \rbrace)$ for every $x\in X$ and $\alpha(\lbrace x \rbrace)\in \mathcal{F}_1(X)\equiv X\ \Rightarrow\ \pi^{n_k}(x)\longrightarrow \alpha'(x)$. Since $E(X)$ has point open topology, $\pi^{n_k}\longrightarrow \alpha'$. So, $\alpha'\in E(X),$ which shows that $\Theta$ is well-defined.

\textbf{$\Theta$ is continuous:} Let $\mathrm{W}=\langle x_1,x_2,\ldots,x_n;U_1,U_2,\ldots,U_n\rangle$ be any open set in $E(X,f)$. Let $\alpha\in \Theta^{-1}(\mathrm{W})\ \Rightarrow\ \Theta(\alpha)\in \mathrm{W}\ \Rightarrow\ \alpha'(x_i)\in U_i$ for $1\leq i \leq n$.

\textbf{Claim: $\alpha\in \langle \lbrace x_1,x_2,\ldots,x_n \rbrace ; \langle U_1,U_2,\ldots,U_n \rangle\rangle \subset \Theta^{-1}(\mathrm{W}$)}. Since $\alpha(x_i)\in U_i$. So,
\begin{center}
$\alpha(\lbrace x_1,x_2,\ldots,x_n \rbrace)=\lbrace \alpha(x_1),\alpha(x_2),\ldots,\alpha(x_i) \rbrace\subseteq \bigcup_{i=n}^{n}U_i$ and
\end{center}
\begin{center}

$\lbrace \alpha(x_1),\alpha(x_2),\ldots,\alpha(x_i) \rbrace \cap U_i \neq \emptyset\ \Rightarrow\ \alpha(\lbrace x_1,x_2,\ldots,x_n \rbrace)\in \langle U_1,U_2,\ldots,U_n \rangle$,
\end{center}
which means
\begin{center}
$\alpha\in \langle \lbrace x_1,x_2,\ldots,x_n \rbrace ; \langle U_1,U_2,\ldots,U_n \rangle\rangle$.
\end{center}
 Now for any $\beta\in \langle \lbrace x_1,x_2,\ldots,x_n \rbrace ; \langle U_1,U_2,\ldots,U_n \rangle\rangle$

$\beta(\lbrace x_1,x_2,\ldots,x_n \rbrace)\in \langle U_1,U_2,\ldots,U_n \rangle \Rightarrow \lbrace \beta(x_1),\beta(x_2),\ldots,\beta(x_i) \rbrace\subseteq \bigcup_{i=n}^{n}U_i$

and $\lbrace \beta(x_1),\beta(x_2),\ldots,\beta(x_i) \rbrace \cap U_i \neq \emptyset$ for $1\leq i \leq n$. So, for every $U_i\ \exists\ l,1\leq l \leq n$ such that $\beta(x_l)\in U_i$. which means $\Theta(\beta)=\beta'\in \langle x_1,x_2,\ldots,x_n ; U_1,U_2,\ldots,U_n\rangle=\mathrm{W}\ \Rightarrow \beta \in \Theta^{-1}(\mathrm{W})$. So $\langle \lbrace x_1,x_2,\ldots,x_n \rbrace ; \langle U_1,U_2,\ldots,U_n \rangle\rangle \subset \Theta^{-1}(\mathrm{W})$. Hence the claim is true and $\Theta^{-1}(\mathrm{W})$ is open. Therefore $\Theta$ is continuous.

%\textbf{$\Theta$ is a bijection:} Let $\gamma\in E(X)$ be any element. So, there is a net $\lbrace n_k \rbrace$ in $\T(\Z\ or\ \N)$ such that $f^{n_k}\longrightarrow \gamma$. Define $\gamma_{\ast}:2^X\rightarrow2^X$ as $\gamma_{\ast}(B)=\lbrace \gamma(y): y\in B \rbrace$.

%Now let any $B\in 2^X$ then $f^{n_k}(y)\longrightarrow \gamma(y)$ for all $y\in B$. For any open nhd. $\langle W_1,W_1,\ldots,W_p \rangle$ of $\gamma_{\ast}(B)$ in $2^X,\ \gamma_{\ast}(B)\subseteq \bigcup_{j=1}^p W_j$ and $\gamma_{\ast}(B)\cap W_j$ for $1\leq j \leq p$.

%Here since $f^{n_k}(y)\longrightarrow \gamma(y)$ for all $y\in B$ and since for every $j,\ 1\leq j \leq p\ \exists\ y\in B$ such that $\gamma(y)\in W_j$. So,there is a $k_y$ in the directed set such that $f^{n_k}(y)\in W_j$ for all $k\geq k_y$. By the property of directed set, there will be a $l,\ 1\leq l \leq p$ such that $\lbrace f^{n_k}(y): y\in B,k\geq l\rbrace \subseteq \bigcup_{j=1}^p W_j$ and $\lbrace f^{n_k}(y): y\in B\rbrace \cap W_j \neq \emptyset$ for $1\leq j \leq p$.

%Therefore $f_{\ast}^{n_k}(B)\subseteq \bigcup_{j=1}^p W_j$ and $f_{\ast}^{n_k}(B)\cap W_j\neq \emptyset$ for all $k\geq l$. Hence $f_{\ast}^{n_k}(B)\in \langle W_1,W_1,\ldots,W_p \rangle$ for all $k\geq l$. So, $f_{\ast}^{n_k}\longrightarrow \gamma_{\ast}$. So, $\gamma_{\ast}\in E(2^X,f_{\ast})$ and $\Theta(\gamma_{\ast})=\gamma$ and $\gamma_{\ast}$ is unique such. Hence $\Theta$ is a bijection. Also since $E(2^X,f_{\ast})$ and $E(X,f)$ are compact and Hausdorff, $\Theta$ is a homeomorphism.

\textbf{$\Theta$ is flow homomorphism:} For any $\alpha\in E(2^X)$ and $x\in X$, $ \Theta(\pi_{\ast}^t\circ\alpha)(x)=(\pi_{\ast}^t\circ\alpha)'(x)= \pi_{\ast}^t\alpha(\lbrace x \rbrace)=\pi^t(\Theta(\alpha))(x)$, which shows that $\Theta(\pi_{\ast}^t\circ\alpha)= \pi^t(\Theta(\alpha))$.

%$E(2^X,f_{\ast})\cong E(X,f)$ as flow under $\Theta$.

%\textbf{$\Theta$ is semigroup isomorphism:} For $\alpha,\beta\in E(2^X,f_{\ast}),\ \Theta(\alpha\circ\beta)=(\alpha\circ\beta)|_X=\alpha|_X \circ \beta|_X=\Theta(\alpha)\circ \Theta(\beta)$.
\end{proof}

\begin{corollary} \label{7.1}
Suppose $(2^X,T)$ is a weakly almost periodic (WAP) flow then $\Theta$ is injective.
\end{corollary}
\begin{proof}
Suppose
\begin{align*}
\Theta(\alpha)&=\Theta(\beta)\\
\alpha^{\prime}&= \beta^{\prime}\\
\alpha(\lbrace x \rbrace) &= \beta (\lbrace x \rbrace)
\end{align*}
for $\alpha,\beta \in E(2^X)$ and each $x\in X$.

Now for any finite $A\in 2^X$, Lemma \ref{10.3} gives
\begin{align*}
\alpha(A)= \bigcup_{x\in A} \lbrace  \alpha (\lbrace x \rbrace)  \rbrace = \bigcup_{x\in A} \lbrace  \beta (\lbrace x \rbrace)  \rbrace = \beta(A).
\end{align*}
Since $(2^X,T)$ is WAP, so $\alpha, \beta\in E(2^X)$ are continuous and set of finite sets is dense in $2^X$. Therefore $\alpha(B)=\beta(B)$ for any $B\in 2^X \Rightarrow \alpha=\beta\Rightarrow \Theta$ is injective.
\end{proof}

\begin{corollary} \label{7.2}
Suppose $(X,T)$ is weakly almost periodic (WAP) then $E(X)\setminus \Theta(E(2^X))$ contains no ideal.
\end{corollary}
\begin{proof}
Since $(X,T)$ is WAP, by Theorem \ref{thm} $E(X)$ has a unique minimal ideal $I$, and $I$ contains a unique idempotent $u$.

Suppose  $I\subset E(X)\setminus \Theta(E(2^X))$. Let $u\in I$ is an idempotent and let $p\in \Theta(E(2^X))$ be any point. So $p=\Theta(r)$ for some $r\in E(2^X)$. Also there is a net $\lbrace t_i \rbrace$ in $T$ such that $\lim\limits_{i}\pi^{t_{i}}=u$.

We observe that for any $p\in E(X)$ and  the unique minimal  ideal $I$ with $u$ an identity in $I$, $pu=u(pu)=(up)u=up$. So every element $p\in E(X)$ commutes with the minimal idempotent $u$.

Since $\Theta$ is continuous and inter-twining, so
\begin{align*}
pu=up=u\Theta(r)&= \lim\limits_{i}\pi^{t_{i}}\Theta(r)\\
&= \lim\limits_{i} \Theta(\pi_{\ast}^{t_i} r)\\
&= \Theta(\lim\limits_{i} \pi_{\ast}^{t_i} r)\\
&= \Theta(sr)\ \text{where}\ \lim\limits_{i} \pi_{\ast}^{t_i}=s\ \text{in}\ E(2^X).
\end{align*}

So $pu \notin I$, which contradicts that $I$ is an ideal. So  $E(X)\setminus \Theta(E(2^X))$ contains no ideals.

\end{proof}

\begin{remark} We note here that though $E(X)\setminus \Theta(E(2^X))$ contains no ideals, we still cannot conclude that $E(X) = \Theta(E(2^X))$. $p \in E(X)$ implies that there is a net $\{t_i\}$ in $T$ such that
$\pi^{t_i} \to p$. Now $p$ is continuous and so induces the continuous map $p_*$ on $2^X$. But we cannot say that $\pi_*^{t_i} \to p_*$ on $2^X$ since that would imply uniform convergence of $\pi^{t_i} \to p$ on $X$.

We can also say that for the unique ideal $I \subset E(X)$,  $I = \Theta(J_*)$ for all minimal ideals $J_*$ in $E(2^X)$. Though, we need more information to even say if $\Theta$  will be a finite to one map.
\end{remark}

\begin{remark} Since $(2^X,T)$ is WAP, it has a unique minimal ideal say $I_*$. By  Theorem \ref{T10.2} and Corollary \ref{7.1}, $\Theta: E(2^X)\longrightarrow E(X)$ is an injective flow homomorphism. Hence, $\theta(I_*)$ is the unique minimal ideal in $E(X)$. \end{remark}

 We note that by Proposition \ref{wap}, if $(2^X,T)$ is a weakly almost periodic  flow then $(X,T)$ is also weakly almost periodic. We note that the converse is not true here, i.e. $(X,T)$ can be a WAP flow but $(2^X,T)$ need not be  WAP.

\begin{example} We recall Example \ref{5.2}. Here $X= \{ x^n : x^n = (x^n_i)$ such that $x^n_i = 1$ if $i = n$ and $x^n_i = 0$ otherwise, $n \in \Z \} \cup \{\overline{0}\} \ \subset \emph{2}^{\Z}$, where $\overline{0}$ is the  sequence of all $0$s. And $(X, \sigma)$ is a subshift of the $\emph{2}$-shift. Note that $(X, \sigma)$ is a WAP cascade since $E(X) = \{ \sigma^k: k \in \Z\} \cup \{p\}$ where $p: X \rightarrow X$ is the constant map $p(x) = \overline{0}$.

Now for  $X \in 2^X$ observe that $\sigma_*^k(X) \longrightarrow X$ and $\sigma_*^{-k}(X) \longrightarrow X $ as $k \longrightarrow \infty$. And for any $A (\subset X) \in 2^X$, $\sigma_*^k(A) \longrightarrow \{\overline{0}\}$ and $\sigma_*^{-k}(A) \longrightarrow \{ \overline{0}\}$ as $k \longrightarrow \infty$.

Thus if $p^\dag: 2^X \mapsto 2^X$ is defined as  $p^\dag(A) = \left\{
                                                          \begin{array}{ll}
                                                            \{\overline{0}\}, & \hbox{$A \neq X$;} \\
                                                            X, & \hbox{$A = X$.}
                                                          \end{array}
                                                        \right.$,

then $\sigma_*^k \longrightarrow p^\dag$ and $\sigma_*^{-k} \longrightarrow p^\dag$ pointwise as $k \longrightarrow \infty$. So for the induced cascade $(2^X, \sigma_*)$,  $E(2^X) = \{ \sigma_*^k: k \in \Z\} \cup \{p^\dag\}$. Note that $I_* = \{p^\dag\}$ is the unique minimal ideal in $E(2^X)$.

Now if $\{A_n\}$ is a sequence in $2^X$ with $A_n \neq X$, $\ \forall n \in \N$ and $A_n \to X$ in $2^X$, then $p^\dag(A_n) \nrightarrow p^\dag(X)$ i.e. $p^\dag$ is not continuous. Thus $(2^X, \sigma_*)$ is not WAP.

\end{example}

\vskip .5cm

It is known that if $(X,T)$ is equicontinuous, $E(X)$ is a group of homeomorphisms of $X$ and the topologies of pointwise and uniform convergence coincide on $E(X)$.

\begin{theorem} \label{TH 10.3}
If $(X,T)$ is an equicontinuous flow, then $E(X)$ is conjugate to $E(2^X)$.
\end{theorem}
\begin{proof}
We recall that the map defined in last theorem $\Theta: E(2^X)\longrightarrow E(X)$ is a flow homomorphism. Since $(X,T)$ is equicontinuous, $E(X)$ is a group of homeomorphism of $X$ and the topologies of pointwise and uniform convergence coincide on $E(X)$.

So, for any $\alpha\in E(X)$, we define $\alpha_{\ast}:2^X \rightarrow 2^X$ as $\alpha_{\ast}(A):=\alpha(A)$, which is well defined because $\alpha$ is homeomorphism.

For surjectivity we have to prove that $\alpha_{\ast}\in E(2^X)$. Since $\alpha\in E(X)$, there is a net $\lbrace t_i \rbrace \subset T$ such that $\pi^{t_i}\longrightarrow \alpha$. Since $E(X)$ has topology of uniform convergence, for every $x\in X$ and an open $U \subset X$ and $x\in X$, there is an $l$ from the directed set of $\lbrace t_i \rbrace$ such that $\pi^{t_i}(x)\in U$ for all $i\geq l$.

Now let $A\in2^X$ be any element and $\alpha_{\ast}(A)\in \langle U_1, U_2,\ldots, U_n\rangle$ is an open set.
So,  $\alpha_{\ast}(A)\subset \bigcup_{j=1}^n U_j$ and $\alpha_{\ast}(A)\cap U_j \neq \emptyset$ for $1\leq j \leq n$.

So, $\bigcup_{j=1}^n U_j$ is neighbourhood of $\alpha(x)$ for every $x\in A$. Therefore for each $x_k\in A$, there is an $l_k$ such that $\pi^{t_i}_{\ast}(\lbrace x_k \rbrace)=\pi^{t_i}(x_k)\in \bigcup_{i=1}^n U_j$ for all $i\geq l_k$. Since the topology is of uniform convergence, there is an $l$ such that for each $x\in A$, $\pi^{t_i}_{\ast}(\lbrace x \rbrace)=\pi^{t_i}(x)\subset \bigcup_{i=1}^n U_j$ for all $i\geq l$. Therefore $\pi^{t_i}_{\ast}(A)=\pi^{t_i}(A)\subset \bigcup_{i=1}^n U_j$ for all $i\geq l$. Also for any $U_k$, there is an $x\in A$ such that $\alpha(x)\in U_k$. So there is an $l_k$ in the directed set such that $\pi^{t_i}(x)\in U_k$ for all $i\geq l_k\Rightarrow \pi^{t_i}(A)\cap U_k\neq \emptyset$ for all $i\geq l_k$. In the similar way, for every $U_j(1\leq j \leq n)$, there is an $l_j$ in the directed set such that $\pi^{t_i}(A)\cap U_j\neq \emptyset$ for all $i\geq l_j$.

Choose $l_0$ as $l_0\geq l$ and $l_0\geq l_j$ for all $1\leq j \leq n$. For this $l_0$, $\pi^{t_i}(A)\cap U_j\neq \emptyset$ for all $i\geq l_0$ and for $1\leq j \leq n$. Therefore for all $i\geq l_0$, $\pi^{t_i}_{\ast}(A)\in \bigcup_{j=1}^n U_j$ and $\pi^{t_i}_{\ast}(A)\cap U_j\neq \emptyset$ for each $1\leq j \leq n$. Which means $\pi^{t_i}_{\ast}(A)\in \langle U_1, U_2,\ldots, U_n\rangle$ for all $i\geq l_0$.

So, $\pi^{t_i}_{\ast}(A)\longrightarrow \alpha_{\ast}(A)$ for all $A \in 2^X$ and hence $\pi^{t_i}_{\ast}\longrightarrow \alpha_{\ast}$. So $\Theta(\alpha_{\ast})=\alpha$. So $\Theta$ is surjective.

Now if $\Theta(\alpha)=\Theta(\beta)$ then $\alpha(\lbrace x\rbrace)= \beta(\lbrace x\rbrace)$ for every $x\in X$. So $\alpha(A)=\beta(A)$ for each finite set $A\in 2^X$. Again since $\alpha,\beta$ are continuous, the same procedure in Theorem \ref{T 10.1} we get, $\alpha(E)=\beta(E)$ for each $E\in 2^X\Rightarrow \alpha=\beta$. Hence $\Theta$ is injective also. Now since $E(2^X, T)$ is compact $\Rightarrow \Theta$ is a homeomorphism. Since $\Theta$ is a flow homomorphism, $\Theta$ is a conjugacy.
\end{proof}

Here $(X,T)$ being equicontinuous is a sufficient condition, and not necessary. This can be seen in the example below:

\begin{example} We recall Example \ref{EX 5.1} here:

Let $f: [0,1]\rightarrow [0,1]$ as $f(x)=x^2$ then $f$ is a homeomorphism on $[0,1]$. The system $([0,1],f)$ is not distal but almost equicontinuous and \begin{center}

$E(X)=\lbrace \ldots f^{-n},\ldots,f^{-1},f^{0}=e ,f^{1},\ldots,f^{n},\ldots\rbrace \cup \lbrace g_1, g_2\rbrace$,
\end{center} where
\begin{center}

$g_1(x)= \left\{
           \begin{array}{ll}
             0, & \hbox{$x\in [0,1)$;} \\
             1, & \hbox{$x=1$.}
           \end{array}
         \right.$

 $g_2(x)= \left\{
            \begin{array}{ll}
              1, & \hbox{$x\in (0,1]$;} \\
              0, & \hbox{$x=0$.}
            \end{array}
          \right.$
\end{center}
Note that $g_1$,$g_2$ are idempotents and $E(X)$ is the two point compactification of $\Z$.

Now consider the induced system $(2^X, f_{\ast})$. This system is also not distal but almost equicontinuous and
\begin{center}
$E(2^X)=\lbrace \ldots f_{\ast}^{-n},\ldots,f_{\ast}^{-1},f_{\ast}^{0}=e ,f_{\ast}^{1},\ldots,f_{\ast}^{n},\ldots\rbrace \cup \lbrace h_{1}^{\ast}, h_{2}^{\ast} \rbrace$,
\end{center} where
\begin{center}
$h_{1}^{\ast}(A) =
\left\{ \begin{array}{ll} [0,1]  \ & A = [0,1] \\ \lbrace0\rbrace \ & A\subset [0,1) \\ \lbrace1\rbrace \ & A=\lbrace1\rbrace \\ \lbrace 0,1 \rbrace \ & A \cap [0,1)\neq  \emptyset, A \cap \lbrace1\rbrace \neq \emptyset  \end{array}\right.$

\vskip .3cm

$h_{2}^{\ast}(A)= \left\{ \begin{array}{ll}  [0,1]  \ & A = [0,1] \\ \lbrace1\rbrace \ & A\subset (0,1] \\ \lbrace0\rbrace \ & A=\lbrace0\rbrace \\ \lbrace 0,1 \rbrace \ & A \cap (0,1]\neq  \emptyset, A \cap \lbrace0\rbrace \neq \emptyset  \end{array}\right.$
\end{center}
Again $h_{1}^{\ast}$ and $h_{2}^{\ast}$ are idempotents and $E(2^X) \cong E(X)$.
\end{example}

In general, $E(X)$ and $E(2^X)$ need not to be conjugate. We discuss an example here. 

\begin{example} \label{EX 10.2} Recall Example \ref{EX 5.3} here:

Consider $X= \lbrace (r,\theta): 0\leq \theta \leq 2\pi,\ r\in \lbrace 1- \frac{1}{2^n}:n\in \N\rbrace  \cup \lbrace 0,1 \rbrace \rbrace $ and $f(r,\theta)=(r, \theta + 2\pi r)$.

For this cascade $(X,f)$, we see that each orbit closure is an equicontinuous system, in fact $f$ is identity on the circle $r=1$. Hence $(X,f)$ is distal  but not eqicontinuous. Since  $(2^X,f_*)$ is distal if and only if $(X,f)$ is equicontinuous \cite{AN}, $(2^X,f_*)$ will not be distal. Again since $(X,f)$ is distal if and only if $E(X)$ is a group, we get here that $E(X)$ is a group but $E(2^X)$ is not a group.
\end{example}

\begin{remark} We recall from \cite{ME}, for a metric hereditarily almost equicontinuous system $(X, T)$ its enveloping semigroup $E(X)$ is again a metrizable hereditarily almost equicontinuous system.

In example \ref{EX 10.2} $(X,f)$ is hereditary almost equicontinuous but $(2^{X}, f_{\ast})$ is not so.

Hence we deduce that $E(X)$ metrizable need not imply $E(2^X)$ metrizable. \end{remark}

\begin{theorem}
For an equicontinuous cascade $(X,f)$, point-open topology of $E(2^X)$ is equivalent to the compact-open topology of $E(X)$.
\end{theorem}
\begin{proof}
Since for an equicontinuous cascade $(X,f)$, $E(X)\subset \mathcal{C}(X,X)$ and $E(2^X)\subset \mathcal{C}(2^X,2^X)$. By Theorem \ref{TH 10.3}, $E(X)$ and $E(2^X)$ are conjugate. The compact subsets of $X$ are points of $2^X$. Therefore compact open topology of $E(X)$ is equivalent to the point open topology of $E(2^X)$.
\end{proof}

\vskip .5cm

\vskip .5cm

We note by the results in \cite{LPYZ} that $(2^X,f_*)$ is weakly rigid if and only if $(2^X,f_*)$ is uniformly rigid if and only if $(X,f)$ is uniformly rigid. So by Theorem \ref{2.1};

\begin{theorem} For a cascade $(X,f)$, $(E(2^X), f_*)$ is transitive if and only if the identity $e_*$ is not isolated in $E(2^X)$ if and only if $(X,f)$ is uniformly rigid.

And in this case $(E(X), f)$ is transitive. \end{theorem}

The cascade $(2^X,f_*)$ induces the cascade $(2^{2^X}, f_{**})$, and so by Theorem \ref{2.1}  and \cite{LPYZ} we have

\begin{corollary}  $E(2^X,f_*)$ is transitive if and only if $E(2^{2^X}, f_{**})$ is transitive. \end{corollary}

We have already discussed examples when $(E(X), f)$ will be transitive but $(E(2^X), f_*)$ will fail to be transitive, the examples where $(X,f)$ is weakly rigid but not uniformly rigid.

\begin{remark} We note that when $(X,f)$ is weakly mixing and uniformly rigid, $(2^X,f_*)$ is also weakly mixing and uniformly rigid. Hence $E(2^X,f_*)$ is transitive and $(2^X,f_*)$ is a factor of $E(2^X,f_*)$.

In [41] the existence of minimal weakly mixing but nonetheless uniformly rigid
dynamical systems is demonstrated.

  The enveloping semigroup of the `time one map' of a classical horocycle flow is weakly mixing \cite{AEB}. But this flow is not uniformly rigid. And so its induced flow will not be transitive.

\end{remark}

\begin{remark} We note that when $(2^X,f_*)$ is  mixing, $(X,f)$ is also  mixing. Thus $E(2^X,f_*)$ is mixing if and only if $(X,f)$ is both mixing and uniformly rigid. This is impossible.

Since for nonempty, open $U, V \subset X$ such that $\overline{U} \cap \overline{V} = \emptyset$, mixing gives an $N \in \N$ such that $f^n(U) \cap V \neq \emptyset$ for all $n \geq N$ whereas uniform rigidity gives a sequence $\{n_k\}$ in $\N$ such that $f^{n_k} \to$ identity uniformly, implying that there cannot exist any $u \in U$ such that $f^{n_k}(u) \in V$ for $n_k > N$.

Hence $E(2^X,f_*)$ can never be mixing.\end{remark}

\vspace{.5cm}

\section{Enveloping Semigroup of the $\emph{2}$-shift, its Subshifts and the Induced Shifts}

There are some important relations between the $\emph{2}$-shift  and it's induced system. For the system $(\emph{2}^{\Z}, \sigma)$ its induced system $(2^{\emph{2}^{\Z}},{\sigma}_*)$  is topological mixing and has a dense set of periodic sets.

\subsection{Enveloping Semigroup of $(\emph{2}^{\Z}, \sigma)$, its subshifts and of $(2^{\emph{2}^{\Z}},{\sigma}_*)$}

In case of the full $\emph{2}$-shift, $(\emph{2}^{\Z},\mathbb{Z})$, we can calculate it's enveloping semigroup and which comes out to be isomorphic to $\beta \Z$. Though this is a known result we include a brief proof for the sake of completion, and also since we use this method for computing enveloping semigroups of some subshifts. This is given as an exercise in \cite[p. 31, Ex. 1.25]{JOIN} and as in \cite[p. 496, Theorem 19.15]{HIN} in case of semigroups.

\begin{theorem} \label{betaz} $E( \emph{2}^{\Z},\sigma) = \beta \Z$. \end{theorem}

\begin{proof} First of all for every bi-infinite sequence $(x_{i})$ in $(\emph{2}^{\Z},\sigma)$, we have a set $A=\lbrace m\in \mathbb{Z}: x_{m}= 1\rbrace $ and on the other way for any $K\subseteq \mathbb{Z}$ there is a bi-infinite sequence $(x_{i})$ where
 $$ x_{i} = \left\{
                           \begin{array}{ll}
                             1, & \hbox{$i\in K $;} \\
                             0, & \hbox{$i\in \mathbb{Z} \smallsetminus K$.}
                           \end{array}
                         \right.$$

So, we can identify each bi-infinite sequence in $(\emph{2}^{\Z},\mathbb{Z})$ with a subset of $\Z$. Also  for each $m\in \Z,\ m\equiv h(m)=\lbrace A\subset \Z: m\in A\rbrace$ is an ultrafilter. Also $h(A) = \lbrace B \subset \Z: A \subset B\rbrace$ is an ultrafilter such that $A \in h(A)$.

Now let $(x_{i}) \approx A\subset \Z$ in $\emph{2}^{\Z}$ and $n\in \Z$,
\begin{align*}
n\cdot A= \sigma^n (x_{i})&= (x_{i-n})\\ &= \lbrace k-n \in \Z: k\in A\rbrace\\ &=\lbrace s\in \Z: s+n \in A\rbrace \\ &= \lbrace s\in \Z: A\in h(s+n)\rbrace \\ &= \lbrace s\in \Z: h(s+n)\in h(A)\rbrace \\ &= \lbrace s\in \Z: s+n\equiv h(s+n)\in h(A)\rbrace
\end{align*}
Now if we define the action $`*'$ of $\beta \Z$ on $(\emph{2}^{\Z}, \Z)$ as
$p\in \beta \Z$ and $A\in \emph{2}^{\Z}$,
\begin{align*}
p*A=\lbrace \gamma\in \Z: \gamma p\in h(A)\rbrace
\end{align*}
Then from the above discussion, this action is the extension of the usual shift action of $\Z$ on $\emph{2}^{\Z}$. So, by the Stone-$\check{C}$ech compactification, the natural map $\Phi_{\emph{2}^{\Z}}:\beta Z\rightarrow E(\emph{2}^{\Z})$ defined as $\Phi_{\emph{2}^{\Z}}(p):\emph{2}^{\Z}\rightarrow \emph{2}^{\Z}$ and $\Phi_{\emph{2}^{\Z}}(p)(A)= p*A$ is a homomorphism and $\Phi_{\emph{2}^{\Z}}(\beta \Z)=E(\emph{2}^{\Z})$. Our claim is that $\Phi_{\emph{2}^{\Z}}$ is injective also.

Suppose,
\begin{align*}
\Phi_{\emph{2}^{\Z}}(p)&= \Phi_{\emph{2}^{\Z}}(q)\\
\Phi_{\emph{2}^{\Z}}(p)(A)&= \Phi_{\emph{2}^{\Z}}(q)(A)\ \forall\ A \in \emph{2}^{\Z}\\
\lbrace \gamma\in \Z: \gamma p\in h(A)\rbrace &= \lbrace \eta \in \Z: \eta q\in h(A)\rbrace
\end{align*}
Now $0\in \Z$ is identity and $p\in \beta \Z$ is an ultrafilter, then
\begin{align*}
0p\equiv h(0)p &=\lbrace A\subset \Z: A*h(0)\in p\rbrace \\ &=\lbrace A\subset \Z: \lbrace s\in \Z: R_{s}^{-1}(A)\in h(0)\rbrace\in p\rbrace\\
&= \lbrace A\subset \Z: \lbrace s\in \Z: 0\in R_{s}^{-1}(A)\rbrace\in p\rbrace\\
&= \lbrace A\subset \Z: \lbrace s\in \Z: s\in A\rbrace\in p\rbrace\\
&= \lbrace A\subset \Z: A\in p\rbrace\\
&= p.
\end{align*}
So, if $A\in p$ means $p\in h(A)\Rightarrow 0p\in h(A)\Rightarrow 0\in p*A=\Phi_{\emph{2}^{\Z}}(p)(A)=\Phi_{\emph{2}^{\Z}}(q)(A)= q*A$

So, $0\in q*A\Rightarrow 0q\in h(A)\Rightarrow q\in h(A)\Rightarrow A\in q$

$\therefore\  p\subset q$ and similarly $q\subset
p,\Rightarrow p=q$. Hence $\Phi_{\emph{2}^{\Z}}$ is injective. So, $E(\emph{2}^{\Z})\cong \beta \Z$. \end{proof}

\begin{corollary} For the induced system $(2^{\emph{2}^{\Z}},{\sigma}_*)$, $E(2^{\emph{2}^{\Z}}) \cong \beta \Z$. \end{corollary}

\vskip .5cm

We look into the ideas in \cite{BGH}. By Theorem \ref{boy}, any irreducible subshift of finite type $X$ with topological entropy greater than $\log 2$ factors onto the $\emph{2}-$shift  \cite{B}. Hence $E(X) = \beta \Z$.

 A cartesian product of mixing subshift of finite type $Y$ is also a mixing subshift of finite type.  Thus, there is a $K \in \N$ such that the Cartesian product $Y^k$ has the full $\emph{2}-$shift as a factor, for all $k \geq K$  and so the enveloping semigroup $E(Y^k) \cong \beta \Z$, for all $k \geq K$.

 Now $E(Y)$ is a factor of $E(Y^k)$, and also $ \Delta E(Y)^k \cong E(Y^k)$. Thus,

 $$ \beta \Z \cong E(Y^k) \cong \Delta E(Y)^k  \mapsto E(Y), \ \forall k \geq K.$$

 Since the natural factor from $\Delta E(Y)^k \mapsto E(Y)$ is the projection, and the product of $\beta \Z$ with itself just gives the product of ultrafilters, which again are ultrafilters, we must have $E(Y) = \beta \Z$. This proves that:

\begin{theorem} \label{msft} For any mixing subshift of finite type $Y$, we have $E(Y) = \beta \Z$. \end{theorem}

\vskip .5cm

Recalling a result from \cite{AN}: If $(X,\sigma)$ is a mixing subshift of finite type then it is not isomorphic to $(2^Y,g_*)$ for any compact system $(Y,g)$. In particular for a mixing subshift of finite type $(X, \sigma)$, $(2^X, \sigma_*)$ will not be a subshift of finite type. Yet $E(2^X, \sigma_*) = \beta \Z$.

Here, we note that  it would be interesting to isolate the properties of those systems $(Y,g)$ which are isomorphic to
$(2^X, \sigma_*)$  for some  mixing subshift of finite type $(X,\sigma)$.

 \vskip .5cm

 The Golden Mean Shift $\mathcal{G}$ is a mixing subshift of finite type  and so has $\beta \Z$ as its enveloping semigroup. Let $X$ be the Even Shift. Then $X$ is a factor of the Golden Mean Shift $\mathcal{G}$,  and so $E(X)$ will be a  factor of $E(\mathcal{G})$.  Since $ E(\mathcal{G}) = \beta \Z$,  a natural question here is { \it ``What factor of $\beta \Z$ will be the enveloping semigroup of the Even Shift?''} We thank \emph{\textbf{Dona Strauss}} for help with this following theorems.

 \begin{theorem} The Enveloping semigroup of the Even Shift $X$, $E(X) = \beta \Z$. \end{theorem}

 \begin{proof} We note that $\Phi_X: \beta\Z \to E(X)$ is a homomorphism. We recall the proof of Theorem \ref{betaz} and build up on that.

Let $p \neq q \in \beta \Z$. Then there is a $P \subset \Z$ such that $P \in p$ but $P \notin q$. Without loss of generality we can assume that $0 \in P$ and $a \cong b \ (mod \ 4), \ \forall a,b \in P$.

Define $(x_i) \in \emph{2}^{\Z}$, by taking $x_i = \left\{
                                         \begin{array}{ll}
                                           1, & \hbox{if $i \in P \cup P+1$;} \\
                                           0, & \hbox{otherwise.}
                                         \end{array}
                                       \right.$

We note that in $(x_i)$ there are even no. of $0's$ between any two occurrences of $1$ and so $(x_i) \in X$.

Then for $m \in P$,  $\sigma^m((x_i))_0 = 1$ and
$\sigma^m((x_i))_1 = 1$. But if $m \notin P$, then this is not true.

Now $P \in p \Rightarrow p \in h(P) \Rightarrow 0, 1 \in p \ast P = \Phi_X(p)(P)$. But $0,1 \notin q \ast P = \Phi_X(q)(P)$.

Thus, $\Phi_X(p) \neq \Phi_X(q)$ and so $\Phi_X$ is injective.

 \end{proof}

 This gives another proof to the fact that the enveloping semigroup of the Golden Mean Shift is $\beta \Z$. Also by the discussions above recalling Theorem \ref{boy} and noting that the product of mixing sofic shifts will be mixing sofic, we can say that the enveloping semigroup of any mixing sofic shift will be $\beta \Z$.

 \vskip .5cm

 Similarly let $Z \subset \{0,1\}^{\Z}$ is defined as

 $ Z = \{ (x_i) : $ for some fixed $K \in \N$, between any two occurrences of $1's$ in $(x_i)$ there are $K \Z_+$ number of $0's \}$.

 So all blocks of the form $10^{Kn+j}1, j =1, \ldots, K-1$ are forbidden in any $x \in Z$ for $n \in \Z_+$.

 \begin{theorem} The Enveloping semigroup of the subshift $Z$, $E(Z) = \beta \Z$. \end{theorem}

 \begin{proof} We can prove this exactly on the lines of the previous theorem.

Let $p \neq q \in \beta \Z$. Then there is a $P \subset \Z$ such that $P \in p$ but $P \notin q$. Without loss of generality we can assume that $0 \in P$ and $a \cong b \ (\mod \ 2K), \ \forall a,b \in P$.

Define $(x_i) \in \{0,1\}^{\Z}$, by taking $x_i = \left\{
                                         \begin{array}{ll}
                                           1, & \hbox{if $i \in P \cup (P+1) \cup (P+2) \cup \ldots \cup (P + K-1)$;} \\
                                           0, & \hbox{otherwise.}
                                         \end{array}
                                       \right.$

We note that in $(x_i)$ there are $K \Z_+$ no. of $0's$ between any two occurrences of $1$ and so $(x_i) \in Z$.

Then for $m \in P$,  $\sigma^m((x_i))_0 = 1$, $\sigma^m((x_i))_1 = 1$, \ldots,
$\sigma^m((x_i))_{K-1} = 1$. But if $m \notin P$, then this is not true.

Now $P \in p \Rightarrow p \in h(P) \Rightarrow 0, 1, \ldots, K-1 \in p \ast P = \Phi_Z(p)(P)$. But $0,1, \ldots, K-1 \notin q \ast P = \Phi_Z(q)(P)$.

Thus, $\Phi_Z(p) \neq \Phi_Z(q)$ and so $\Phi_Z$ is injective.

 \end{proof}

 For any $S \subset \N$, define $X_S = \{\{x_i\}  : x_i = x_j = 1 \Leftrightarrow |i - j| \in S \cup \{0\}\} \subset \emph{2}^{\Z}$. It can be easily seen that $X_S$ is a subshift of the $\emph{2-}$shift, and is called the \emph{spacing subshift}. The above theorem shows that $E(X_S) = \beta \Z$ for specific $S$. However, for a general $S \subset \N$, it would be challenging to determine $E(X_S)$.

 \begin{remark} This observation leads to interesting questions:

 1. Can we characterize all subshifts with enveloping semigroups $\beta \Z$? 

 2. Since we just saw an example of a proper factor $\pi:A \to B$ for which the factor $\theta: E(A) \to E(B)$ is an isomorphism. Can we characterize all the properties of the shift space $A$ that $B$ too inherits so that $E(A) \cong E(B)$?

 \end{remark}

 \vskip .5cm

 Now the period doubling system on $\mathbb{S}^1$ is a factor of the $\emph{2}-$shift. Any such factor map will fail to be one-one only on a set of measure zero. Thus $\beta \Z$ factors onto $E(\mathbb{S}^1)$. Also the period doubling system on $\mathbb{S}^1$ is not weakly rigid and so will not have a transitive enveloping semigroup. \textbf{ We conjecture: $E(\mathbb{S}^1) = \beta \Z$.}

 \vskip .5cm

Now the interesting thing is to look at the fact when we change the action of $\Z$ on $\emph{2}^{\Z}$ and take $2\Z$ as our acting group, i.e. we consider the map $\sigma^2$ on $\emph{2}^{\Z}$.

\subsection{Enveloping semigroup of $(\emph{2}^{\Z}, 2\Z)$:} If we proceed as above in this case then we will have the extension action $'\ast'$ of $\beta 2\Z$ on $(\emph{2}^{\Z}, 2\Z)$ as for $p\in \beta2\Z$ and $A\in \emph{2}^{\Z}$
\begin{align*}
p*A=\lbrace 2\gamma\in 2\Z: 2\gamma p\in h(A)\rbrace
\end{align*}
Then by the the same procedure, the natural map $\Phi_{2^{\Z}}: \beta2\Z \rightarrow E(2^{\Z},2\Z)$ defined as $\Phi_{2^{\Z}}(p)(A)= p\ast A$.

Which gives $\Phi_{2^{\Z}}(\beta2\Z)= E(2^{\Z},2\Z)$.

The same procedure as above will shows that this is one-one also.

Suppose,
\begin{align*}
\Phi_{2^{\Z}}(p)&= \Phi_{2^{\Z}}(q)\\
\Phi_{2^{\Z}}(p)(A)&= \Phi_{2^{\Z}}(q)(A)\ \forall\ A \in 2^{\Z}\\
\lbrace 2\gamma\in 2\Z: 2\gamma p\in h(A)\rbrace &= \lbrace 2\eta \in 2\Z: 2\eta q\in h(A)\rbrace
\end{align*}
Now $0\in 2\Z$ is identity and $p\in \beta 2\Z$ is an ultrafilter, then
\begin{align*}
0p\equiv h(0)p &=\lbrace A\subset 2\Z: A*h(0)\in p\rbrace \\ &=\lbrace A\subset 2\Z: \lbrace 2s\in 2\Z: R_{s}^{-1}(A)\in h(0)\rbrace\in p\rbrace\\
&= \lbrace A\subset 2\Z: \lbrace 2s\in 2\Z: 0\in R_{s}^{-1}(A)\rbrace\in p\rbrace\\
&= \lbrace A\subset 2\Z: \lbrace 2s\in 2\Z: 2s\in A\rbrace\in p\rbrace\\
&= \lbrace A\subset 2\Z: A\in p\rbrace\\
&= p.
\end{align*}
So, if $A\in p$ means $p\in h(A)\Rightarrow 0p\in h(A)\Rightarrow\ 2\cdot0p\in h(A) \Rightarrow 0\in p*A=\Phi_{\emph{2}^{\Z}}(p)(A)=\Phi_{\emph{2}^{\Z}}(q)(A)= q*A$.

So, $2\cdot0\in q*A\Rightarrow 2\cdot0q\in h(A)\Rightarrow q\in h(A)\Rightarrow A\in q$.

Thence $\  p\subset q$ and similarly $q\subset
p,\Rightarrow p=q$. Hence $\Phi_{\emph{2}^{\Z}}$ is injective.
Hence  $E(\emph{2}^{\Z},2\Z)\cong \beta 2\Z$. In general for any $n\in \N$, $E(2^{\Z}, n\Z)\cong \beta n\Z$.

\subsection{Isomorphism of $\beta \Z$ and $\beta 2\Z$}

The following result about the isomorphism of $\beta \Z$ and $\beta 2\Z$ is given in \cite{HIN}[Lemma 3.30, Ex. 3.4.1]. We give a detailed proof here;

\begin{theorem} $\beta\Z$ and $\beta 2\Z$ are isomorphic as a topological space, as a flow and also as a semigroup.
\end{theorem}

\begin{proof}
First of all $\Z$ and $2\Z$ are isomorphic as topological groups under the map $f: \Z \rightarrow 2\Z$ as $f(n)= 2n$ and $\beta \Z$ and $\beta 2\Z$ are Stone-$\check{C}$ech compactifications of $\Z$ and $2\Z$ respectively. So $i_{\Z}: \Z \rightarrow \beta\Z $ and $i_{2\Z}:2\Z \rightarrow \beta 2\Z$ are the natural maps defined as $i_{\Z}(n)= h(n)$ and $i_{2\Z}(2n)= h(2n)$ respectively, where $n\equiv h(n)$ identified as an ultrafilter over $\Z$ and the maps look as inclusions. So, by taking composition we have that $i_{2\Z}\circ f=g:\Z \rightarrow \beta 2\Z$ as $g(n)= h(2n)$ is a continuous map. So, by using the Stone-$\check{C}$ech compactifications $\beta 2\Z$ of $2\Z$ this map has a continuous extension $\Phi: \beta 2\Z \rightarrow \beta \Z$ defined as $\Phi(p)=\lbrace f(A): A \in p \rbrace$.

\textbf{$\Phi$ is well defined}: Suppose $p$ is an ultrafilter on $\Z$ and we have to show that $\Phi(p)= \lbrace f(A): A \in p \rbrace$ is an ultrafilter on $2\Z$.

if $f(A), f(B) \in \Phi(p)$ where $A, B \in p$ then $f(A)\cap f(B) \subseteq f(A\cap B)$. Since $\emptyset \notin p\ \Rightarrow \emptyset \notin \Phi(p)$. For $A,B\in p$, so $A\cap B \neq \emptyset\ \Rightarrow f(A\cap B)\neq \emptyset$. So, $f(A)\cap f(B)\neq \emptyset$.

Also let $f(A)\in \Phi(p)$ and $B\subset \beta 2\Z$ such that $f(A)\subset B$. Here $A\in p$ and since $f$ is a bijection then $A\subset f^{-1}(B)$ where $f^{-1}(B)\subset \beta \Z$.

As $A\subset f^{-1}(B)\ \Rightarrow f^{-1}(B)\in p\ \Rightarrow f(f^{-1}(B))\in \Phi(p)\ \Rightarrow B\in \Phi(p)$. So, $\Phi(p)$ is a filter on $\beta 2\Z$.
Now, if $K \notin \Phi(p)\ \Rightarrow f^{-1}(F)\notin p$. Since $p$ is an ultrafilter $\Rightarrow\ \Z\setminus f^{-1}(F)\in p$.

Since $f$ is bijective. So, $\Z\setminus f^{-1}(F)= f^{-1}(2\Z \setminus F)$. Therefore $f^{-1}(2\Z \setminus F) \in p\ \Rightarrow f(f^{-1}(2\Z \setminus F))\in \Phi(p)\ \Rightarrow 2\Z \setminus F\in \Phi(p)$, which shows that $\Phi(p)$ is an ultrafilter.
Also if $p=q$ then by the definition of $\Phi,\ \Phi(p)=\Phi(q)$. So, $\Phi$ is a well defined map.

\textbf{$\Phi$ is one-one}: Suppose $p\neq q$ then $\exists$ $A\in p$ and $A\notin q$. Since $q$ is an ultrafilter $\Rightarrow \Z \setminus A \in q$. So, $f(A)\in \Phi(p)$ and $f(\Z\setminus A)\in \Phi(q)$.

Again the bijection of $f$ shows that $2\Z\setminus f(A)\in \Phi(q)$. Since $\Phi(q)$ is also an ultrafilter $\Rightarrow\ f(A)\notin \Phi(q)$. Which shows that $\Phi(p)\neq \Phi(q)$. Therefore $\Phi$ is one-one.

\textbf{$\Phi$ is onto}: Suppose $q\in \beta
2\Z$ is an ultrafilter over $2\Z$. Then consider the set $\lbrace f^{-1}(B): B\in q\rbrace$. Since $f$ is bijection then in the similar way as we did above for $f,\ \lbrace f^{-1}(B): B\in q\rbrace$ is an ultrafilter over $\Z$ and if we take $p= \lbrace f^{-1}(B): B\in q\rbrace$ then $\Phi(p)=q$, which shows that $\Phi$ is onto also.

Hence $\Phi$ is a continuous bijection between $\beta \Z$ and $\beta2\Z$ and since both spaces are compact, they are topologically homeomorphic.

Since $\beta\Z$ and $\beta2\Z$ are semingroups also. So, we will prove that $\Phi$ is semigroup homomorphism also.

\textbf{$\Phi$ is semigroup homomorphism}: To prove $\Phi(pq)=\Phi(p)\Phi(q)$.

Let \begin{align*}
A\in \Phi(pq)\ & \Leftrightarrow\ f^{-1}(A)\in pq \\
& \Leftrightarrow \lbrace s\in \Z: f^{-1}(A)-s \in p \rbrace \in q \\
& \Leftrightarrow \lbrace s\in \Z: \lbrace m \in \Z: m+s\in f^{-1}(A)\rbrace \in p \rbrace \in q\\
& \Leftrightarrow \lbrace s\in \Z: \lbrace m\in \Z: f(m+s)\in A\rbrace \in p \rbrace \in q\\
& \Leftrightarrow \lbrace s\in \Z: \lbrace m\in \Z: 2m+2s\in A\rbrace \in p \rbrace \in q \\
& \Leftrightarrow \lbrace 2s\in 2\Z: \lbrace 2m\in 2\Z: 2m+2s\in A\rbrace \in \Phi(p) \rbrace \in \Phi(q)\\
& \Leftrightarrow A \in \Phi(p)\Phi(q).
\end{align*}
Which shows that $\Phi(pq)=\Phi(p)\Phi(q)$.
\end{proof}

\begin{corollary} For $n \in \N$, $\beta\Z$ and $\beta n\Z$ are isomorphic as a topological space, as a flow and also as a semigroup. \end{corollary}

\begin{theorem} For the $\emph{2}-$shift, we have $E(\emph{2}^{\Z}, \sigma) \cong E(\emph{2}^{\Z}, \sigma^2) \cong E(\emph{2}^{\Z}, \sigma^3) \cong \ldots$ \end{theorem}

\begin{corollary} For any mixing subshift of finite type or mixing sofic shift $X$, we have $E(X, \sigma) \cong E(X, \sigma^2) \cong E(X, \sigma^3) \cong \ldots$ \end{corollary}

\section{Saturated Enveloping Semigroups}

We look into an elementary observation:

\begin{proposition}
Let $(X,f)$ be a cascade and $f$ be surjective on $E(X)$, then
$$E(X,f)=E(X,f^n)\cup f(E(X,f^n))\cup f^2(E(X,f^n))\cup f^3(E(X,f^n))\cup \ldots \cup f^{n-1}(E(X,f^n))$$

 $\forall\ n\in \mathbb{N}$, considered as a semigroup.
\end{proposition}
\begin{proof}
Since $E(X,f^n)=\overline{\lbrace \ldots, f^{-2n}, f^{-n}, e, f^n, f^{2n},\ldots, f^{kn},\ldots\rbrace}\subseteq E(X,f)$

and so for $1 \leq m \leq n-1$, $f^m(\overline{\lbrace \ldots, f^{-2n}, f^{-n}, e, f^n, f^{2n},\cdots, f^{kn},\cdots\rbrace})$

  $= \overline{\lbrace \ldots, f^{m-2n}, f^{m-n}, f^m, f^{m+n},f^{m+2n},\cdots, f^{m+kn},\cdots \rbrace}\subseteq E(X,f)$.

  So, $E(X,f^n)\cup f(E(X,f^n))\cup f^2(E(X,f^n))\cup f^3(E(X,f^n))\cup \cdots \cup f^{n-1}(E(X,f^n))\subseteq E(X,f)$.

On the other way, fix $n \in \N$ and let $p\in E(X,f)$ then there is a sequence $\lbrace n_k \rbrace$ such that
$f^{n_k}\longrightarrow p$.

Now consider the set $S_1 = \lbrace k\in \N:n_k\in n\mathbb{Z}\rbrace$.

\textbf{Case-1:} if for every $k\in \N,$ there is an $l\in \N$ with $l\geq k$ such that $n_{l}\in n\Z$. So, for every $k\in \N$ there is an $l\in \N$ with $l\geq k$ such that $l \in S_1$.

Therefore $\lbrace n_{k_l} : k_{l} \in S_1 \rbrace$ is a subsequence of $\lbrace n_{k} \rbrace$ such that $f^{n_{k_l}}$ converges to p. So, $p \in E(X,f^n)$.

\textbf{Case-2:} Suppose there is a $k_0 \in \N$ such that $\forall\ k\geq k_0, \ n_k \in \lbrace n\Z+1\rbrace \cup \lbrace n \Z+2 \rbrace \cup \ldots \cup \lbrace n\Z+(n-1)\rbrace$, i.e. $\forall \ k \geq k_0, \ n_k \notin n\Z$.

Now if for every $k\in \N$ there is an $l\in \N$ as $l\geq k$ such that $n_l\in \lbrace n\Z+1\rbrace$. Then $S_2=\lbrace k\in \N: k\geq k_0\ \texttt{and}\ n_k \in \lbrace n\Z +1 \rbrace \rbrace$.

Again as the last case, $\lbrace n_{k_m} : k_{m} \in S_2 \rbrace$ is a subsequence of $\lbrace n_{k} \rbrace$ such that $f^{n_{k_m}}$ converges to $p$. So, $p\in \overline{\lbrace \ldots, f^{-2n+1}, f^{-n+1}, f, f^{n+1},f^{2n+1},\ldots,f^{ln+1},\ldots\rbrace}$

$ = f(\overline{\lbrace \ldots, f^{-2n}, f^{-n}, e, f^{n},f^{2n},\ldots,f^{ln},\ldots\rbrace})=f(E(X,f^n))$.

If there is $k_1$ such that $\forall\ k\geq k_1,\ n_k \in \lbrace n \Z +2 \rbrace \cup \lbrace n \Z +3 \rbrace \ldots \cup \lbrace n\Z +(n-1)\rbrace$. Again by preceding in similar way as above, $p\in f^{r}(E(X,f^{n}))$ for some $r\in \{ 1,2, \ldots, n-1\}$. Therefore $p\in E(X,f^n)\cup f(E(X,f^n))\cup f^2(E(X,f^n))\cup f^3(E(X,f^n))\cup \ldots \cup f^{n-1}(E(X,f^n))$.

Hence $E(X,f)= E(X,f^n)\cup f(E(X,f^n))\cup f^2(E(X,f^n))\cup f^3(E(X,f^n))\cup \ldots \cup f^{n-1}(E(X,f^n))$.
\end{proof}

\begin{remark} Note that $E(X,f)= E(X,f^n)\cup f(E(X,f^n))\cup f^2(E(X,f^n))\cup f^3(E(X,f^n))\cup \ldots \cup f^{n-1}(E(X,f^n))$ as a semigroup. Thus, $E(X,f^n) \subseteq E(X,f)$ as a semigroup. However, we cannot claim any relation between the systems $(E(X,f^n), f^n)$ and $(E(X,f),f)$.

\vskip .5cm

For example consider the irrational rotation $(\mathbb{S}^1, T_\alpha)$, where $T_\alpha$ is the rotation by angle $\alpha$. Then $T^n_\alpha = T_{n\alpha}$. Hence $(E(\mathbb{S}^1, T_\alpha), T_\alpha) \cong (\mathbb{S}^1, T_\alpha)$ and $(E(\mathbb{S}^1, T_{n\alpha}), T_{n\alpha}) \cong (\mathbb{S}^1, T_{n\alpha})$.

\vskip .5cm

And so $(E(\mathbb{S}^1, T_\alpha), T_\alpha) \ncong (E(\mathbb{S}^1, T_{n\alpha}), T_{n\alpha})$. \end{remark}

\vskip .5cm

Then, can we expect any relation between such systems here?

\begin{definition}
For a flow $(X,T)$ or a cascade $(X,f)$, the enveloping semigroup $E(X)$ is called \emph{saturated enveloping semigroup}\ (SES) if $E(X)\cong E(2^X)$ both as a semigroup as as a flow.
\end{definition}

\begin{definition}
For a cascade $(X,f)$, the enveloping semigroup $(E(X),f)$ is called super saturated if $E(X,f)\cong E(X,f^2)\cong E(X,f^3)\cong \ldots \cong E(X,f^n)\cong\ldots \cong E(2^X, f_{\ast})$ for each $n\in \N$, both as a semigroup and as a flow.
\end{definition}

 Note that every equicontinuous system has saturated enveloping semigroup and the simplest example of saturated enveloping semigroup is the irrational rotation on the unit circle $\mathbb{S}^{1}$.

 \vskip .5cm

\begin{example}
Since we have shown that in case of $\emph{2}$-shift, $\beta\Z \cong E(\emph{2}^{\Z},\sigma)\cong E(\emph{2}^{\Z},\sigma^{2})\cong E(\emph{2}^{\Z},\sigma^{3})\cong \cdots \cong E(\emph{2}^{\Z},\sigma^{n})\cong \cdots \cong E(2^{\emph{2}^\mathbb{Z}}, \sigma)$ as a semigroup as well as a flow. So, $E(\emph{2}^\mathbb{Z},\mathbb{Z})$ is a super  saturated enveloping semigroup.
\end{example}
\begin{example}
$\mathbb{S}^{1}$ is the unit circle and $\alpha$ is an irrational number the $T_\alpha:\mathbb{S}^{1}\rightarrow \mathbb{S}^{1}$ is an irrational rotation as $T_\alpha(e^{i\theta})=e^{i(\theta+2\pi \alpha})$. The system $(\mathbb{S}^{1},T_\alpha)$ is an equicontinuous minimal system. So  $E(\mathbb{S}^{1},T_\alpha)\cong (\mathbb{S}^{1},T_\alpha)$. Again, $E(2^{\mathbb{S}^1}, {T_\alpha}_*) \cong E(\mathbb{S}^{1},T_\alpha)$. So $E(\mathbb{S}^{1},T_\alpha)$ is saturated. Now if we consider the system $(\mathbb{S}^{1},T_\alpha^2)$ then this will be the irrational rotation by angle $2\alpha$ and so $(\mathbb{S}^{1},T_\alpha)\ncong (\mathbb{S}^{1},T_\alpha^2)$. Since $(\mathbb{S}^{1},T_\alpha^2)$ is also an equicontinuous minimal system, $E(\mathbb{S}^{1},T_\alpha^2)\cong (\mathbb{S}^{1},T_\alpha^2)$. So, $E(\mathbb{S}^{1},T_\alpha^{2})\ncong E(\mathbb{S}^{1},T_\alpha)$. Therefore the enveloping semigroup of an irrational rotation on unit circle is not super saturated.
\end{example}

Glasner \cite{STM} has defined a \emph{complete system}.  The system $(X, T)$ is \emph{$n-$complete} if for every point $(x_1, \ldots , x_n) \in X^n$ with distinct components the orbit $T(x_1, \ldots , x_n)$ is dense in $X^n$. It is called \emph{complete} when it is $n-$complete for every $n \in N$.

\begin{theorem} \cite{STM} Let $(X, T)$ be a dynamical system. Then $E(X, T) = X^X$ if and only
if $(X, T)$ is complete. \end{theorem}

We see that there is a connection between complete systems and those with saturated enveloping semigroups.

Recall that  $\phi: T \mapsto X^X$ has a continuous extension $\Phi_{X}:\beta T \rightarrow X^X$ and $\Phi_{X}(\beta T)= E(X)$ is a subgroup of $X^X$. When $E(X) = X^X$, it means that $\Phi_X$ is surjective. Can $\Phi_X$ be also injective? We note that we have examples when $E(X) = \beta \Z$, and such systems are always saturated and in this case $\Phi_X$ is injective.

\vskip .5cm

In \cite{TM} Glasner observed: If $\phi$ is a nontrivial continuous automorphism of a system $(X, T)$ then $\phi p = p \phi$
for every $p \in  E(X)$. Thus when the group $Aut (X, T)$ is nontrivial then
$E \subset \{p \in X^X : \phi p = p \phi, \ \forall \phi \in Aut (X, T)\}$. In particular, when $T$ is commutative
$$E \subset \{p \in X^X : p \alpha = \alpha p, \ \forall \alpha \in T \}$$
He raised a simple question: Are there dynamical systems $(X, T)$ for which this inclusion is an equality? And further proved that:

\begin{theorem} \cite{TM} There does not exist an infinite minimal cascade $(X, f)$ for which
$$E(X, T) = \{p \in X^X : p f = f p\}$$
\end{theorem}

We have a simple question here: \textbf{What can be said about this  problem for a non-minimal cascade?
}

\section{Enveloping Semigroup for Semiflows}

We acknowledge the work  done on topological dynamics  for semiflows in
\cite{AUSD}. Generally, dynamical properties of semiflows are much different from that of flows. This motivates us to study enveloping semigroups of semiflows also. We present here our version for the enveloping semigroups
for semiflows.

\vskip .5cm

In this section, we consider the action by monoids instead of
  groups. So, we consider the semiflow $(X,S)$ where $X$ is a compact
metric space and $S$ is a
monoid acting on $X$. We can also consider the semicascade $(X,g)$ where
$g$ is continuous though not necessarily surjective or injective. We
investige the dynamical properties of $E(X)$ in such a case and see how
they differ from the case when the action is  by a group.

Here the interesting  point to note is that when we consider the action by a group and by it's subgroup then the enveloping semigroup of the given system is not the same.

We look into a few examples of Enveloping semigroups for semiflows(semicascades):

\begin{example}  $E( \emph{2}^{\N},\sigma) = \beta \N$.

The proof of for this  works same as  in Theorem \ref{betaz} by replacing $\N$ by $\Z$. \end{example}

\begin{example} Recall  Example 6.3. We modify it a bit to let $n \in \N$ and $X_{n} = \lbrace n \rbrace \times \N^{\ast} \times \lbrace 1,2,\ldots, n\rbrace$ where $\N^{\ast}= \N \cup \lbrace \infty \rbrace$ is the one point compactification of $\N$, and $f_{n}: X_{n} \rightarrow X_{n}$ defined as $f_{n}(n, k, l )=(n, k+1, (l+1) \mod n )$. In the cascade $(X_n, f_n)$, $f_{n}^{ni}(n, k, l)= (n, k+ni, l) \ \forall i$. So for each $(n, k, l )\in X_n$, $f_{n}^{ni}(n, k, l)\longrightarrow (n, \infty, l)$. So $E(X_n)= \lbrace f_{n}^{k}: k\in \Z\rbrace \cup \lbrace p_n, f_np_n, \ldots, f_n^{n-1}p_n \rbrace$ where $p_{n}(n, k, l)= (n, \infty, l)$. Also $f_{n}^{n}p_{n}(n, k, l)= f_{n}^{n}(n, \infty, l)=(n, \infty, l)\Rightarrow f_{n}^{n}p_{n}= p_{n}$. Therefore $p_{n}$ is periodic point of period $n$. Also $p_{n}p_{n}=p_{n}$. Therefore for each $n\in \N$, there is a cascade $(X_{n}, f_{n})$ such that $E(X_{n})$ contains a periodic point of period $n$ which is an idempotent. \end{example}

\color{black}

\begin{example} We recall  Example \ref{EX 5.1} where $X=[0,1]$ and $f: [0,1]\rightarrow [0,1]$ as $f(x)=x^2$. When we take the action by  semigroup $\N$ then
 $E(X)$ is the one point compactification of $\N$.

 Here we can see that the function is invertible and we can take backward orbits as well. But if we are interested in taking only the forward orbits  then  the structure of the enveloping semigroup will change.
\end{example}

 \begin{remark}
It is important to note that all the theory for flows $(X,T)$ discussed in Section 5 above, is also valid for the semiflows $(X,S)$.
\end{remark}

\subsection{Distality of $E(X,S)$}

If each $s \in S$ is surjective,  then $(X, S)$ is called a \emph{surjective semiflow}. If each $s \in S$ is bijective, then \emph{$[S]$} denotes the smallest group of self-homeomorphisms of $X$ containing $S$. In such a case the semiflow $(X, S)$ induces the flow $(X,[S])$. A point $x \in X$ is called a \emph{distal point} of $(X,S)$ if $x$ is the only point proximal to itself in $\overline{Sx}$.

We note that

\begin{theorem} \cite{AUSD}
If $(X,S)$ is an equicontinuous surjective semiflow, then $(X,S)$ is distal.
\end{theorem}

\vskip .5cm

We know that if a flow $(X,T)$ is distal then $E(X,T)$ is  a group. Now let $S$  be a subsemigroup of $T$, and consider the semiflow $(X,S)$. Will $E(X,S)$ be also a group wherein inverses of $s\in S$  lie in $E(X) \setminus \lbrace s: s\in S\rbrace$?

\vskip.5cm

We recall Ellis' theorem.

\begin{theorem} \textbf{[Ellis Theorem]} \cite{REL}
Let $(X,T)$ be a distal flow. Then $(X,T)$ is pointwise almost periodic.
\end{theorem}

We note that Ellis' theorem is true (in general) for semiflows also. We include a proof here given to us by \textbf{Auslander}. For that we need the following definition and properties.

\begin{definition}\cite{AEAP}
For a semiflow $(X,S)$, a set $A\subset X$ is said to be an \emph{almost periodic set} if any point $z\in X^{|A|}$ with
$range(z)=A$ is an almost periodic point of $(X^{|A|},S)$, where $|A|$ is the cardinality of $A$.
\end{definition}

\begin{lemma}\cite{AEAP} \label{app}
\begin{enumerate}
\item Any non-empty subset of an almost periodic set is an almost periodic
set. (In particular, every point of an almost periodic set is almost periodic.)
\item Let $A$ be a maximal almost periodic set for $(X,S)$, and let $x\in X$. Then there is an $x^{\prime}\in A \cap \overline{Sx}$ such that $x$ and $x^{\prime}$ are proximal.
\item Let $\pi: (X,S)\longrightarrow (Y,S)$ be an onto
flow homomorphism, and let $B$ be an
almost periodic set in Y. Then there is an almost periodic set $A$ in X such that $\pi(A) = B$.
\item If A is a maximal almost periodic set in X, then $\pi(A)$ is a maximal almost periodic set in $Y$.
\item Suppose $range(z)$ is a maximal almost periodic set. Let $z^{\prime} \in \overline{Sz}$. Then $range(z^{\prime})$ is a maximal almost periodic set.

\end{enumerate}
\end{lemma}

\color{black}

\begin{theorem}\textbf{[Generalized Ellis Theorem]}
Let $(X,S)$ be a distal semiflow. Then $(X,S)$ is pointwise almost periodic.
\end{theorem}

\begin{proof}
 $(X,S)$ is a distal semiflow.

If $x\in X$ is an almost periodic point, then  by Zorn's Lemma  there is a maximal almost periodic set in $X$ which contains $x$. Also every point of an almost periodic set is almost periodic.

Let $A\subset X$ be a  maximal almost periodic set and assume that $A\neq X$. Let $z\in X^{|A|}$ with $range(z)=A$. Then $z$ is the almost periodic point of $(X^{|A|},S)$. Let $x\in X\setminus A$ and consider the point $(x,z)\in X\times X^{|A|}$. The orbit closure $\overline {S(x,z)}$ is a closed invariant subset of $X\times X^{|A|}$ and so there exists a minimal $M\subset \overline {S(x,z)}\subset X \times X^{|A|}$. Let $(x^{\prime},z)\in M$ then $\overline {S(x^{\prime},z)}=M$ being minimal,  $(x^{\prime},z)$ is an almost periodic point in $ X\times X^{|A|}$. Since $z$ is an almost periodic point of $X^{|A|}$ with $range(z)=A$ and $(x^{\prime},z)\in \overline {S(x,z)}$, there is a net $\lbrace s_{i}\rbrace$ such that $s_{i}(x,z)\longrightarrow (x^{\prime},z)$. Now since $\overline {S(x^{\prime},z)} = M$ is a minimal set in $X\times X^{|A|}$, so by Lemma \ref{app} $A \cup \lbrace x^{\prime}\rbrace$ is an almost periodic set of $(X,S)$. But  this contradicts the maximality of the almost periodic set $A$,  and so $x^{\prime}\in A$.

Now we have $s_{i}x\longrightarrow x^{\prime}$ and $s_{i}z\longrightarrow z$. Here $range(z)=A$ and $x^{\prime}\in A$, and take $\pi_{x^{\prime}}$ to be the projection map so that $s_{i}x^{\prime}=\pi_{x^{\prime}}(s_{i}z)\longrightarrow \pi_{x^{\prime}} z=x^{\prime}$. Hence $s_{i}(x,x^{\prime})\longrightarrow (x^{\prime},x^{\prime})$ giving $x$ and $x^{\prime}$ to be proximal, which violates the distality of $(X,S)$. So our assumption $A\neq X$ is false. Therefore $X=A$ is an almost periodic set, i.e. every point of $X$ is almost periodic point.
\end{proof}

\begin{corollary}
If $(X,S)$ is a distal semiflow. Then $(X,S)$ is surjective.
\end{corollary}
\begin{proof}
First note that if $(Y,f)$ be any distal semicascade then $(Y,f)$ will be pointwise almost periodic by the \emph{generalized Ellis Theorem}. So $Y=\bigsqcup_{i}N_{i}$, where $\bigsqcup$ denotes disjoint union and each $N_{i}$ is minimal subset of $Y$. Then for each $i$, $f(N_{i})\subset N_{i}$ is closed invariant and so by minimality of $N_{i}$, $f(N_{i})=N_{i}$. So $f$ is surjective on $Y$.

Now for any $s\in S$, consider the semicascade  $(X,s)$. Since $(X,S)$ is distal, as a subsemiflow, $(X,s)$ is also distal. By the above observation,  $s$ is surjective on $X$. So for each $s\in S$, $s$ is surjective on $X$. Hence $(X,S)$ is surjective semiflow.

\end{proof}

\begin{remark}
Distal systems are always  injective, and so from the above we note that distal systems are always bijective. Thus every distal semiflow $(X,S)$ induces a distal flow $(X,[S])$.
\end{remark}

This observation gives us another generalization:

\begin{theorem}
For a semiflow $(X,S)$ the following conditions are equivalent:
\begin{enumerate}
\item $(X,S)$ is distal.
\item $(E(X,S), S)$ is minimal.
\item $E(X,S)$ is a group with identity $e$ being the only idempotent.
\end{enumerate}

\end{theorem}
\begin{proof}
$(1)\Rightarrow (2)$ Since we know that for a  distal flow $(X,T)$, $E(X)$ is the only minimal ideal and $e$\ (identity) is the only idempotent.  Since $(X,S)$ is distal, it induces the distal flow $(X,[S])$. Now $E(X,[S])$ has a  unique idempotent $e$. Hence,  the identity $e$ is the unique idempotent in $E(X,S)$ and so $(E(X,S),S)$ is minimal.

$(2)\Rightarrow (3)$ We know that for a minimal ideal $I\subset E(X)$ and a minimal idempotent $u\in I$, $Iu$ is a group with identity $u$. Since $(E(X, S),S)$ is minimal, $E(X, S)$ is the only minimal ideal in $E(X, S)$ with only idempotent $e$. Hence $E(X,S)$ is a group.

$(3)\Rightarrow (1)$ Let $(x,y)\in P(X)$, then there is a minimal ideal $I\subset E(X,S)$ such that  $rx=ry$ for all $r\in I$.  Now $E(X)$ is a group, and so $ x=y$. Hence $(X,S)$ is distal.
\end{proof}

On similar lines, we can also say:

\begin{proposition} Let $(X,S)$ be a transitive, distal semiflow then  $(E(X,S), S)$ is distal.
\end{proposition}
\begin{proof} Suppose $(p,q)\in P(E(X,S))$ then there is a minimal ideal $I \subset E(E(X, S))$ such that  $rp=rq$ for all $r\in I$. Since $(X,S)$ is transitive, $E(E(X)) \cong E(X)$. So $rp=rq$ for all $r\in I\subset E(X)$. Since $E(X)$ is a group by the above generalization, we have $ p=q$. Hence $(E(X,S),S)$ is distal. \end{proof}
\begin{theorem}
Suppose $(X,S)$ is a distal system then $E(X,S)= E(X,[S])$ as subsets of  $X^X$.
\begin{proof}
Since $(X,S)$ is distal. $E(X,S)$ is a group with unique idempotent  $e$ and $E(X,S)$ is unique minimal ideal n $E(X,S)$.

Now $S \subset \overline{S}=E(X,S)$ but $[S]$ is the smallest group containing $S$. So $[S]\subseteq E(X,S)$. Hence $E(X,[S])= \overline{[S]} \subseteq \overline{E(X,S)}= E(X,S)$.
\end{proof}
\end{theorem}

We see the same in the following example:

\begin{example}
Consider the irrational rotation $(S^{1}, T_{\alpha})$. This flow is invertible, minimal and distal. So $E(S^1, T_{\alpha})$ is a group and $E(S^{1}, T_{\alpha})\cong S^{1}$. In case we take forward orbits only, since the irrational rotation is minimal and also an isometry, so the identity $e$ and hence the inverses $T_{\alpha}^{-n}$ will all be limit points of the forward iterates. So  $E(S^{1}, \N) \cong S^{1} \cong E(S^{1}, T_{\alpha}) = E(S^{1}, \Z)$.
\end{example}

\color{black}

\subsection{Continuity of elements in $E(X,S)$}

Each $s \in S$ is continuous on $X$, but that need not be so for elements of $E(X,S)$. We notice here that the results for semiflows vary from the results for the known cases for flows.

%\begin{theorem} \cite{AUSD}
%Let $(X,S)$ be a semiflow with $S$ an Abelian semigroup and the set of idempotents $J(E(X))
%\subset C(X, X)$ then the following two statements hold;
%\begin{enumerate}
%\item There exists a unique minimal left ideal $I$ in $E(X)$ and
%moreover $I$ contains a unique idempotent
%$u$. Hence $P(X)$ is an equivalence relation on $X$.
%\item If $x \in X$ is an almost periodic point, then it is a distal
%point. Hence if there is a dense set of almost periodic points then $(X,S)$ is distal.
%\end{enumerate}
%\end{theorem}

We recall an equivalent definition of a WAP system from \cite{MST}.
For a compact metric space $X$ and a continuous (not necessarily injective or surjective) $g: X\to X$, the semicascade $(X,g)$ is \emph{WAP} if for each $f$ in $\mathcal{C}(X)$ (space of all continuous complex valued functions on $X$)
and each sequence $\{N_{0}\}$ in $\Z^{+}$, there is a subsequence $\{N_1\}$ of $\{N_0\}$ so that $\lbrace f\circ g^{n}: n \in N_{1}\rbrace$ converges pointwise (and boundedly) to a continuous limit.

\vskip .5cm

We have the following for a WAP semicascade $(X,g)$.
\begin{proposition}
Suppose $(X,g)$ is a WAP semicascade then all members of $E(X,g)$ are continuous.
\end{proposition}
\begin{proof}
Let $h\in E(X,g)$. There is a sequence $\lbrace n_{i} \rbrace$
 such that $g^{n_{i}} \longrightarrow h$. Since $(X,g)$ is WAP system, there is a subsequence $\lbrace n_{k_{i}}\rbrace$
of $\lbrace n_{i}\rbrace$ such that  $g^{n_{k_{i}}}$ converges pointwise to a continuous limit. Since $g^{n_{i}} \longrightarrow h$, so $g^{n_{k_{i}}} \longrightarrow h$ and hence $h$ is continuous.
\end{proof}

In case of equicontinuous cascade, all members of the enveloping
semigroup are homeomorphisms which is not so in case of
equicontinuous semicascade(since they need not be distal). Consider the below example;
\begin{example}
Let $X=[0,1]$ and $f:X\rightarrow X$ defined as $f(x)=\frac{x}{2}$.
Then $(X,f)$ is equicontinuous semicascade.
We can see that for every sequence $\lbrace n_{i} \rbrace \in \N$ and every $x \in X$, the only limit point of $\lbrace f^{n_{i}}(x)\rbrace$
is $0$. So $E(X)= \lbrace f^{n}:n\in \N\rbrace \cup \lbrace h\rbrace $, where $h(x)=0$ for all $x\in [0,1]$.
\end{example}

We note that in the example above the map $f$ is injective but not surjective. However surjective maps giving equicontinuous systems will be bijective and we do get homeomorphisms in their enveloping semigroups.

\begin{remark}
Let $(X,S)$ be a surjective, equicontinuous semiflow. Then $(X,S)$ is distal and so  $E(X,S) = E(X,[S])$ is a group of self-homeomorphisms on $X$.
\end{remark}

What happens in general?

\begin{proposition} \label{eec}
If $(X,S)$ is an equicontinuous semiflow, then all members of $E(X)$ are continuous.
\end{proposition}
\begin{proof}
Let $p \in E(X)$  and $x\in X$ with $\epsilon > 0$. There
is a net $\lbrace s_{i}\rbrace$ in $S$ such that
$s_{i}\longrightarrow p$. So for any $y \in X$, there is an directed set $N_{y}$ such that $d(s_{i}(y),p(y))< \frac{\epsilon}{3}$ for all $i$ preceding $N_{y}$ in this directed set. Also since $(X,S)$ is equiconinuous, then there is an open neighbourhood $U \ni x$ such that for all $y\in U$ and
for all $s\in S$, $d(sx,sy)<\frac{\epsilon}{3}$.

Now let $s_{j}$ be a subnet given by $j$ preceding both $ N_{x},N_{y}$ in this directed set and some $y\in U$. Then
\begin{center}
$d(p(x),p(y))\leq d(p(x),s_{j}x)+d(s_{j}x,s_{j}y)+
d(s_{j}y,p(y))<
\frac{\epsilon}{3}+\frac{\epsilon}{3}+\frac{\epsilon}{3}=\epsilon$.
\end{center}

Thus for every $y \in U$, $d(p(x),p(y)) < \epsilon$.

Hence $p$ is continuous. \end{proof}

\vskip .5cm

\subsection{Transitivity of $E(X,S)$}

We recall the concepts of \emph{weak rigidity, rigidity and uniform rigidity}  first defined by Glasner and Maon \cite{RIG} for  cascades. We note that the same definitions also hold for semicascades. Now the semicascade $(X,g)$ gives the enveloping semicascade $(E(X,g), g)$. We recall Theorem \ref{2.1} here and note that the theorem  below can be proved on the same lines.

\begin{theorem}  For the semicascade $(X,g)$ the following are equivalent:

\begin{enumerate}
  \item The semicascade $(E(X), g)$  is topologically transitive.
  \item $(X,g)$ is weakly rigid.
  \item The identity $e = g^0$ is not isolated in $E(X)$.
\end{enumerate}

\end{theorem}

\vskip .5cm

\begin{example} Consider the one-sided $\emph{2}-$shift $(\emph{2}^{\N}, \sigma)$. The enveloping semigroup here will be $E(X) = \beta \N$. Now take $k, l \in \N$ with $k < l$, then for points $x =  1^{k} 0^\infty, y =  1^\infty \in X$  observe that $x \in [1^k], y \in [1^l]$ but there is no $n \in \N$ for which $\sigma^n([x;[1^k]]) \cap [y; [1^{l}]] \neq \emptyset$. Thus $(E(\emph{2}^{\N}), \sigma)$ is not  transitive or $(\beta \N, \N)$ is not transitive.
\end{example}

\vskip .5cm

We recall the properties of weakly mixing enveloping semigroups  studied in \cite{AEB}, and note that we can prove the following in case of semicascades too:

\begin{theorem} For a semicascade $(X, g)$ the following are equivalent:

(i) The system $(X, g)$ has a weakly mixing enveloping semigroup.

(ii) For any $k$-tuple $(x_1, x_2, \ldots , x_k) \in X^k$ the  system  $(\overline{O(x_1, x_2, \ldots , x_k)}, g^{(k)})$ is weakly mixing.
\end{theorem}

\vskip .5cm

We also have:

\begin{theorem} The semicascade $(E(X),g)$ is mixing $\Rightarrow$ for the semicascade $(X,g)$ every $x \in X$ is essentially non-wandering. \end{theorem}

\vskip .5cm

But again, since we believe that there need not be any mixing enveloping cascades, the same would be true for enveloping semicascades.

\subsection{On $E(2^X, S)$}

It is interesting to see  for a semiflow $(X,S)$,  the relation between $E(X,S)$ and $E(2^X,S)$.
As we have seen that for an equicontinuous flow $(X,T)$, $E(2^X,T)$ is conjugate to $E(X,T)$.
This is also true for a semiflow.

\begin{theorem}
If $(X,S)$ is equicontinuous then $E(X,S)$ is conjugate to $E(2^X,S)$.
\end{theorem}
\begin{proof}
Since $(X,S)$ is equicontinuous then $(2^X,S)$ will also be equicontinuous as it is true in flows. So by Proposition \ref{eec}, all the members of $E(X)$ and $E(2^X)$ are continuous and the topology of pointwise convergence coincide with the uniform convergence.

Now the same map $\Theta: E(2^X)\rightarrow E(X)$ defined in Theorem \ref{T10.2} will be a conjugacy and the proof will be same as in Theorem \ref{TH 10.3}.
\end{proof}

\section{Concluding Remarks}

One direction of study that we have not got into here is the `cave' of \emph{tame systems}. An elaborate study of such systems has been done in    \cite{ TM, STM, GM-MTS}. \emph{Tame systems} are those systems whose enveloping semigroups are seperable and Frechet. It is known that a minimal distal metric system is tame if and only if it is equicontinuous. And a  minimal tame cascade has zero topological entropy. It would be interesting to look into these properties for a non minimal case, but we feel that it would be too far fetching for us to get into such systems in this article.

\vskip .5cm

We are indebted to  Ethan Akin and Joseph Auslander for all the discussions throughout writing this article and their many helpful comments. We also thank Dona Strauss for helping us with computing the enveloping semigroup of the Even Shift, and related discussions.

\vspace{12pt}
\bibliography{xbib}

\begin{thebibliography}{99}

\bibitem {AK}
{\bf Ethan Akin,} Recurrence in topological dynamics: Furstenberg families and Ellis actions, {\it The University Series in Mathematics. Plenum Press, New York,} (1997).

\bibitem {AAB1}
{\bf Ethan Akin, Joseph Auslander and Kenneth Berg,} When is a transitive map chaotic, {\it Convergence in Ergodic Theory and Probability, Walter de Gruyter and Co,} (1996) 25-40.

\bibitem {AAB2}
{\bf Ethan Akin, Joseph Auslander and Kenneth Berg,} Almost equicontinuity and the enveloping semigroup, {\it Topological Dynamics and Applications, Contemporary Mathematics,} 215, (1998) 75-81.

\bibitem {AEAP}
{\bf Ethan Akin, Joseph Auslander}
Almost periodic sets and subactions in topological dynamics, {\it
Proc. Amer. Math. Soc.} 131 (2003), no. 10, 3059–3062.

\bibitem {AAG}
{\bf Ethan Akin, Joseph Auslander and Eli Glasner,} The topological dynamics of Ellis actions, {\it Mem. Amer. Math. Soc.} 195 (2008).


\bibitem {AN}
{\bf Ethan Akin, Joseph Auslander and Anima Nagar,} Dynamics of Induced Systems, {\it Ergod. Th. and Dynam. Sys.,} 37 (2017) 2034-2059.

\bibitem {AANT}
{ \bf Ethan Akin, Joseph Auslander and Anima Nagar,} Variations on the Concept of Topological Transitivity, {\it Studia Math.}  235 (3) (2016),  225–249.


\bibitem {AEB}
{\bf Ethan Akin, Eli Glasner and Benjamin Weiss,} On weak rigidity and weakly mixing enveloping semigroups, {\it arXiv:1711.03169,} (2017).

\bibitem {AUS}
{\bf Joseph Auslander,} Minimal flows and their extensions, {\it North-Holland Mathematics studies,} 153 (1988).

\bibitem {AUSD}
{\bf Joseph Auslander and Xiongping Dai,} Minimality, distality and equicontinuity for semigroup actions on
compact Hausdorff spaces, {\it Discrete \& Continuous Dynamical Systems - A},  39(2019) 4647-4711.

\bibitem {AG}
{\bf Joseph Auslander and Eli Glasner,} Distal and highly proximal extensions of minimal flows, {\it Indiana Univ. Math. J.,} 26 (1977) 731-749.


\bibitem {AGN}
{\bf Joseph Auslander, Gernot Greschonig and Anima Nagar,} Reflections on Equicontinuity, {\it Proc. Amer. Math. Soc.,} 142 (2014), 3129-3137.


\bibitem {AY}
{\bf Joseph Auslander and James A. Yorke,} Interval maps, factors of maps, and chaos, {\it Tohoku Mathematical Journal, Second Series,} 32.2, (1980) 177-188.

\bibitem{B}
{\bf Mike Boyle, } Lower entropy factors of sofic systems, {\it Ergodic Theory and Dynamical Systems,} 3, (1983) 541-557.

\bibitem{BGH}
{\bf Kenneth Berg, David Gove and Kamel Haddad}, Enveloping semigroups and mappings onto the two-shift, {\it Proc.
Amer. Math. Soc.,} 126, (1998) 899-905.

\bibitem {CHA}
{\bf Janusz J. Charatonik and Wlodzimierz J. Charatonik,} Inducible mappings between hyperspaces, {\it Bulletin of the Polish academy of sciences Mathematics,} 46.1, (1998), 5-10.

\bibitem {D}
{\bf Tomasz Downarowicz,} Weakly almost periodic flows and hidden eigenvalues, { \it Topological dynamics
and applications, Contemporary Mathematics 215, a volume in honor of R. Ellis,} (1998) 101-120.

\bibitem {DUG}
{\bf James Dugundji,} Topology, {\it Allyn And Bacon, Inc.,} (1966).

\bibitem{REL}
{\bf Robert Ellis,} Distal Transformation Groups, {\it Pacific Journal Of Mathematics,} 8 (1958), 401-405.

\bibitem{EL}
{\bf Robert Ellis,} A semigroup associated with a transformation group, {\it Transactions of the American Mathematical Society,} 94 (1960), 272-281.

\bibitem{ELLIS}
{\bf Robert Ellis,} Lectures on topological dynamics, {\it W. A. Benjamin, Inc., New York,}  (1969).

\bibitem {ELL}
{\bf David B. Ellis and Robert Ellis,} Automorphisms and Equivalence Relations in Topological Dynamics, {\it Cambridge University Press,} (2014).

\bibitem {EGS1}
{\bf Robert Ellis, Shamuel Glasner and Leonard Shapiro,} Proximal Isometric flows, {\it Advances in Mathematics,} 17 (1975), 213-260.



\bibitem{ELN}
{\bf Robert Ellis and Mahesh Nerurkar,} Weakly almost periodic flows, {\it Transactions of the American Mathematical Society,} 313.1, (1989), 103-119.


\bibitem{F}
{\bf  Hillel Furstenberg,} Disjointness in ergodic theory, minimal sets and a problem in diophantine
approximation, {\it Mathematical Systems theory,} 1 (1967), 1-49.

\bibitem{SGT}
{\bf Shmuel Glasner,} Topological dynamics and group theory, {\it Transactions of the American Mathematical Society} 187 (1974), 327-334.

\bibitem{SG}
{\bf Shmuel Glasner,} Compressibility Properties In Topological Dynamics, {\it American Journal of Mathematics,}  97(1975), 148-171.

\bibitem{SGM}
{\bf Shmuel Glasner,} A metric minimal flow whose enveloping semigroup contains finitely many minimal ideals is PI, {\it Israel Journal of Mathematics,} 22(1975), 87-92.

\bibitem{SGP}
{\bf Shmuel Glasner,} Proximal flows {\it Springer Berlin Heidelberg}, (1976) 17-29.

\bibitem {ELI}
{\bf Eli Glasner,} Quasifactors of minimal systems. {\it Topol. Methods Nonlinear Anal.,} 16 (2000), 351-370.

\bibitem {JOIN}
{\bf Eli Glasner,} Ergodic theory via joinings, {\it AMS, Surveys and Monographs,} 101 (2003).

\bibitem {TM}
{\bf Eli Glasner,} On tame dynamical systems, {\it Colloquium Mathematicum,} 105, (2006), 283-295.

\bibitem {ES}
{\bf Eli Glasner,} Enveloping semigroups in topological dynamics, {\it Topology and its Applications} 154 (2007), 2344-2363.

\bibitem {STM}
{\bf Eli Glasner,} The structure of tame minimal dynamical systems, {\it Ergod. Th. and Dynam. Sys.} 27 (2007), 1819-1837.

\bibitem {GG}
{\bf Eli Glasner and Yair Glasner,} A minimal PI cascade with $2^c$ minimal ideals, {\it arXiv:1801.03377,} (2018).

\bibitem {RIG}
{\bf Eli Glasner and David Maon,} Rigidity in topological dynamics, {\it Ergod. Th. Dynam. Sys. 9,} (1989), 309-320.

\bibitem {GMES}
{\bf Eli Glasner, Michael Megrelishvili,} Hereditarily non-sensitive dynamical systems and linear representations, {\it Colloq. Math.,} 104, (2006), 223-283.

\bibitem {GME}
{\bf Eli Glasner, Michael Megrelishvili,} On fixed point theorems and nonsensitivity, {\it Israel J. Math,} 190(2012), 289-305.

\bibitem {GM-MTS}
{\bf Eli Glasner, Michael Megrelishvili,} More on Tame Dynamical Systems, {\it Ergodic Theory and Dynamical Systems in their Interactions with Arithmetics and Combinatorics, Lecture Notes in Mathematics, Springer-Verlag} (2018).

\bibitem {ME}
{\bf Eli Glasner, Michael Megrelishvili and Vladimir V. Uspenskij,} On metrizable enveloping semigroups, {\it Israel journal of Mathematics, March,} 164(2008), 317-332.

\bibitem {ELBW1}
{\bf Eli Glasner and Benjamin Weiss,} Sensitive dependence on initial conditions, {\it Nonlinearity,} 6 (1993), 1067-1075.

\bibitem {ELBW2}
{\bf Eli Glasner and Benjamin Weiss,} Quasifactors of zero-entropy systems, {\it J. Amer. Math. Soc.,} 8(3) (1995), 665-686.

\bibitem {ELBW3}
{\bf Eli Glasner and Benjamin Weiss,} Locally equicontinuous dynamical systems, {\it In Colloq. Math,} 84, 85 (2000), 345-361.


\bibitem {HIN}
{\bf Neil Hindman  and Dona Strauss,} Algebra in the Stone-Cech compactification: theory and applications. Vol. 27, {\it Walter de Gruyter,} (1998).

\bibitem {IN}
{\bf Alejandro Illanes and Sam B. Nadler,} Hyperspaces: fundamentals and recent advances  {\it Monographs
and Textbooks in Pure and Applied Mathematics, 216. Marcel Dekker, Inc., New
York,} (1999).

\bibitem{LPYZ}

{ \bf Jie Li, Piotr Oprocha, Xiangsong Ye and Ruifeng Zhang, } When are all closed subsets recurrent? {\it Ergodic Theory Dynam. Systems} 37 (2017),  2223–2254.


\bibitem {LIN}
{\bf Douglas Lind and Brian Marcus,} An Introduction to Symbolic Dynamics and Coding,{\it Cambridge University Press,} (1995).


\bibitem {M}
{\bf Douglas C. McMahon And Louis J. Nachman,} An Intrinsic Characterization For PI Flows, {\it Pacific Journal Of Mathematics,} 89 (1980) 391-403.

\bibitem{EM}
{\bf Ernest Michael,} Topologies on spaces of subsets {\it Transactions of the American Mathematical Society,} 71.1 (1951), 152-182.

\bibitem{MST}
{\bf James Montgomery, Robert Sine, and Edward Thomas,} Some topological properties of weakly almost periodic mappings, {\it Topology Appl.,} 11 (1980),  69–85.

\bibitem{ANVK}
{\bf Anima Nagar  and V. Kannan} Topological Transitivity for Discrete
Dynamical Systems, {\it Applicable Mathematics In The Golden Age,}
 Narosa Publications(2003), 534-584.

\bibitem {NAMA}
{\bf Katsumi Nakamura,} On bicompact semigroups, {\it Mat. J. Okayama University 1,} (1952), 99-108.

\bibitem{PET}
{\bf Karl Petersen,} Disjointness and weak mixing of minimal sets, {\it Proc. Amer. Math. Soc.,} 24 (1970) 278-280.

\bibitem {PUN}
{\bf Puneet Sharma and Anima Nagar,} Topological dynamics on hyperspaces, {\it Appl. Gen. Topol.,} 11(1) (2010), 1-19.


\bibitem{Ve1}
{\bf William .A. Veech,} {Almost automorphic functions on groups}, {\it Amer. J. Math.} {87} (1965), 719-751.

\bibitem {VER1}
{\bf Jan de Vries,} Elements of Topological Dynamics, {\it Mathematics and its Applications 257, Kluwer, Dordrecht,} (1993).

\bibitem {VER2}
{\bf Jan de Vries,} Topological Dynamical Systems: An Introduction to the Dynamics of Continuous Mappings, {\it De Gruyter study in mathematics,} 59 (2014).

\bibitem {WAL}
{\bf Peter Walters,} An Introduction to Ergodic Theory, {\it Springer, New York,} (1982).

\bibitem {WIL}
{\bf Stephen Willard,} General topology, {\it Addison-Wesley Publishing Co.,} (1970).

\end{thebibliography}

\end{document}